\documentclass[12pt]{amsart}

\usepackage{amsmath,amsthm,amsfonts,amscd,amssymb,amsopn,enumerate}
\usepackage{eucal}
\usepackage{mathrsfs,color,hyperref}
\usepackage{wasysym}
\usepackage[all]{xy}
\usepackage{fullpage}

\newcommand{\mbb}[1]{\mathbb #1}

\newcommand{\mc}[1]{\mathcal #1}
\newcommand{\ms}[1]{\mathscr #1}

\newcommand{\wh}{\widehat}
\newcommand{\Spec}{\operatorname{Spec}}
\newcommand{\GL}{\operatorname{GL}}

\newcommand{\Hom}{\operatorname{Hom}}
\newcommand{\Gal}{\operatorname{Gal}}

\newcommand{\spcl}{\operatorname{sp}}

\newcommand{\cd}{\operatorname{cd}}

\newcommand{\tors}{\operatorname{\mc T}}
\newcommand{\ttors}{\operatorname{\mho}}
\newcommand{\red}{\operatorname{red}}
\newcommand{\ov}[1]{\overline{#1}}
\newcommand{\til}[1]{\widetilde{#1}}

\newcommand{\basering}{R}

\newcommand{\customdiagram}[4]{
\fbox{\xymatrix @C=.33cm @R=-0.3cm {
 & {#2} \ar@/^0.4pc/[dr] \\
{#1} \ar@/^0.4pc/[ur] \ar@/_0.4pc/[dr] & & {#4} \\
 & {#3} \ar@/_0.4pc/[ur]
}}
}

\newcommand{\vardiagram}[2]{
\fbox{\xymatrix @C=.33cm @R=-0.3cm {
 & {#2}_1 \ar@/^0.4pc/[dr] \\
{#2}_0 \ar@/^0.4pc/[ur] \ar@/_0.4pc/[dr] & & {#1} \\
 & {#2}_2 \ar@/_0.4pc/[ur]
}}
}

\newcommand{\diagram}[1]{\vardiagram{#1}{#1}}

\newcommand{\mvardiagram}[2]{
\mbox{\xymatrix @C=.33cm @R=-0.3cm {
 & {#2}_1 \ar@/^0.4pc/[dr] \\
{#2}_0 \ar@/^0.4pc/[ur] \ar@/_0.4pc/[dr] & & {#1} \\
 & {#2}_2 \ar@/_0.4pc/[ur]
}}
}

\newcommand{\mdiagram}[1]{\mvardiagram{#1}{#1}}

\newcommand{\varPdiagram}[2]{
\fbox{\xymatrix @C=.33cm @R=-0.3cm {
 & {#2}_\UU \ar@/^0.4pc/[dr] \\
{#2}_\BB \ar@/^0.4pc/[ur] \ar@/_0.4pc/[dr] & & {#1} \\
 & {#2}_\PP \ar@/_0.4pc/[ur]
}}
}

\newcommand{\Pdiagram}[1]{\varPdiagram{#1}{#1}}

\newcommand{\antidiagram}[1]{
\fbox{\xymatrix @C=.33cm @R=-0.3cm {
 & {#1}_1 \ar@/_0.4pc/[dl] \\
{#1}_0 & & {#1} \ar@/_0.4pc/[ul] \ar@/^0.4pc/[dl] \\
 & {#1}_2 \ar@/^0.4pc/[ul]
}}
}

\theoremstyle{plain}
\newtheorem{thm}{Theorem}[section]

\newtheorem{lem}[thm]{Lemma}
\newtheorem{lemma}[thm]{Lemma}
\newtheorem{cor}[thm]{Corollary}
\newtheorem{prop}[thm]{Proposition}

\newtheorem*{thm*}{Theorem}
\newtheorem*{rem*}{Remark}
\newtheorem*{lem*}{Lemma}

\newtheorem*{cor*}{Corollary}
\newtheorem*{prop*}{Proposition}

\theoremstyle{definition}
\newtheorem{defn}[thm]{Definition}
\newtheorem{ex}[thm]{Example}
\newtheorem{hyp}[thm]{Hypothesis}

\theoremstyle{remark}
\newtheorem{rem}[thm]{Remark}

\newcommand{\SL}{\operatorname{SL}}
\newcommand{\SK}{\operatorname{SK}}
\newcommand{\sheaf}[1]{\mathscr{#1}}
\newcommand{\XX}{\sheaf{X}}
\newcommand{\PP}{\sheaf{P}}
\newcommand{\UU}{\sheaf{U}}
\newcommand{\BB}{\sheaf{B}}
\newcommand{\OO}{\sheaf{O}}
\newcommand{\ZZ}{\sheaf{Z}}

\newcommand{\R}{{\mathrm{R}}}
\renewcommand{\P}{\mathbb P}
\newcommand{\QQ}{\mathbb{Q}}

\newcommand{\A}{\mathbb{A}}
\newcommand{\G}{\mathbb{G}}
\newcommand{\Z}{\mathbb{Z}}
\newcommand{\Gm}{\mathbb{G}_{\mathrm m}}

\newcommand{\N}{\operatorname{N}\!}
\newcommand{\Nrd}{\operatorname{Nrd}}
\newcommand{\Ext}{\operatorname{Ext}}
\newcommand{\cha}{\mathrm{char}}
\newcommand{\sep}{\mathrm{sep}}
\newcommand{\iso}{\ {\buildrel \sim \over \longrightarrow} \ }

\usepackage[OT2,T1]{fontenc}

\DeclareSymbolFont{cyrletters}{OT2}{wncyr}{m}{n}
\DeclareMathSymbol{\Sha}{\mathalpha}{cyrletters}{"58}

\title{Local-global principles for constant reductive groups over semi-global fields}
\date{January 2, 2022}
\author{Jean-Louis Colliot-Th\'el\`ene, David Harbater, Julia Hartmann, \\Daniel Krashen, R. Parimala, and V. Suresh}
\dedicatory{To Gopal Prasad on his 75th birthday} 
\begin{document}

\maketitle

\begin{abstract}
We study local-global principles for torsors under reductive linear algebraic groups over semi-global fields; i.e., over one variable function fields over complete discretely valued fields.  We provide conditions on the group and the semiglobal field under which the local-global principle holds, and we compute the obstruction to the local-global principle in certain classes of examples.  Using our description of the obstruction, we give the first example of a semisimple simply connected group over a semi-global field where the local-global principle fails.  Our methods include patching and R-equivalence.
\end{abstract}

\section*{Introduction}

Local-global principles are classically studied over global fields, and in particular number fields.  In addition to global function fields (i.e., function fields over finite fields), function fields over local fields are also of interest in arithmetic geometry. More generally, one may consider one-variable function fields $F$ over complete discretely valued fields; such fields
are also known as {\em semi-global fields}. Local-global principles for these function fields have been considered in a number of recent works.

In this paper, we study local-global principles for the triviality of torsors under connected  linear algebraic groups over semi-global fields. As in the case of number fields, there is a local-global principle that holds for torsors under a connected linear algebraic group that is rational as an $F$-variety (\cite{HHK09}, \cite{HHK15}), but this is not the case for torsors under general nonrational connected linear algebraic groups (e.g.\ nonrational tori, \cite{CTPS16}).
Nevertheless, it was shown in \cite[Theorem~7.3]{CHHKPS} that local-global principles do hold for torsors over a semi-global field $F$ if the group is a torus over the valuation ring $R$ of $K$, and the closed fiber of a regular projective model $\XX$ of $F$ over $R$ satisfies a certain property related to trees.
The results in that work relied on two key properties of tori: that they are commutative and that they have flasque resolutions.

In the present paper, we use a quite different strategy to treat local-global principles for torsors over $F$ under more general reductive groups $G$ over $R$, where commutativity might not hold and where flasque tori do not control the group in the same way.  
The obstruction to such a local-global principle is a Tate-Shafarevich set $\Sha(F,G)$; viz., the set of isomorphism classes of $G$-torsors over $F$ that become trivial over each completion $F_v$ of $F$ at a divisorial discrete valuation $v$.  The following is a special case of our Theorem~\ref{main}:

\begin{thm}
Let $K$ be a complete discretely valued field, $R$ its ring of integers, and $F$ a semi-global field over $K$. Let $\XX$ be a regular projective model of $F$ over $R$.
Assume that the residue field $k$ of $R$ is of characteristic zero; that the closed fiber $\XX_k$ of $\XX$ is reduced and has normal crossings; and that the reduction graph associated to $\XX_k$ is a tree and remains so under all finite extensions $k'/k$.  Then for any reductive linear algebraic group $G$ over $R$ we have $\Sha(F,G)=1$.
\end{thm}

The central technical tool in proving this theorem is a factorization result (Proposition~\ref{anisotropic factorization}). It implies that under these hypotheses, any torsor over $F$
that is locally trivial with respect to discrete valuations of $F$
can be lifted to a torsor over a model $\XX$ of $F$ that is trivial along the reduced closed fiber of $\XX$ (Theorem~\ref{torsors extend}).  We then invoke
recent work of Gille, Parimala and Suresh \cite{GPS} to conclude that the given torsor over $F$ is in fact trivial, thereby obtaining Theorem~\ref{main}.  In addition we show that under appropriate hypotheses, we may even drop the assumptions that the group is reductive and connected.

In certain situations where the reduction graph does not satisfy the tree hypothesis of the above theorem, we compute an explicit quotient of $\Sha(F,G)$ in Proposition~\ref{dbl coset surj}.  This quotient can be viewed as a ``lower bound'' for $\Sha(F,G)$, and sometimes it turns out to give $\Sha(F,G)$ exactly.  In particular, we have the following result, which computes $\Sha(F,G)$ in terms of the group $G(k)/\R$ of $\R$-equivalence classes of $k$-points of $G$ under appropriate hypotheses (see Theorem~\ref{explicit Sha} for the precise statement):

\begin{thm}
In the above situation (but allowing arbitrary characteristic for $k$), if the closed fiber of $\XX$ is reduced and consists of copies of $\P^1_k$ meeting at $k$-points and forming $m$ cycles, and if $\cha(k)$ is not one of the bad primes for the reductive group $G$ over $R$, then $\Sha(F,G)$ is in bijection with the set of uniform conjugacy classes of $(G(k)/\R)^m$.
\end{thm}

Using this, we answer an open question concerning local-global principles for torsors under semisimple simply connected linear algebraic groups.
Such principles hold in the case of global fields (by work of Eichler, Kneser, Harder, and Chernousov),
and it seemed reasonable to expect that they would hold in the semi-global case.
This was conjectured for semi-global fields over $p$-adic fields in \cite{CTPS12}; and
in that situation the conjecture has been proven in many cases (see \cite{preeti}, \cite{Hu}, \cite{PPS}, \cite{PS20}).
Here we provide the first counterexamples to such a local-global principle over a more general semi-global field (Examples~\ref{triangle ex} and~\ref{nonmono tree}), thus proving the following:

\begin{thm}
For $G$ a semisimple simply connected group over a field $k$, and $F$ a semi-global field over $k((t))$, the local-global principle for $G$-torsors over $F$ does not always hold.
In particular, there are counterexamples in which $k$ has the form $E(x,y)$ or $E((x))((y))$, where $E$ is either a number field or a field of the form $\kappa_0(u,v)$ or $\kappa_0((u))((v))$ for some field~$\kappa_0$.
\end{thm}

In our examples, $\cd(k) \ge 4$, and $\cd(F) \ge 6$.  
For an arbitrary semisimple simply connected group $G$ over a semi-global field $F$,
it is unknown whether there exist counterexamples of smaller cohomological dimension.

Our examples are obtained by using our descriptions of $\Sha(F,G)$ given in Proposition~\ref{dbl coset surj} and Corollary~\ref{explicit Sha} as well as by relying on classical results of Platonov and of Voskresenski\v{\i}.  Building on that, we even obtain examples where $\Sha(F,G)$ is infinite.  We also show that if one considers just the discrete valuations on $F$ that are trivial on $K$, then the corresponding local-global principle 
can fail even when $K$ is $p$-adic (Example~\ref{p-adic counterex}).  For a further discussion of our examples and the general context, see Section~\ref{sect: counterex}.

\medskip

{\bf Structure of the manuscript:}
Following a discussion of background and preliminary results in Section~\ref{overview sec},  we prove local factorization results in Section~\ref{factor sec} for reductive groups in the patching setup.  Section~\ref{local triv sec} builds on those results to prove global factorization statements, from which we obtain Theorem~\ref{torsors extend} mentioned above.  In Section~\ref{lgp mon tree sec}, we use that lifting result to obtain Theorem~\ref{main},
which asserts a local-global principle under the hypothesis that the reduction graph is a ``monotonic tree'' (a notion introduced in \cite{CHHKPS}).  We also obtain extensions of that theorem to groups that need not be reductive or connected.  Afterwards, to treat cases where the monotonic tree hypothesis does not hold, we find an explicit quotient of $\Sha(F,G)$ in Section~\ref{sectionlb} in cases where the components of the closed fiber are projective lines; and we show in Section~\ref{Sha computation section} that this is
an exact computation of $\Sha(F,G)$ when those lines and their intersection points are defined over $k$.  These computations, which are given in terms of $\R$-equivalence classes in $G(k)$,
rely on the double coset presentation of $\Sha(F,G)$ arising from patching.  
The descriptions of $\Sha(F,G)$ given in Sections~\ref{sectionlb} and~\ref{Sha computation section} are then used in Section~\ref{sect: counterex} to obtain our examples of semisimple simply connected groups $G$ for which $\Sha(F,G) \ne 1$.  
We conclude with an Appendix that establishes a ``specialization map'' $G(K)/\R \to G(k)/\R$ on $\R$-equivalence classes that extends work of Gille (see~\cite{GilleTAMS}) and is used especially in Sections~\ref{sectionlb} and~\ref{Sha computation section}.

\medskip

{\bf Acknowledgments:} The authors thank Brian Conrad and Philippe Gille for helpful discussions.

\section{Patching and local-global principles} \label{overview sec}

In this section, we begin by recalling the patching setup for semi-global fields.  Afterwards, we discuss local-global principles over such fields and their models, and we recall the relationships among obstruction sets to such principles.

\subsection{The patching framework} \label{patching subsec}

Let $K$ be a complete discretely valued field, with valuation ring $\basering$, uniformizing parameter $t \in {\basering}$, and residue field $k$.
A \textit{semi-global field over $K$} (or \textit{over $\basering$}) is a one-variable function field $F$ over $K$; i.e., a finitely generated extension of $K$ of transcendence degree one in which $K$ is algebraically closed.  A {\em normal model} of~$F$ is an integral $\basering$-scheme $\XX$ with function field~$F$ that is flat and projective over~$\basering$ of relative dimension one, and that is normal as a scheme. Its closed fiber is denoted by $\XX_k$.  If $\XX$ is a normal model that is regular as a scheme, we say that $\XX$ is a {\em regular model} of $F$. Such a regular model exists by the main theorem in \cite{Lip78} (see
also \cite[Theorem 0BGP]{stacks}).
Moreover, by \cite[page 193]{Lip75}, there exists a regular model $\XX$ for which the reduced closed fiber ${\XX}_k^{\red}$  is a union of regular curves, with normal crossings.
We refer to such a regular model as a (strict) {\em normal crossings model of~$F$}.

Let $\XX$ be a normal model of a semi-global field $F$.  Let $U$ be an affine open subset of ${\XX}_k^{\red}$, and assume that $U$ is irreducible (or equivalently, is contained in an irreducible component of ${\XX}_k^{\red}$).  We consider the ring $R_{U} \subset F$
consisting of the rational functions on~${\XX}$ that are regular at all points of $U$.
The $t$-adic completion $\wh{R}_{U}$ of $R_{U}$ is an $I$-adically complete domain, where
$I$ is the radical  of the ideal generated by $t$ in $\wh{R}_{U}$.
The quotient
$\wh{R}_{U}/I $ equals $k[U]$, the ring of regular functions on the integral affine curve $U$.  We write $F_U$ for the field of fractions of $\wh R_U$.  Also, for
a (not necessarily closed)
 point $P$ of ${\XX}_k^{\red}$, we let~$F_{P}$ denote the field of fractions of the complete local ring $\wh R_P:=\wh{\mc O}_{\XX,P}$ of $\XX$ at $P$.  Note that if $P \in U$ then $F_P$ contains $F_U$.

Finally, for a closed point $P \in {\XX}_k^{\red}$, we consider the height one primes $\wp$ of the complete local ring $\wh R_P$ that contain the uniformizing parameter $t \in R$.
For each such $\wp$, we let $R_\wp$ be the localization of $\wh R_P$ at $\wp$, and we let $\wh R_\wp$ be its $t$-adic (or equivalently, its $\wp$-adic) completion; this is a complete discrete valuation ring.
We write $F_\wp$ for the fraction field of $\wh R_\wp$, and $k_\wp$ for the residue field.
We call such a $\wp$ a {\it branch at} $P$ {\it on} ${\XX}_k^{\red}$ (or {\it on} $U$, if $P \in \ov U$ for $U \subset {\XX}_k^{\red}$ as above if $\wp$ contains the ideal of $\wh R_P$ defining $\ov U$).

We will often choose sets $\PP$ and $\UU$ as follows:

\smallskip

{\narrower\narrower\noindent
We let $\PP$ be a finite nonempty set of closed points of ${\XX}_k^{\red}$
that contains all the
points of ${\XX}_k^{\red}$ at which ${\XX}_k^{\red}$ is not unibranched.
(If $\XX$ is a normal crossings model, these are the points at which ${\XX}_k^{\red}$ is singular; i.e., not regular.) 
We let $\UU$ be the set of irreducible components of the affine curve ${\XX}_k^{\red} \smallsetminus \PP$. \par}

\smallskip 

\noindent In this situation, the
fields of the form $F_P$, $F_U$ for $P\in \PP$, $U\in \UU$ are called {\it patches} for $F$ and the rings $\wh R_P, \wh R_U$ are called {\it patches} on $\XX$.
Let $\BB$ denote the set of all branches at points $P\in \PP$ (each of which lies on some $U\in \UU$). The fields $F_\wp$ (resp., rings $\wh R_\wp$) are referred to as the {\it overlaps} of the corresponding patches $F_P, F_U$ (resp., $\wh R_P, \wh R_U$).  For a branch $\wp$ at $P$ on $U$, 
there is an inclusion $F_{P} \subset F_\wp$ induced by the inclusion $\wh{R}_P\subset \wh{R}_\wp$, and also an inclusion $F_{U} \subset  F_\wp$ that is induced by the inclusion $\wh{R}_{U} \hookrightarrow \wh{R}_\wp$. (See \cite{admis}, beginning of Section~4.)
We also have an associated {\em reduction graph} $\Gamma$, a bipartite graph whose vertices correspond to the elements of $\PP \cup \UU$, and whose edges correspond to the elements of $\BB$; here an edge $\wp \in \BB$ connects vertices $P \in \PP$ and $U \in \UU$ if $\wp$ is a branch at $P$ on $U$.  (See \cite[Section~6]{HHK15} for more details.)

Given sets $\PP, \UU, \BB$ as above, we may consider a {\em refinement} $\PP', \UU', \BB'$, by choosing a finite set of closed points $\PP'$ of ${\XX}_k^{\red}$ that contains $\PP$.  Thus the new points of $\PP'$ (i.e., those that are not in $\PP$) are all regular points of ${\XX}_k^{\red}$.  The set $\BB'$ consists of $\BB$ together with one additional element $\wp'$ for each new element $P'$ of $\PP'$ (viz., the unique branch at the regular point $P'$ of ${\XX}_k^{\red}$).  The set $\UU'$ is in bijection with the set $\UU$; viz., for each $U \in \UU$, the corresponding element of $\UU'$ is obtained by deleting from $U$ the finitely many new points of $\PP'$ that lie on $U$.

\subsection{Local-global principles over semi-global fields} \label{lgp sgp}
In this manuscript we will consider smooth affine group schemes $G$ over a ring $A$, or over a scheme $Z$; we will require these to be of finite type.  In the case that the ring $A$ is a field $F$, we will speak of a {\em linear algebraic group} $G$; and again we will require $G$ to be smooth in this manuscript.

If $G$ is a linear algebraic group over a field $F$, then the isomorphism classes of $G$-torsors over $F$ are in natural bijection with the pointed set $H^1(F,G)$ in Galois cohomology.
Given a family $\{F_\omega\}_{\omega \in \Omega}$ of overfields of $F$, we can consider the local-global principle for $G$-torsors over $F$ with respect to these overfields; this asserts the triviality of a $G$-torsor over $F$ provided that it becomes trivial over each $F_\omega$.  As in the classical case of local-global principles over number fields, there is then the associated obstruction set
\[\Sha_\Omega(F, G) := {\rm ker}\bigl(H^1(F, G) \to \prod_{\omega  \in \Omega} H^1(F_\omega, G)\bigr)\]
to the validity of this local-global principle.
In particular, if $\Omega$ is a set of discrete valuations on $F$, then for the overfields we take the completions $F_v$ of $F$ with respect to $v$, for $v \in \Omega$.

In the case that $F$ is a semi-global field, we may let $\Omega$ be the set of divisorial discrete valuations on $F$; i.e., the discrete valuations associated to prime divisors on regular models of $F$.  In this situation we simply write
\[\Sha(F, G): = {\rm ker}\bigl(H^1(F, G) \to \prod_{v  \in \Omega} H^1(F_v, G)\bigr)\]
for $\Sha_\Omega(F, G)$.
If instead we choose a normal crossings model $\XX$ of $F$ over the valuation ring $R$ of the underlying complete discretely valued field $K$, and if we
consider the set of overfields $F_P$, where $P$ ranges over all the points of the reduced closed fiber $X = \XX_k^{\red}$ of $\XX$, then we obtain the obstruction
\[\Sha_X(F, G) := {\rm ker}\bigl(H^1(F, G) \to \prod_{P \in X} H^1(F_P, G)\bigr).\]
Finally, with $\PP$ and $\UU$ as in Section~\ref{patching subsec}, we may consider the obstruction
\[\Sha_\PP(F, G) := {\rm ker}\bigl(H^1(F, G) \to \prod_{\zeta \in \PP \cup \UU} H^1(F_\zeta, G)\bigr)\]
to the local-global principle with respect to the set of overfields $F_P$ and $F_U$, ranging over $P \in \PP$ and $U \in \UU$.  (Since the set $\PP$ determines $\UU$, the notation mentions just $\PP$.)

Recall that a connected linear algebraic group $G$ over an algebraically closed field $k$ is {\em reductive} if its unipotent radical (the maximal smooth connected unipotent normal subgroup) is trivial.  More generally, by a {\em reductive group} over a ring $A$ or a scheme $Z$ we will mean a connected smooth affine group scheme $G$ over $A$ (resp., $Z$) such that $G$ is reductive in the previous sense over every geometric point.

For convenience, we recall the following:

\begin{prop} \label{sha containments}
Let $F$ be a semi-global field, let $\XX$ be a normal model of $F$, and let $X = {\XX}_k^{\red}$ be its reduced closed fiber.  Let $G$ be a linear algebraic group over $F$.
\begin{enumerate}
\renewcommand{\theenumi}{\alph{enumi}}
\renewcommand{\labelenumi}{(\alph{enumi})}
\item \label{ShaP ShaX}
If $\PP\subset X$ is as above then $\Sha_\PP(F, G) \subseteq \Sha_X(F, G)$.
\item \label{ShaP containments}
If $\PP\subset X$ is as above and $\PP' \subset X$ is
a finite set of closed points that contains $\PP$, then
$\Sha_\PP(F, G) \subseteq \Sha_{\PP'}(F,G)$.
\item \label{ShaX union}
$\bigcup_{\PP}  \Sha_\PP(F, G) =  \Sha_X(F, G)$, where the union is taken over all subsets $\PP$ as above.
\item \label{ShaX subset Sha}
If $\XX$ is a regular model of $F$ then $\Sha_X(F, G) \subseteq \Sha(F, G)$.
\item \label{ShaX equals Sha}
In part~(\ref{ShaX subset Sha}), the containment is an equality if $G$ is reductive over the model $\XX$\!\!, or if $G$ is a finite linear algebraic group over $F$.
\end{enumerate}
\end{prop}

\begin{proof}
Parts~(\ref{ShaP ShaX}) and~(\ref{ShaP containments}) are shown at the beginning of  \cite[Section~5]{HHK15}; part~(\ref{ShaX union}) is shown in \cite[Corollary~5.9]{HHK15}; and part~(\ref{ShaX subset Sha}) is shown in \cite[Proposition~8.2]{HHK15}.  The first 
assertion in~(\ref{ShaX equals Sha}) is shown in \cite[Theorem~8.10(ii)]{HHK15}.  Concerning the  second assertion in (\ref{ShaX equals Sha}), the proof of \cite[Lemma~8.6]{HHK15} shows that 
$\Sha(F_P,G)$ is trivial for every closed point $P$ on $\XX$.  (The statement of that lemma assumed that $G$ is a finite constant group, but that additional hypothesis was not used in the proof.)
The asserted equality then follows from \cite[Proposition~8.4]{HHK15}.
\end{proof}

See also Theorem~\ref{equalshas} below, where a stronger conclusion is shown under additional hypotheses. 

It will be convenient to consider all patches of a given type and all overlaps together. For this purpose, we define $F$-algebras:
\begin{equation} \label{patch fields}
	F_\UU = \prod_{U \in \UU} F_U, \ \
	F_\PP = \prod_{P \in \PP} F_P, \ \
	F_\BB = \prod_{\wp \in \BB} F_\wp,
\end{equation}
and rings
\begin{equation} \label{patch rings}
	\wh R_\UU = \prod_{U \in \UU} \wh R_U, \ \
	\wh R_\PP = \prod_{P \in \PP} \wh R_P, \ \
	\wh R_\BB = \prod_{\wp \in \BB} \wh R_\wp.
\end{equation}
We note that we have diagonal inclusions $F \subset F_\UU,\! F_\PP \subset F_\BB$ and $\wh R_\UU,\! \wh R_\PP \subset \wh R_\BB$, 
and we can write
\begin{equation} \label{Sha_P short}
\Sha_\PP(F,G) = \ker\bigl(H^1(F, G) \to H^1(F_\UU, G) \times H^1(F_\PP, G)\bigr).
\end{equation}
By Corollary 3.6 of \cite{HHK15}, this has a double coset description:
\begin{equation} \label{Sha_P dbl coset}
\Sha_\PP(F,G) \simeq G(F_\UU)\backslash G(F_\BB)/G(F_\PP).
\end{equation}
Namely, this isomorphism of pointed sets sends each $G$-torsor $\xi \in \Sha_\PP(F,G)$ to the double coset corresponding to the induced patching problem (see the proof of \cite[Theorem~2.4]{HHK15}).  
To see this, choose trivializations $\xi_{F_U} \simeq G_{F_U}$ and $\xi_{F_P} \simeq G_{F_P}$ for each $U \in \UU$ and $P \in \PP$, where $\xi_{F_U}, \xi_{F_P}$ are the base changes of $\xi$ to the fields $F_U, F_P$.  Then
for each branch $\wp \in \BB$ on $U \in \UU$ at $P \in \PP$, we obtain an induced isomorphism $G_{F_U} \otimes_{F_U} F_\wp \to G_{F_P} \otimes_{F_P} F_\wp$ of trivial torsors, given by left multiplication by some element $g_\wp \in G(F_\wp)$.  The double coset associated to $\xi$ is then the one represented by $(g_\wp)_{\wp \in \BB}$.

We can reinterpret Proposition~\ref{sha containments} in these terms.
If $\PP',\UU',\BB'$ is a refinement of $\PP,\UU,\BB$, then the containment
$\Sha_\PP(F,G) \subseteq \Sha_{\PP'}(F,G)$ corresponds to the inclusion of the corresponding double coset spaces.  This inclusion takes the class of $(g_\wp)_{\wp \in \BB} \in G(F_\BB)$ to the class of $(g_\wp')_{\wp \in \BB'} \in G(F_{\BB'})$, where $g_\wp' = g_\wp$ if $\wp \in \BB$ and where $g_\wp'=1$ otherwise, by the above explicit description of the identification in (\ref{Sha_P dbl coset}).  As $\PP$ varies, these double coset spaces form a direct system.  If $\XX$ is a regular model, then the direct limit is identified with $\Sha_X(F,G)$, and also with $\Sha(F,G)$ if in addition $G$ is reductive over $\XX$.  See also the comment after Theorem~\ref{equalshas} below.

In the sequel, it will be useful to consider the spectra of the fields and rings above.  The inclusions $F \subset F_\UU,\! F_\PP \subset F_\BB$ and $\wh R_\UU,\! \wh R_\PP \subset \wh R_\BB$
yield commutative diagrams
\begin{equation} \label{patch diagrams}
\Spec F_\bullet = \Pdiagram{\Spec F\!{}}\ , \ \ \ \XX_\bullet = \varPdiagram{\XX}{\Spec \wh R}\,.
\end{equation}

Passing to the reduced closed fiber, we also obtain the commutative diagram
\begin{equation} \label{closed fiber diagram}
	(\XX_k^{\red})_\bullet = \customdiagram{\Spec \wh R_\BB/\sqrt{ (t)}}{\Spec \wh R_\UU/\sqrt{(t)}}{\Spec \wh R_\PP/\sqrt{(t)}}{\XX_k^{\red}}\,,
\end{equation}
where we write $\sqrt{(t)}$ for the radical of the ideal $(t)$ in the respective ring.
The morphisms $\Spec\,F \to \XX$ and $\XX_k^{\red} \to \XX$ are each compatible  with the corresponding morphisms between the schemes in (\ref{patch diagrams}) and (\ref{closed fiber diagram})
associated to $\BB, \UU, \PP$; and so we also have 
morphisms of diagrams:
\begin{equation} \label{diagram diagram}
\xymatrix{
& \XX_\bullet\\
\Spec F_\bullet \ar[ur]& &
(\XX_k^{\red})_\bullet \ar[ul]\\
}
\end{equation}
We will use this viewpoint in the next subsection.

\subsection{Local-global principles over models}\label{cosets}

We would like to extend definition~(\ref{Sha_P short}) and equation~(\ref{Sha_P dbl coset})
from the situation of groups over a semi-global field $F$ to groups over a model $\XX$ of $F$.  The latter will be done in Corollaries~\ref{model double coset} and~\ref{closed fiber double coset}.  In order to extend~(\ref{Sha_P short}), we will first consider a more abstract framework for local-global principles.

For a scheme $Y$, a group scheme $G$ over $Y$, and a $Y$-scheme $Z$, let
$\tors_G(Z)$ be the category of (right) $G$-torsors over $Z$. Note that for
$Z \to Z'$ a morphism of $Y$-schemes, there is a functor $\tors_G(Z') \to
\tors_G(Z)$ defined by base change (well defined up to natural
isomorphism). We will abuse notation in the case that $Z = \Spec A$, and
write $\tors_G(A)$ for $\tors_G(Z) = \tors_G(\Spec A)$. Note that $\tors_G(\coprod Z_i)$ is naturally equivalent to the product of categories $\prod \tors_G(Z_i)$, and we consider this as an identification.

Recall that if we are given categories $C_0, C_1, C_2$, and functors
$\iota_1: C_1 \to C_0$, and $\iota_2: C_2 \to C_0$, then we define the
2-fiber product of categories $C_1 \times_{C_0} C_2$, as follows: The
objects of $C_1 \times_{C_0} C_2$ are triples $(x_1, x_2, \phi)$, where
$x_i$ is an object in $C_i$, and $\phi: \iota_1 x_1 \to \iota_2 x_2$ is an
isomorphism. Morphisms $(x_1, x_2, \phi) \to (x_1', x_2', \phi')$ are
defined to be pairs of morphisms $(f_1, f_2)$ with $f_i : x_i \to x_i'$
such that we have a commutative diagram:
\[\xymatrix{
\iota_1 x_1 \ar[r]^{\iota_1 f_1} \ar[d]_{\phi} & \iota_1 x_1'
\ar[d]^{\phi'} \\
\iota_2 x_2 \ar[r]_{\iota_2 f_2} & \iota_2 x_2'.
}\]

If we have a commutative diagram of schemes
$$\mdiagram{Y}$$
and a group scheme $G$ over $Y$, base change induces a natural functor
\begin{align*}
	\tors_G(Y) &\to \tors_G(Y_1) \times_{\tors_G(Y_0)} \tors_G(Y_2) \\
	P &\mapsto (P_{Y_1}, P_{Y_2}, 1),
\end{align*}
where $1$ stands for the natural
identification $(P_{Y_1})_{Y_0} = (P_{Y_2})_{Y_0}$, coming from
commutativity of the diagram of schemes. We will refer to this functor as
the one {\it induced by base change}.

\begin{defn}\label{def_sha}
Given a commutative diagram of schemes
\[Y_\bullet = \diagram{Y}\]
and a smooth affine group scheme $G$ over $Y$, we let 
\[\Sha(Y_\bullet, G) =
\ker\bigl(H^1(Y, G) \to H^1(Y_1, G) \times H^1(Y_2, G)\bigr).\]
\end{defn}

Here $H^1(Z,G)$ denotes the first \'etale cohomology set of $G$ over a scheme $Z$; this classifies isomorphism classes of $G$-torsors over $Z$.

As a special case of the definition, if we take $Y_\bullet = \Spec F_\bullet$, where
$F$ and $\Spec F_\bullet$ are as in Equation~(\ref{patch diagrams}), then
\[\Sha(\Spec F_\bullet, G) = \Sha_{\PP}(F,G) = G(F_\UU)\backslash G(F_\BB)/G(F_\PP).\]

It will be useful to give such a double coset description of $\Sha(Y_\bullet, G)$ more generally, in order to extend equation~(\ref{Sha_P dbl coset}) from fields to schemes.
Giving such a description is possible whenever base change induces an equivalence of categories for torsors as in the previous section:

\begin{prop} \label{double coset lemma}
Given a commutative diagram of schemes $$Y_\bullet = \diagram{Y}$$ and a smooth affine group scheme $G$ over $Y$,
suppose that the functor
\[\tors_G(Y) \to \tors_G(Y_1) \times_{\tors_G(Y_0)} \tors_G(Y_2) \]
induced by base change is an equivalence. Then we have an identification
\(\Sha(Y_\bullet, G) = G(Y_1) \backslash G(Y_0) / G(Y_2).\)
\end{prop}

\begin{proof}
Let $\ttors(Y_\bullet, G)$ be the full subcategory of $\tors_G(Y_1) \times_{\tors_G(Y_0)}
\tors_G(Y_2)$ consisting of objects of the form $(G_{Y_1}, G_{Y_2}, \phi)$,
for some $\phi$ (here, $G_{Y_i}$ is considered as the trivial torsor over
itself). If $P/Y$ is a $G$-torsor, then by definition, $P_{Y_1}$ and
$P_{Y_2}$ are trivial torsors if and only if the image of $P$ under the
functor induced by base change is isomorphic to an object of $\ttors(Y_\bullet, G)$. Since
base change induces an equivalence of categories, it follows that $\ttors(Y_\bullet, G)$
is equivalent to the essential image of those torsors that are trivial
when restricted to both $Y_1$ and $Y_2$, i.e., to the torsors classified by the elements of $\Sha(Y_\bullet, G)\subseteq H^1(Y,G)$.

It remains to show that we have a bijection between isomorphism classes of
objects in $\ttors(Y_\bullet, G)$, and elements of the double coset space $G(Y_1)
\backslash G(Y_0) /G(Y_2)$.
For this, first note that the automorphisms of right $G_{Y_i}$-torsors may be
identified with elements of $G(Y_i)$, acting via left multiplication.
We define our bijection by mapping $(G_{Y_1},
G_{Y_2}, \phi)$ to the double coset $G(Y_1) g G(Y_2)$, where $\phi$ acts by
left multiplication by $g \in G(Y_0)$. We note that
$(G_{Y_1}, G_{Y_2}, \phi) \cong (G_{Y_1}, G_{Y_2}, \phi')$  if and only if
we can find morphisms $\psi_i : G_{Y_i} \to G_{Y_i}$  of $G_{Y_i}$-torsors (for $i=1,2$)
such that the diagram
\[\xymatrix{G_{Y_0} \ar[r]^{\phi} \ar[d]_{(\psi_1)_{Y_0}} & G_{Y_0}
\ar[d]^{(\psi_2)_{Y_0}} \\
G_{Y_0} \ar[r]_{\phi'} \ar[r] & G_{Y_0},}\]
commutes. But if $\phi, \phi'$ are induced by $g, g' \in
G(Y_0)$, respectively, and each
$\psi_i$ is induced by $g_i \in G(Y_i)$, then the diagram reads:
\[ g' = g_1 g g_2^{-1}. \]
This shows that the objects are isomorphic if and only if $g$ and $g'$ are
in the same double coset, as claimed.
\end{proof}

Coming back to the concrete situation of Sections~\ref{patching subsec} and~\ref{lgp sgp}, we then have the following:

\begin{cor}\label{model double coset}
Let $\XX$ be a normal model of a semi-global field~$F$, and let $\XX_\bullet$ be as in diagram~(\ref{patch diagrams}). Then for every smooth affine group scheme $G$  over $\XX$,
$$\Sha({\XX_\bullet},G)=G(\wh R_\UU)\backslash G(\wh R_\BB)/G(\wh R_\PP).$$
\end{cor}

\begin{proof}
By \cite[Corollary~3.1.5]{HKL}, base change
induces an equivalence of categories
\[\tors_G(\XX) \to \tors_G(\wh R_\UU) \times_{\tors_G(\wh R_{\BB})} \tors_G(\wh R_\PP).\] The assertion then follows from Proposition~\ref{double coset lemma}.
\end{proof}

\begin{cor}\label{closed fiber double coset}
Let $\XX$ be a normal model of a semi-global field~$F$, and let $({\XX}_k^{\red})_\bullet$ be as in diagram~(\ref{closed fiber diagram}). Then for any smooth affine group scheme $G$ over $\XX$,
$$\Sha(({\XX}_k^{\red})_\bullet,G)=G(\wh R_\UU/\sqrt{(t)}\,)\backslash G(\wh R_\BB/\sqrt{(t)}\,)/G(\wh R_\PP/\sqrt{(t)}\,).$$
\end{cor}

\begin{proof}
By \cite[Corollary~3.1.2]{HKL}, base change
induces an equivalence of categories
\[\tors_G(\XX_k^{\red}) \to \tors_G(\wh R_\UU/\sqrt{(t)}\,) \times_{\tors_G(\wh R_{\BB}/\sqrt{(t)}\,)} \tors_G(\wh R_\PP/\sqrt{(t)}\,).\]
The assertion then follows from Proposition~\ref{double coset lemma}.
\end{proof}

Note that the identifications in Equation~(\ref{Sha_P dbl coset}), Corollary~(\ref{model double coset}), and Corollary~(\ref{closed fiber double coset}) are compatible, because each is given by base change.  Thus we have the following:

\begin{cor} \label{commu-diag}
In the notation of Corollaries~\ref{model double coset} and~\ref{closed fiber double coset},
let $G$ be a smooth affine group scheme over $\XX$.
Then we have the following commutative diagram of sets
\[
\begin{array}{ccccc}
H^1(F,G) & \supseteq & \Sha_\PP(F,G)  & \to &  \left.  G(F_\UU)  \middle\backslash G(F_\BB)  \middle\slash G(F_\PP) \right. \\[.15cm]
\uparrow & & \uparrow & & \uparrow \\[.15cm]
H^1(\XX,G) & \supseteq & \Sha(\XX_\bullet, G) & \to &   G(\wh{R}_\UU)  \backslash G(\wh{R}_\BB)  \slash G(\wh{R}_\PP)  \\[.15cm]
\downarrow & & \downarrow & & \downarrow \\[.15cm]
H^1(\XX_k^{\rm red},G) & \supseteq & \Sha((\XX_k^{\rm red})_\bullet, G) & \to &  \left.  G(\wh{R}_\UU/\sqrt{(t)}\,)  \middle\backslash G(\wh{R}_\BB/\sqrt{(t)}\,)  \middle\slash G(\wh{R}_\PP/\sqrt{(t)}\,) \right.
\end{array}
\]
where the upper and lower right vertical maps are respectively induced by the inclusion and reduction maps.
\end{cor}

\begin{cor} \label{torsors closed fiber}
Let $\XX$ be a normal model of a semi-global field~$F$, let $\PP,\UU,\BB$ be as in Section~\ref{patching subsec},
and let $G$ be a smooth affine group scheme over $\XX$.
\renewcommand{\theenumi}{\alph{enumi}}
\renewcommand{\labelenumi}{(\alph{enumi})}
\begin{enumerate}
\item \label{torsors same cl fib}
If $g,g' \in G(\wh R_\BB)$ are congruent modulo $\sqrt{(t)}$, then the corresponding $G$-torsors over $\XX$ (as in Corollary~\ref{model double coset}) restrict to isomorphic $G$-torsors over $\XX_k^{\red}$.
\item \label{torsors triv cl fib}
In particular, if $g \in G(\wh R_\BB)$ is trivial modulo $\sqrt{(t)}$, then the corresponding $G$-torsor over $\XX$ restricts to the trivial $G$-torsor over $\XX_k^{\red}$.
\end{enumerate}
\end{cor}

\begin{proof}
This is immediate from the compatible identifications in Corollaries~\ref{model double coset}-\ref{commu-diag}.
\end{proof}

If $L$ is any field, then an {\em $L$-torus} is a linear algebraic group $T$ over $L$ such that $T\times_L L^\sep$
is isomorphic to a product of copies of ${\mbb G}_{{\rm m},L^\sep}$, where $L^\sep$ is a separable algebraic closure of~$L$.  More generally, if $Z$ is a scheme, then by a {\em $Z$-torus} (or {\em torus over $Z$}) we mean a smooth affine group scheme $T$ over $Z$ for which
there exists a finite \'etale morphism $Y \to Z$
such that $T\times_Z Y$ is isomorphic to a product of copies of ${\mbb G}_{{\rm m}, Y}$.  A torus is {\em split} if it is isomorphic to a product of copies of $\Gm$ over the given base.

\begin{prop} \label{GPS tori}
Let $R$ be a complete  discrete valuation ring with fraction field $K$ and residue field $k$.
Let $\XX$ be a normal model of a semi-global field~$F$ over $K$, and let $\XX_k^{\rm red}$ be its reduced closed fiber.
Let $T$ be a torus over $\XX$.
Then the  image of $\ker(H^1(\XX,T) \to  H^1(\XX_k^{\rm red},T))$  in $H^1(F,T)$ is trivial.
\end{prop}

\begin{proof}  Let $\zeta \in $  ker$(H^1(\XX,T) \to  H^1(\XX_k^{\rm red},T))$.
The image of $\zeta$ in $H^1(\XX_k^{\rm red}, T)$ is trivial, hence so is the image of $\zeta$ in $H^1(\wh{R}_\UU/\sqrt{(t)}, T )$.
 Since $\wh{R}_\UU$  is  $(t)$-adically complete and since $G$ is smooth, it follows that
 the natural map $H^1(\wh{R}_\UU, T) \to H^1(\wh{R}_\UU/\sqrt{(t)}, T)$ is injective (see \cite{Strano}, Theorem~1).
Hence the image of $\zeta $ in $H^1(\wh{R}_\UU, T )$ is trivial.
Similarly the image of $\zeta $ in $H^1(\wh{R}_\PP,  T )$ is trivial.
Thus $\zeta \in \Sha(\XX_\bullet, T) \subseteq H^1(\XX, T)$; moreover, its image in $H^1(F,T)$ lies in
$\Sha_\PP(F,T)$, by the top left part of the diagram in Corollary~\ref{commu-diag}. 

Under the isomorphism
$ \Sha(\XX_\bullet, T) \simeq T(\wh{R}_\UU)\backslash T(\wh{R}_\BB) / T(\wh{R}_\PP)$
in Corollary~\ref{model double coset},
$\zeta$ corresponds to the double coset of an element $(g_\wp) \in T(\wh{R}_\BB)$.
Since $\zeta \in $  ker$(H^1(\XX,T) \to  H^1(\XX_k^{\rm red},T))$, it follows from Corollary~\ref{commu-diag} that
the image of $(g_\wp)$ in $T(\wh{R}_\BB/\sqrt{(t)}\,)$ is in $ T(\wh{R}_\UU/\sqrt{(t)}\,)   T(\wh{R}_\PP/\sqrt{(t)}\,)$.
By formal smoothness, the maps $T(\wh{R}_\UU) \to   T(\wh{R}_\UU/\sqrt{(t)}\,)$
and  $T(\wh{R}_\PP) \to   T(\wh{R}_\PP/\sqrt{(t)}\,)$ are surjective;
so after adjusting the choice of double coset representative of $\zeta$ we may assume
that $(g_\wp) \in T(\wh{R}_\BB)$ has trivial image in $T(\wh{R}_\BB/\sqrt{(t)}\,)$.
By Proposition~2.2 of \cite{GilleTAMS}, it follows from the description of $\Sha_\PP(F,T)$ in \cite[Theorem 3.1(b)]{CHHKPS} that
the image of $\zeta$ in $\Sha_\PP(F,T) \subseteq H^1(F,T)$ is the trivial class.  (In the notation of \cite{CHHKPS}, $F_{U,P}$ denotes $F_\wp$, where $\wp$ is the branch on $U$ at $P$.  For more about the notion of R-equivalence used there, see the discussion following the proof of Lemma~\ref{equal-in-rpmt} below.)
\end{proof}

\section{Local factorization} \label{factor sec}

In this section, we prove
local factorization results for reductive groups, in the context of semi-global fields.  
Our main results here are Proposition~\ref{branch factorization} and Corollary~\ref{atp}, which will permit us afterwards to reduce the study of general reductive groups to those that are anisotropic, in order to obtain results about local-global principles.

We let $\XX$ be a regular model of a semi-global field~$F$ over a complete discretely valued field $K$ having valuation ring $R$ and residue field $k$.
We choose a non-empty finite set $\PP$ of closed points satisfying the conditions given in Section~\ref{patching subsec}, and as in that earlier discussion we 
obtain associated sets $\UU$ and $\BB$.
In addition to the notation from Section~\ref{patching subsec}, we now introduce some other terminology and notation that will be used throughout the remainder of the manuscript.

For a closed point $P \in \XX_k^{\rm red}$, we write $\kappa(P)$ for the residue field of $P$.
If $U$ is a non-empty affine open subset of an irreducible component of $\XX_k^{\red}$,
we write $\kappa(U)$ for the ring of constants on $U$; i.e., $\kappa(U) = \OO(\bar U)$, the ring of functions on the closure of $U$ in $\XX_k^{\red}$.  The field $k'=\kappa(U)$ is a finite extension of $k$; and we write $k'[U]$ and $k'(U)$ for the ring of regular functions and the field of rational functions on $U$.

If $\wp \in \BB$ is a branch at $P \in \PP$, then $\wp$ is a prime ideal of $\wh R_P$, and we may consider the residue ring $\wh R_P/\wp$.  This discrete valuation ring is the complete local ring at $P$ of the irreducible component of $\XX_k^{\red}$ on which $\wp$ lies.  The fraction field of $\wh R_P/\wp$ is the residue field $k_\wp = \wh R_\wp/\wp$ of $\wh R_\wp$; and the residue field of $\wh R_P/\wp$ is $\kappa(P)$.

For a smooth affine group scheme $G$ over $\wh R_\wp$, we write $G_\wp(\wh R_\wp)$ for the kernel of $G(\wh R_\wp) \to G(k_\wp)$; i.e., for the set of elements that become trivial on the closed fiber.

Our first lemmas deal with reductive groups that are anisotropic; i.e., that contain no copy of $\Gm$.  

\begin{lem}\label{singular factor}
Let $\XX$ be a normal crossings model of a semi-global field and let $P$ be a singular point of the reduced closed fiber.
Let $G$ be a reductive group over $\wh R_P$ that is anisotropic over $\ell := \kappa(P)$. 
Let $\wp_1, \wp_2$ be the two branches at $P$.  
\begin{enumerate}[(a)]
\item Then $G(F_{\wp_i}) = G(\wh R_{\wp_i})$ and 
$G(k_{\wp_i}) = G(\wh R_P/\wp_i)$.
\item  Suppose we are given $g_i \in G(F_{\wp_i})$ such that $s_1(g_1)=s_2(g_2)$ under the compositions
\[s_i:G(F_{\wp_i}) = G(\wh R_{\wp_i}) \to G(k_{\wp_i}) = G(\wh R_P/\wp_i) \to G(\ell)\quad (i=1,2).\]
Then there exists $g' \in G(\wh R_P)$ such that $g'g_i \in G_{\wp_i}(\wh R_{\wp_i})$ for $i = 1, 2$.
\end{enumerate}
\end{lem}

\begin{proof}
Since $G$ is anisotropic over $\ell$, it is anisotropic over $k_{\wp_i}$ by Proposition~\ref{local anisotropy} in the Appendix. So by Proposition~\ref{BTR}, we have $G(k_{\wp_i}) = G(\wh R_P/{\wp_i})$ and
$G(F_{\wp_i}) = G(\wh R_{\wp_i})$.  Thus
$g_i \in G(\wh R_{\wp_i})$, and 
the compositions are defined.
We may write $\wp_i = t_i \wh R_P$, for some $t_1,t_2 \in \wh R_P$.  Here 
$t_1, t_2$ form a regular system of parameters for $\wh R_P$, since $\XX$ is a normal crossings model.  Thus $(t_1t_2) = (t_1) \cap (t_2)$, and there is a short exact sequence
\[0 \to \wh R_\wp/\wp_1\wp_2 \, \to \, \wh R_\wp/\wp_1 \times \wh R_\wp/\wp_2 \ \buildrel - \over {\, \to} \ \wh R_\wp/(\wp_1+\wp_2) \to 0.\] 
That is, we have a pullback square of rings
\begin{equation} \label{pullback rings}
\xymatrix{
    \wh R_P/\wp_1\wp_2 \ar[r] \ar[d] & \wh R_P/\wp_1 \ar[d] \\
    \wh R_P/\wp_2 \ar[r] & \wh R_P/(\wp_1 + \wp_2),}
\end{equation}
where the lower right ring is $\kappa(P)=\ell$.
By hypothesis, the images $\bar g_i \in G(\wh R_P/\wp_i)$ of $g_i$ agree in $G(\ell)$, for $i=1,2$.  Since $g_1,g_2$ have the same image in $\wh R_P/(\wp_1 + \wp_2)$, they are induced by a common element $\bar g'' \in G(\wh R_P/\wp_1\wp_2)$.
But $G$ is smooth and $\wh R_P$ is $\wp_1\wp_2$-adically complete, so by formal smoothness we may lift $\bar g''$ to an element $g'' \in G(\wh R_P)$. Let $g':=(g'')^{-1}$. Then $g'g_i\in G(\wh R_{\wp_i})$, and in fact the product is in $G_{\wp_i}(\wh R_{\wp_i})$ because $g''$ and $g_i$ have the same image in $G(\wh R_P/\wp_i)$.
\end{proof}

For the next lemma and some later results,
we need to define rings associated with somewhat more general (not necessarily open) subsets of the reduced closed fiber.  Following \cite[Section~3.1]{HHK15a},
let $W$ be a non-empty affine open subset of an irreducible component
of the reduced closed fiber $\XX_k^{\red}$ of $\XX$.
Thus $W$ is connected.  Note
that $W$ is not necessarily open in $\XX_k^{\red}$; viz., it
will not be open if $W$ contains a singular point of $\XX_k^{\red}$
(i.e., one that lies on two distinct components of $\XX_k^{\red}$).  Let
$R_W$ be the subring of $F$ consisting of those elements that are
regular at every point of $W$, and let $J_W$ be the Jacobson radical
of $R_W$.  Thus $\Spec(R_W/J_W)$ is the reduced closed fiber $W$ of $\Spec(R_W)$ (see \cite[Lemma~3.3]{HHK15a} and the paragraph before that); hence
$R_W/J_W = \ms O(W)$, the affine coordinate ring of $W$.

Let $\wh R_W$ be the $J_W$-adic completion of $R_W$.  The extension of $J_W$
to $\wh R_W$ (which we again call $J_W$) is the Jacobson radical of $\wh R_W$.
Since $W$ is connected, the ring $\wh R_W$ is a domain (by \cite[Proposition~3.4]{HHK15a}).
Also, if $W'$ is a non-empty open
subset of $W$, then $R_W$ is contained in $R_{W'}$, and $J_W^n$ is the
restriction of $J_{W'}^n$ to $R_W$ for every $n \ge 1$; so $\wh R_W$ is contained in $\wh
R_{W'}$.

Note that in the special case that $W$ is a non-empty affine open subset
$U$ of the regular locus of $\XX_k^{\red}$, the definition above
of $\wh R_W$ agrees with the prior definition of $\wh R_U$, since $J$ is
the radical of the ideal $(t)$ in this case.

\begin{lem}\label{adjust over U}
Let $\XX$ be a regular model of a semi-global field over a complete discrete valuation ring with residue field $k$.  Choose a non-empty finite set $\PP$ as in Section~\ref{patching subsec}, with associated sets $\UU,\BB$.  Let
$W$ be a non-empty affine open subset of an irreducible component
of the reduced closed fiber $\XX_k^{\red}$.  Let $P \in \PP \cap W$ be a closed point such that $U := W \smallsetminus \{P\}$ is in $\UU$, and suppose that $\kappa(U)$ is equal to $\kappa(P) =:\ell$.  Consider a reductive group $G$ over $\wh R_W$ that is anisotropic over $\ell$, and take $h\in G(\ell)$.
Let $\wp$ be a branch at $P$ along $U$, and let $g \in G(F_\wp)$. Then
there exists $g' \in G(\wh R_U) \subset G(\wh R_\wp)$ such that $gg'$ maps to $h$ under the composition
\[s:G(F_\wp) = G(\wh R_\wp) \to G(k_\wp) = G(\wh R_P/\wp) \to G(\ell).\]
\end{lem}

\begin{proof}
Since $\kappa(U) =\ell$,
we have an inclusion
$\ell \hookrightarrow \mc O(U) = \ell[U]$, which induces an inclusion $G(\ell) \hookrightarrow G(\ell[U])$.  Let $g_0' \in G(\ell[U])$ be the image of $s(g)^{-1}h \in G(\ell)$ under $G(\ell) \hookrightarrow G(\ell[U])$.  Since $\wh R_U$ is complete and $G$ is smooth,
there is an element $g' \in G(\wh R_U)$ whose image under the map $G(\wh R_U) \to G(\ell[U])$ is $g_0'$, by formal smoothness.
Viewing $G(\wh R_U)$ as a subgroup of $G(\wh R_\wp)$,
we have that $s(g')=s(g)^{-1}h$ by the commutativity of the following diagram:
    \[\xymatrix{
    G(\wh R_U) \ar@{^{(}->}[r]  \ar@{->>}[d] & G(\wh R_\wp)  \ar[d] &  \\
    G(\ell[U]) \ar@{^{(}->}[r] & G(k_\wp)  \ar@{=}[r] & G(\wh R_P/\wp) \ar[d] \\
    G(\ell) \ar@{^{(}->}[u] \ar[rr]^{\rm id} && G(\ell)\\
  }\]
Hence $s(gg')=h$ as asserted.
\end{proof}

Consider a smooth affine group scheme $G$ over $\wh R_P$, and let $\wp$ be a branch on $U \in \UU$ at $P \in \PP$.  Then $G(\wh R_P)$ is a subgroup of $G(\wh R_\wp)$, and
$G_\wp(\wh R_\wp)$ is a normal subgroup of $G(\wh R_\wp)$.  So we may take the product
\begin{equation}\label{specializable}G_{\rm s}(\wh R_\wp) :=
G_\wp(\wh R_\wp)G(\wh R_P) =
G(\wh R_P)G_\wp(\wh R_\wp).\end{equation} 
This is a subgroup of $G(\wh R_\wp)$, referred to as the subgroup of {\em specializable} elements. It
is the inverse image of $G(\wh R_P/\wp) = G(\wh\OO_{\bar U,P})$
under the reduction map
$G(\wh R_\wp) \to G(k_\wp) = G(k(U)_P)$; here $k(U)$ is the function field of $U$, and 
$k(U)_P$ is its completion with respect to the discrete valuation defined by $P$.
Note that $G_{\rm s}(\wh R_\wp)$ and $G(\wh R_P)$ have the same image in $G(k_\wp)$.
Let $\mathfrak m_P$ be the maximal ideal of $\wh R_P$.
Composing the reduction maps $G_{\rm s}(\wh R_\wp) \to G(\wh R_P/\wp)$
and
$G(\wh R_P/\wp) \to G(\wh R_P/\mathfrak m_P) = G(\kappa(P))$ yields a {\em specialization} map  $\Theta_\wp:G_{\rm s}(\wh R_\wp) \to G(\kappa(P))$. (This is related to the specialization maps that appear in
Sections~\ref{sectionlb} and~\ref{Sha computation section} and in the Appendix.)

\begin{lemma} \label{equal-in-rpmt}
Let $\XX$ be a normal crossings model of a semi-global field over a complete discrete valuation ring $R$, and consider a finite set $\PP$ of points on the reduced closed fiber as in Section~\ref{patching subsec}.  Let
$G$ be a reductive  group over $R$, and let $P \in \PP$. 
\renewcommand{\theenumi}{\alph{enumi}}
\renewcommand{\labelenumi}{(\alph{enumi})}
\begin{enumerate}
\item \label{equal reg}
If $P$ is a regular point of the reduced closed fiber with $\wp$ its unique branch, and if $g_\wp,g_\wp'$ are elements of $G_{\rm s}(\wh R_\wp)$, then there exists $\beta_P \in G(\wh R_P)$ such that the elements $g_\wp\beta_P$ and $g'_\wp$ of $G_{\rm s}(\wh R_\wp)$ have equal images in $G(k_\wp)$.
\item \label{equal nodal}
If $P$ is a singular point of the reduced closed fiber with branches $\wp_1,\wp_2$, and if
$g_{\wp_i},g_{\wp_i}' \in G_{\rm s}(\wh R_{\wp_i})$ are elements satisfying
$\Theta_{\wp_i}(g_{\wp_i}) = \Theta_{\wp_i}(g'_{\wp_i}) \in G(\kappa(P))$ for $i=1,2$,
then there exists $\beta_P \in G(\wh R_P)$ such that the elements
$g_{\wp_i}\beta_P$ and $g'_{\wp_i}$ of $G_{\rm s}(\wh R_{\wp_i})$ have equal images in $G(\wh R_{\wp_i}/{\wp_i})$ for $i=1,2$.
\end{enumerate}
\end{lemma}

\begin{proof}
For part~(\ref{equal reg}), since $G_{\rm s}(\wh R_\wp)$ and $G(\wh R_P)$ have the same image in $G(k_\wp)$, there exists a lift $\beta_P \in G(\wh R_P) \subseteq G_{\rm s}(\wh R_\wp)$ of the image of $g_\wp^{-1}g'_\wp \in G_{\rm s}(\wh R_\wp)$ in $G(k_\wp)$.  Then $\beta_P$ has the desired property.

For part~(\ref{equal nodal}), 
we will proceed similarly to the proof of Lemma~\ref{singular factor}
(but using the assumption here that the given elements are specializable, to replace the anisotropy hypothesis of Lemma~\ref{singular factor}).  
Namely,
the maximal ideal of $\wh R_P$ is generated by two elements $t_1,t_2$, where
$\wh R_{\wp_i}$ is the completion of the localization of $\wh R_P$ at $(t_i)$, 
and $\wp_i = t_i \wh R_P$.
For $i=1,2$ let $\bar g_{\wp_i}, \bar g'_{\wp_i} \in G(\wh R_P/\wp_i)$ be the images of $g_{\wp_i}, g'_{\wp_i}$ and write $\bar \beta_{\wp_i} = \bar g_{\wp_i}^{-1}\bar g'_{\wp_i}$.  Since $\Theta_{\wp_i}(g_{\wp_i}) = \Theta_{\wp_i}(g'_{\wp_i}) \in G(\kappa(P))$, the image of $\bar \beta_{\wp_i}$ in $G(\kappa(P))$ is equal to $1$, for $i=1,2$.

We have the pullback diagram of rings (\ref{pullback rings}) as in the proof of Lemma~\ref{singular factor}.
Hence there is a well-defined element $\tilde \beta_P \in G(\wh R_P/\wp_1\wp_2)$ that maps to $\bar \beta_{\wp_i}$ in $G(\wh R_P/\wp_i)$ for $i=1,2$.
Since $G$ is a smooth $R$-scheme and the ring $\wh R_P$
is complete with respect to $\wp_1\wp_2$,
by formal smoothness we may choose a lift $\beta_P \in G(\wh R_P)$ of $\tilde \beta_P$.
Thus the elements $g_{\wp_i}\beta_P$ and $g'_{\wp_i}$ of $G_{\rm s}(\wh R_{\wp_i})$ have equal images in $G(\wh R_{\wp_i}/\wp_i)$ for $i=1,2$, as required.
\end{proof}

Recall that for any field $E$, two $E$-points of an $E$-variety $V$ are {\em directly $\R$-equivalent} if they are in the image of $E$-points of $\A^1_E$ under some rational map $f:\A^1_E \dashrightarrow V$.  The transitive closure of this relation is {\em $\R$-equivalence}.  If $V$ is a linear algebraic group $G$ over a field $E$, then any two $\R$-equivalent points of $G(E)$ are directly $\R$-equivalent (see \cite[II.1.1(b)]{GilleIHES}).  The points of $G(E)$ that are $\R$-equivalent to the identity are called $\R$-{\em trivial}, and these form a normal subgroup $\R = \R G(E)$ of $G(E)$.  The quotient group $G(E)/\R$ is the group of $\R$-{\em equivalence classes} of $G(E)$.

The following result will be the key ingredient in reducing to the anisotropic case
in the proof of Proposition~\ref{anisotropic factorization}; and that proposition in turn will lead to the proofs of our local-global results in the later sections, via Theorems~\ref{torsors extend} and~\ref{main}.

\begin{prop} \label{branch factorization}
Let $\XX$ be a regular model of a semi-global field over a complete discretely valued field having residue field $k$.
Let $P$ be a closed point of  ${\ms
X}_k^{\red}$, and suppose $G$ is a reductive group over $\wh R_P$.
Let $S \subseteq G$ be a maximal split torus in $G$ and let $H =
C_G(S)$, the centralizer of $S$ in $G$.
Let $\wp_1, \wp_2, \ldots, \wp_r$ be a collection of distinct branches at
$P$. Suppose we are given $g_i \in G(F_{\wp_i})$ for every $i$. Then there is an element $g \in G(F_P)$
that is $\R$-trivial in each $G(F_{\wp_i})$, together with
elements $h_i^* \in H(\wh R_P)$ for all $i$, such that $h_i^* g_ig \in G_{\wp_i}(\wh R_{\wp_i})$.  That is, each $g_i g \in H(\wh R_P) G_{\wp_i}(\wh R_{\wp_i})
\subseteq G(\wh R_P)G_{\wp_i}(\wh R_{\wp_i}) = G_{\rm s}(\wh R_{\wp_i})$.
\end{prop}

\begin{proof}
As in the statement of Proposition~\ref{decomposition} of the Appendix, the group $H$ is reductive, and there are parabolic subgroups of
$G$ with unipotent radicals $\mc U, \mc U'$ such that the morphism $H \times \mc U \times \mc U' \to G$ is an open immersion.  (See the proof of Proposition~\ref{decomposition} for more details.)
Let $\mc C$ be the image of this open immersion.

Since $G$ is reductive, it follows that $G_{F_P}$ is unirational (see \cite[Chap.~V, Theorem~18.2]{Borel}).
Choose a dominant morphism $\phi: V \to G_{F_P}$ where $V \subseteq \mbb A^N_{F_P}$ is a Zariski affine open set. We may assume that $\phi(V)$ contains the identity element of $G$.
For $i = 1, \ldots, r$, 
let $V_i = \phi^{-1}(g_i^{-1}\mc C)$; 
this is an open subscheme of $V_{F_{\wp_i}}$. Since these are all nonempty open in affine space, we may choose points $v_i \in V_i(F_{\wp_i})$.
Since $V_i$ is Zariski open, the set $V_i(F_{\wp_i})$ is open in the $\wp_i$-adic topology on $V(F_{\wp_i})$. By weak approximation, we may find $v \in \mbb A^N(F_P)$ that is sufficiently close $\wp_i$-adically to $v_i$ for each $i$,
so as to make $v \in V_i(F_{\wp_i})$ for all $i$.  
Since $V \subseteq \mbb A^N_{F_P}$ is an open subset and $\phi(V)$ contains the identity of $G$, the element
$\tilde g := \phi(v)$ is $\R$-trivial in $G(F_P)$, and hence in each $G(F_{\wp_i})$.  Moreover 
$g_i\tilde g \in \mc C(F_{\wp_i})$ for each $i$.

Hence we may write 
$g_i\tilde g = h_iu_i u_i'$, where $h_i \in H(F_{\wp_i})$, $u_i \in 
\mc U(F_{\wp_i})$, and $u_i' \in \mc U'(F_{\wp_i})$.
By Proposition~\ref{H and S}(\ref{almost anisotropic factorization}) in the Appendix, we may write $h_i = h_i's_i$ for some $h_i' \in H(\wh R_{\wp_i})$ and $s_i \in S(F_{\wp_i})$.
Let $\bar h_i'$ be the reduction of $h_i'$ modulo $\wp_i$; thus $\bar h_i' \in H(k_{\wp_i})$, where $k_{\wp_i}$ is the residue field of the complete discrete valuation ring $\wh R_{\wp_i}$.
Since $k_{\wp_i}$ is also the fraction field of the complete discrete valuation ring $\wh R_P/{\wp_i}$, it again follows from Proposition~\ref{H and S}(\ref{almost anisotropic factorization})
that $\bar h_i' = \bar h_i'' \bar s_i'$ for some $\bar h_i'' \in H(\wh R_P/{\wp_i})$ and $\bar s_i' \in S(k_{\wp_i})$.  By the smoothness of $H$
and $S$, we may lift these elements to some $h_i'' \in H(\wh R_P)$ and $s_i' \in S(\wh R_{\wp_i})$.  Thus $h_i'$ and $h_i''s_i'$ are elements of $H(\wh R_{\wp_i})$ that have the same image modulo $\wp_i$.
Hence $h_i' = h_i''s_i'h_i'''$ for some $h_i''' \in H_{\wp_i}(\wh R_{\wp_i})$ and thus
$g_i\tilde g = h_iu_i u_i' = h_i's_iu_iu_i' =
h_i''s_i'h_i'''s_iu_iu_i' = h_i''h_i'''s_i's_iu_iu_i'$.

Consider the map $\psi:S \times \mc U \times \mc U' \to G$ defined by $\psi(s,u,u')= suu'$.  Since $S$ is a split torus, $S_{F_P}$ is a rational variety.  Moreover, $\mc U_{F_P}$ and $\mc U'_{F_P}$ are also rational varieties, by \cite[Exp.~XXVI, \S 2, Corollaire~2.5]{SGA3.3} (where the notation $\mathrm{W}$ in that corollary is defined 
in \cite[I, D\'efinition~4.6.1]{SGA3.1}).
By weak approximation and continuity we may find
$s \in S(F_P)$, $u \in \mc U(F_P)$, and $u' \in \mc U'(F_P)$
such that
$s_i's_iu_iu_i'(suu')^{-1} = \psi(s_i's_i,u_i,u_i')\psi(s,u,u')^{-1}
\in G_{\wp_i}(\wh R_{\wp_i})$ for all $i$.
We can rewrite this element as
$h_i'''^{-1}h_i''^{-1}g_i\tilde g(suu')^{-1}$, by the above expression for 
$g_i\tilde g$.  But $h_i''' \in H_{\wp_i}(\wh R_{\wp_i}) \subseteq
G_{\wp_i}(\wh R_{\wp_i})$.  So $h_i''^{-1}g_i\tilde g(suu')^{-1} \in G_{\wp_i}(\wh R_{\wp_i})$,
and we can take $g = \tilde g(suu')^{-1} \in G(F_P)$
and $h_i^* = h_i''^{-1} \in H(\wh R_P)$.  We observed that the element $\tilde g$ is $\R$-trivial in each $G(F_{\wp_i})$; and so are $s,u,u'$, by the rationality of the groups $S_{F_P},\mc U_{F_P}, \mc U'_{F_P}$.  Hence $g$ is also $\R$-trivial in each $G(F_{\wp_i})$, as asserted.
\end{proof}

\begin{cor}\label{atp}
Let $\XX$, $P$, and $G$ be as in Proposition~\ref{branch factorization}, let $\wp$ be a branch of the reduced closed fiber of $\XX$ at $P$, and let $g_\wp \in G(F_\wp)$. Then there exists $g_P \in G(F_P)$ such that $g_\wp g_P \in G_\wp(\wh R_\wp)$.
\end{cor}

\begin{proof}
By Proposition~\ref{branch factorization}, there is an element $g \in G(F_P)$
such that $g_\wp g \in G_{\rm s}(\wh R_\wp)$.  That is, $g_\wp g = g'_\wp g_0$ with
$g'_\wp \in G_\wp(\wh R_\wp)$ and $g_0 \in G(\wh R_P)$.  We may thus set $g_P = g g_0^{-1}$.
\end{proof}

\section{Locally trivial torsors} \label{local triv sec}

In this section, we will prove a result on lifting torsors over a semi-global field $F$ to torsors over a model of $F$ that are trivial over the reduced closed fiber; see Theorem~\ref{torsors extend}.  This theorem will be key to proving the local-global results of Section~\ref{lgp mon tree sec}.  Throughout this section, we 
restrict attention to normal crossings models $\XX$ of a semi-global field $F$ over a complete discrete valuation ring $R$.  As before, we let $K$ and $k$ be the fraction field and residue field of $R$, respectively. 

We begin this section by proving a local-global result for torsors on a patch (Proposition~\ref{sha-u-p}).  We then use that to prove the equivalence of various notions of a $G$-torsor being locally trivial.  Namely, in Theorem~\ref{equalshas} we
show that for $G$ a reductive group over $R$, and $F$ a semi-global field over $R$, the local-global obstruction sets $\Sha_\PP(F,G)$ are independent of the choice of $\PP$ and even of the choice of normal crossings model $\XX$; viz., they are all equal to $\Sha(F,G)$.  Hence by Proposition~\ref{sha containments}, if $X$ is the reduced closed fiber of a normal crossings model $\XX$, the sets $\Sha_\PP(F,G)$ are also all equal to $\Sha_X(F,G)$.  This assertion is similar to that of \cite[Thm.~4.4]{CHHKPS}, which concerned the case of a torus over $\XX$, though the proof there was different (it relied on flasque resolutions of tori).

Afterwards, we turn our attention to a special class of normal crossings models that was defined in \cite{CHHKPS}; viz., those models whose reduction graph is a monotonic tree (see below).  
For such a model $\XX$, we show in Theorem~\ref{torsors extend} that any $G$-torsor over $F$ whose class lies in $\Sha(F,G)$ can be lifted to a torsor over $\XX$ that is trivial along the reduced closed fiber $\XX_k^{\rm red}$.  Theorem~\ref{torsors extend} is proven using Proposition~\ref{anisotropic factorization}, a global factorization result that relies on the local factorization results of Section~\ref{factor sec}.

\subsection{Comparison of obstruction sets}
In this subsection, we prove a comparison result for obstructions sets. A key ingredient is the following 

\begin{prop} \label{sha-u-p}
Let $\XX$ be a normal crossings model of a semi-global field $F$ over a discrete valuation ring $R$, take an irreducible component $Z$ of its reduced closed fiber $\XX_k^{\red}$, and let $V$ be a nonempty affine open subset of $Z$ that does not intersect any
other irreducible component of $\XX_k^{\red}$.  Let $P \in V$ be a closed point and let $W = V \smallsetminus \{ P \}$. Let $\xi \in H^1(F_V, G)$, where $G$ is
a reductive group over $R$.   If $\xi$ is trivial over $F_W$ and $F_{P}$, then $\xi$ is trivial over $F_V$.
\end{prop}

\begin{proof}
Let $\wp$ be the unique branch at $P$ on $V$ (or equivalently, on $W$).  In the situation of
Definition~\ref{def_sha},
let $Y_\bullet$ be the commutative diagram of schemes associated to
$Y_0=\Spec(F_\wp)$, $Y_1=\Spec(F_W)$, $Y_2=\Spec(F_P)$, and $Y=\Spec(F_V)$.
Since $\xi$ is trivial over $F_W$ and $F_{P}$, it lies in $\Sha_P(F_V, G) := \Sha(Y_\bullet, G)$.

By \cite[Theorem 2.4]{HHK15} and \cite[Proposition 3.9]{HHK15a}, the natural map
\[\Sha_P(F_V, G) \to G(F_W) \backslash G(F_{\wp}) \slash G(F_P)\]
is a bijection.  Let
$g_0 \in G(Y_0) = G(F_\wp)$ be a representative of the double coset corresponding to $\xi$.
By Corollary~\ref{atp},
there exists $g_2 \in G(Y_2) = G(F_P)$ such that $g_\wp:=g_0g_2 \in G_\wp({\wh R}_\wp) \subset G(F_\wp)$.
The elements $g_0,g_\wp$ lie in the same double coset, and so $g_\wp$ also represents $\xi$.

Let $\PP \subset \XX_k^{\red}$ be the finite set of closed points consisting of the points in $\ov W \smallsetminus W$ together with the points at which irreducible components of $\XX_k^{\rm red}$ intersect.  Let $\UU$ be the set of connected components of $\XX_k^{\rm red} \smallsetminus \PP$, and let $\BB$ be the set of branches at the points of $\PP$.
By \cite[Corollary~3.6]{HHK15}, the natural map
\[\Sha_\PP(F, G) \to \prod_{U' \in \UU} G(F_{U'}) \backslash \prod_{\wp' \in \BB} G(F_{\wp'}) \slash \prod_{P' \in \PP} G(F_{P'})\]
is a bijection.  For $\wp' \in \BB$ with $\wp' \ne \wp$, let $g_{\wp'} = 1$, and
let $\zeta \in \Sha_\PP(F, G)$
be the $G$-torsor over $F$ corresponding to the class of
$(g_{\wp'})_{\wp' \in \BB} \in \prod_{\wp' \in \BB} G(F_{\wp'})$ in this double coset space.  Thus $\zeta$ induces $\xi$ under the natural map $H^1(F,G) \to H^1(F_V,G)$.

Consider the subset $\Sha_\PP(\XX,G) \subset H^1(\XX,G)$ consisting of the $G$-torsors over $\XX$ that are trivial over each $\wh R_{P'}$ and each $\wh R_{U'}$, for $P' \in \PP$ and $U' \in \UU$.
The natural map
\[\Sha_\PP(\XX, G) \to \prod_{U' \in \UU} G(\wh R_{U'}) \backslash \prod_{\wp' \in \BB} G(\wh R_{\wp'}) \slash \prod_{P' \in \PP} G(\wh R_{P'})\]
is bijective, by Corollary~\ref{model double coset}.
Let $\tilde \zeta$ be the $G$-torsor over $\XX$ corresponding to the class of $(g_{\wp'})_{\wp' \in \BB} \in \prod_{\wp' \in \BB} G(\wh R_{\wp'})$ in this double coset space.  Thus $\tilde \zeta$ induces $\zeta$ under the natural map $H^1(\XX,G) \to H^1(F,G)$, and hence also induces $\xi \in H^1(F_W,G)$.

Consider the restriction of $\tilde \zeta$ to a $G$-torsor on $\XX_k^{\rm red}$.  Since each $g_{\wp'}$ is congruent to $1$ modulo $\wp' = \sqrt{t\wh R_{\wp'}}$,
this restriction is trivial by Corollary~\ref{closed fiber double coset}.  But the natural map $H^1(\wh R_V,G) \to H^1(k[V],G)$ is injective, by smoothness of $G$ and completeness of $R_V$ (see \cite{Strano}, Theorem~1).
Thus the generic fiber $\zeta$ of $\tilde\zeta$ is a trivial $G$-torsor over $F$.  Hence its image $\xi \in H^1(F_V,G)$ is the trivial $G$-torsor over $F_V$.
\end{proof}

The following assertion generalizes \cite[Theorem~4.4]{CHHKPS} from the case of tori to the case of reductive groups.  In the torus case we had permitted the group to be defined over $\XX$, whereas in the next theorem we require it to be defined over $R$.

\begin{thm} \label{equalshas}
Let $\XX$ be a normal crossings model of a semi-global field $F$ over a complete discrete valuation ring $R$ having residue field $k$, and let $X = \XX_k^{\red}$.  As in Section~(\ref{patching subsec}), let $\PP$ be a non-empty finite set of closed points of $\XX_k^{\rm red}$ containing all the singular points.  Let
$G$ be a reductive group over $R$.  
Then $\Sha_{\PP}(F, G) = \Sha_X(F, G) = \Sha(F, G)$.
\end{thm}

\begin{proof}  Since $G$ is a reductive group over $R$,
it is in particular reductive over $\XX$. So by parts~(\ref{ShaX union}) and~(\ref{ShaX equals Sha}) of Proposition~\ref{sha containments},
\[\Sha(F, G) = \Sha_X(F,G) = \cup_{\PP'} \Sha_{\PP'}(F, G),\]
where $\PP'$ runs over
all finite sets of closed points of $X$ containing  $\PP$.
Thus it suffices to show that the containment $\Sha_{\PP}(F, G) \subseteq
\Sha_{\PP'}(F, G)$ is an equality, for all finite subsets $\PP'$  of closed points of $\XX$
containing  $\PP$.

So let $\PP'$ be such a finite subset, and let $\xi \in \Sha_{\PP'}(F, G)$.
For each $P \in \PP$, we have $P \in \PP'$ and hence $\xi$ is trivial over $F_P$.

To prove $\Sha_{\PP} (F, G) = \Sha_{\PP'}(F, G)$, it remains to show that
$\xi$ is  trivial over $F_U$, for each irreducible component $U$ of $X\smallsetminus \PP$.
Since $\PP \subseteq \PP'$, there exists a  unique irreducible component  $V$ of $X\smallsetminus \PP'$ such that $V \subseteq U$.
Thus $\xi$ is trivial over $F_V$, since $\xi \in \Sha_{\PP'}(F, G)$.

Suppose $V \neq U$, and let $P \in U  \smallsetminus V$.
Thus $P \in \PP' \smallsetminus \PP$.   In particular, $P$ is a regular point of $X$. 
Also, $\xi$ is trivial over $F_P$, since $\xi \in \Sha_{\PP'}(F, G)$.
Hence,  by Proposition~\ref{sha-u-p}, $\xi$ is trivial over $F_{V \cup \{P\}}$.  Since $U \smallsetminus V$ is finite, it follows by induction that
 $\xi$ is trivial over $F_U$.  Thus $\xi \in \Sha_{\PP}(F, G)$.
 \end{proof}
 
Thus each of the double coset spaces
$\Sha_\PP(F,G) \simeq G(F_\UU)\backslash G(F_\BB)/G(F_\PP)$ in~(\ref{Sha_P dbl coset}) of Section~\ref{lgp sgp}, and not just their direct limit, becomes identified with $\Sha(F,G)$ 
under the hypotheses of Theorem~\ref{equalshas}.

\subsection{Monotonic trees}\label{subsec mono}

As before, let $F$ be a semi-global field with normal crossings model $\XX$.  Choose sets $\PP, \UU, \BB$ as in Section~\ref{patching subsec}, and let $\Gamma$ be the corresponding reduction graph as defined there.  
To each vertex $v \in \PP \cup \UU$ of $\Gamma$ we associate the field $\kappa(v)$ as before (viz., $\kappa(P)$ or $\kappa(U)$; see the beginning of Section~\ref{factor sec}).  Following \cite{CHHKPS}, Section~7, we
say that the reduction graph $\Gamma$ is a {\em monotonic tree} if it is a tree and there is some vertex $v_0$ (the {\em root}) with the following property:
If $v,w$ are vertices of $\Gamma$ and $v$ is the parent of $w$ (i.e., $v,w$ are adjacent, and $v$ lies on the unique (non-repeating) path connecting $v_0$ to $w$)
then $\kappa(v) \subseteq \kappa(w)$. 

If $P \in \PP$ lies on the closure of $U \in \UU$, then $\kappa(P)$ necessarily contains $\kappa(U)={\mc O}(\bar U)$; hence if $P$ is the parent of $U$ on a monotonic tree then necessarily $\kappa(P)=\kappa(U)$. For that reason, the root of a monotonic tree can always be chosen so as to correspond to some $U_0 \in \UU$; and we will always do so from now on.  
For every vertex $v$ of $\Gamma$, we may consider its {\em distance} from $v_0$; viz., the number of edges in the unique path from $v_0$ to $v$.  Since each edge connects an element of $\UU$ to an element of $\PP$,
the vertices corresponding to elements of $\UU$ have even distance, while those corresponding to elements of $\PP$ have odd distance.

As shown in \cite[Proposition~7.6]{CHHKPS}, if $\Gamma$ is a monotonic tree, then for every finite extension $k'/k$ of the residue field $k$, the graph associated to the base change from $k$ to $k'$ of the reduced closed fiber is also a tree.  Moreover, the converse holds if $k$ is perfect.

\begin{rem} \label{tree independent}
Enlarging the set $\PP$ preserves the monotonic tree property; so since any two choices of $\PP$ have a common refinement, the property is independent of the choice of $\PP$ on a given normal crossings model.  Moreover, the monotonic tree property is preserved under blow-up, and as a consequence it is independent of the choice of normal crossings model of $F$ (see \cite[Remark~7.1]{CHHKPS}). Thus it is meaningful to say that the reduction graph of a semi-global field $F$ is a monotonic tree, without referring to $\XX$ or $\PP$.
\end{rem}

As at the end of Section~\ref{lgp sgp}, given  a refinement $\PP',\UU',\BB'$
of $\PP,\UU,\BB$, to each element of
$G(F_\BB) = \prod_{\wp \in \BB} G(F_\wp)$ there is an associated element of $G(F_{\BB'})$, 
obtained by inserting the identity of $G$ at each entry indexed by an element of $\BB' \smallsetminus \BB$.

\begin{prop} \label{anisotropic factorization}
Let $\XX$ be a normal crossings model of a semi-global field $F$ over a complete discrete valuation ring $R$.  Let $\PP, \UU, \BB$ as in Section~\ref{patching subsec}, and assume
that the associated reduction graph is a monotonic tree.
    Let $G$ be a reductive group over $\basering$, and let $(g_\wp)_{\wp \in \BB} \in G(F_\BB)$.
  Then possibly after $\PP,\UU,\BB$ have been refined (and $(g_\wp)_\wp$ has been replaced by the associated element of $G(F_{\BB'})$),
  there exist $(g_U)_{U \in \UU} \in G(F_\UU)$ and $(g_P)_{P \in \PP} \in G(F_\PP)$  such that $g_U g_\wp g_P \in G_{\wp}(\wh R_\wp)$ whenever $\wp$ is a branch at $P$ along $U$.
\end{prop}

\begin{proof}
As above, we assume that the root of the reduction graph $\Gamma$ is a vertex $v_0$ corresponding to some $U_0 \in \UU$.
We need to define elements $g_v$ for all vertices $v$, and we will do this by induction on the distance from $v_0$.
Write $\PP^{\text{sing}}$ for the subset of $\PP$ consisting of singular points of $\XX_k^{\red}$ and write $\PP^{\text{reg}} \subseteq \PP$ for the subset of regular points.
Then the vertices corresponding to elements of $\PP^{\text{reg}}$ are leaves of the tree $\Gamma$, connected to the rest of the tree by just one edge; while for those in $\PP^{\text{sing}}$, there are two edges at the vertex, of which exactly one is on the path connecting the vertex to $v_0$.

The unique vertex of distance zero is $v_0$, corresponding to $U_0$, and we set $g_{v_0} = g_{U_0}$ equal to $1 \in G(F_{U_0})$.  For the inductive step, let $n \ge 0$ and assume that $g_v$ has been defined for every vertex $v$ of distance at most $2n$.  We will define $g_v$ for all vertices $v$ of distance $2n+1$ and $2n+2$ (which are respectively in $\PP$ and in $\UU$).
In the process, we may refine $\PP, \UU, \BB$ by adding finitely many regular points of $\XX_k^{\rm red}$ to $\PP$.  These points correspond to new terminal vertices of distance $2n+3$ from the root, and they are handled at the next step of the induction.  Note that the inductive process will still terminate, since refining $\PP, \UU, \BB$ multiple times can increase the maximum distance of vertices from the root by at most one in total.

Let $P\in \PP$ correspond to a vertex of distance $2n+1$ from $v_0$.  We first consider the case that $\PP^{\text{reg}}$.  Thus $P$ is unibranched on the reduced closed fiber, and is a leaf of the tree $\Gamma$.  The unique component $V \in \UU$ on which $P$ lies is the parent of $P$ in $\Gamma$, having distance $2n$ from $v_0$ (so that $g_V$ is already defined).  Let $\wp$ be the unique branch at $P$, and set $g_{\wp}'=g_Vg_{\wp} \in G(F_\wp)$.  By Corollary~\ref{atp}, there exists $g_P \in G(F_P)$ such that $g_\wp' g_P \in G_\wp(\wh R_\wp)$.
So then $g_Vg_{\wp}g_P = g_{\wp}'g_P\in G_{\wp}(\wh R_{\wp})$, as desired.

We next consider the case that $P \in \PP^{\text{sing}}$.  Then $P$ is the parent of a unique $U \in \UU$, since $P$ is not a leaf and since $\XX$ is a normal crossings model.  Here $U$ has distance $2n+2$ from $v_0$; and every vertex of distance $2n+2$ arises in this manner for a unique $P \in \PP^{\text{sing}}$ (since each vertex other than $v_0$ in the tree $\Gamma$ has a unique parent).  Let $V \in \UU$ be the unique parent of $P$; thus $V$ has distance $2n$, and $g_V$ is already defined by the inductive hypothesis.  Let $\wp_1$ be the branch at $P$ lying on $V$, and let $\wp_2$ be the branch at $P$ lying on $U$.

Let $W = U \cup \{P\}$. If needed, we delete a regular point from $U \subset W$, so that we may assume that $W$ is affine.  In this case we add to $\PP$ the point deleted from $W$, thereby refining $\PP, \UU, \BB$; and we set $g_{\wp}=1$ for $\wp$ the (unique) branch at that deleted point.  Note that this new point added to $\PP$ has distance $2n+3$ from $v_0$ in the enlarged graph.
Since $W$ is a non-empty affine open subset of an irreducible component of $\XX_k^{\rm red}$, there is an associated ring $\wh R_W$, as in the discussion before Lemma~\ref{adjust over U}.

Set $g_1:=g_Vg_{\wp_1} \in G(F_{\wp_1})$ and $g_2:=g_{\wp_2} \in G(F_{\wp_2})$, where the branches $\wp_1$ and $\wp_2$ at $P$ are as above.
We wish to define $g_P \in G(F_P)$ and $g_U \in G(F_U)$ such that
$g_Vg_{\wp_1}g_P \in G_{\wp_1}(\wh R_{\wp_1})$ and
$g_Ug_{\wp_2}g_P \in G_{\wp_2}(\wh R_{\wp_2})$.

Let $\bar S$ be maximal among the split tori in $G_{\kappa(P)} = G \times_R \kappa(P)$.  Since $P$ is the parent of $U$ in a monotonic tree, $\kappa(P)=\kappa(U)=\kappa(W)$.  Here $\kappa(W)$ is contained in the coordinate ring $k[W]$ of $W$.  
So $\bar S$
defines a torus $\bar S_W$ in $G_W := G \times_R k[W]$ via base change, and $\bar S_W$ remains split.  
By Proposition~\ref{split tori lift}(\ref{split lifting}) in the Appendix, we may therefore lift $\bar S_W$ to a split torus $S$ over $\wh R_W$. 
Here $S$ is also
split over $\wh R_P$ since $\wh R_W \subset \wh R_P$; and
$S_{\kappa(P)}=\bar S$.
Consider the group $H=C_G(S)$ over $\wh R_W$, and write $\bar H = H_{\kappa(P)}$.  Then $(H/S)_{\kappa(P)} = \bar H/\bar S$ is anisotropic over $\kappa(P)$ by applying Proposition~\ref{H and S}(\ref{aniso quotient}) over $\kappa(P)$,
since $\bar S$ is a maximal split torus in $G_{\kappa(P)}$ and since 
$\bar H = C_{G_{\kappa(P)}}(\bar S)$.
By Proposition~\ref{local anisotropy}, $H/S$ is therefore anisotropic over $F_P$ and $\wh R_P$, and thus over $\wh R_W$.

By Proposition~\ref{branch factorization}, there exist $g\in G(F_P)$ and $h_i \in  H(F_{\wp_i})$ for $i=1,2$ such that $h_ig_ig
\in G_{\wp_i}(\wh R_{\wp_i})$ for $i=1,2$.
Let $\bar h_i$ denote the image of $h_i$ in $(H/S)(F_{\wp_i})$. By Lemma~\ref{adjust over U} applied to the anisotropic reductive group $H/S$, there exists $\bar h_U \in (H/S)(\wh R_U)$ so that $\bar h_1$ and $\bar h_2 \bar h_U$
map to the same element in $(H/S)(\kappa(P))$. By Lemma~\ref{singular factor}, there is an element $\bar h_P\in (H/S)(\wh R_P)$ so that $\bar h_{\pi_1} := \bar h_P \bar h_1\in (H/S)_{\wp_1}(\wh R_{\wp_1})$
and $\bar h_{\pi_2} := \bar h_P \bar h_2 \bar h_U
\in (H/S)_{\wp_2}(\wh R_{\wp_2})$.
Since $\wh R_P$ is local and $S$ is a split torus, $H^1(\wh R_P,S)=0$ by Hilbert 90. So the cohomology exact sequence associated to $1\rightarrow S\rightarrow H\rightarrow H/S\rightarrow 1$ asserts that $H(\wh R_P)$ surjects onto $(H/S)(\wh R_P)$, and we may lift $\bar h_P$ to an element $h_P\in H(\wh R_P)$. For analogous reasons, we can lift $\bar h_{\pi_i}$ to some $h_{\pi_i}\in H_{\wp_i}(\wh R_{\wp_i})$.

We claim that after refining $\PP,\UU,\BB$ by deleting some finite collection of points from $U$ (and hence also from $W$), the element $\bar h_U$ may be lifted to some $h_U \in H(\wh R_U)$. To see this, consider the image $\xi$ of $\bar h_U$ under the coboundary
map $(H/S)(\wh R_U) \to H^1(\wh R_U,S)$.  Let $R_\eta$ be the local ring of $\XX$ at the generic point $\eta$ of $U$. Then $\xi$ has trivial image in $H^1(R_\eta,S)$, since $R_\eta$ is local and $S$ is a split torus; so there is a Zariski affine open neighborhood $\ms N = \Spec(E)$ of $\eta$ in $\XX$ such that $\xi$ maps to the trivial element in $H^1(E,S)$. Now $U' := \ms N \cap U$ is a dense open subset of $U$, and $E \subset R_{U'} \subset \wh R_{U'}$.  Thus $\bar h_U$, or equivalently its image $\bar h_{U'} \in (H/S)(\wh R_{U'})$, maps to the trivial element in $H^1(\wh R_U',S)$; hence
$\bar h_{U'}$ can be lifted to some $h_{U'} \in H(\wh R_{U'})$, as claimed.
We now proceed, replacing $U$ by $U'$ (and $W$ by $U' \cup \{P\}$).  We add the deleted points to $\PP$, and set $g_{\wp}=1$ for $\wp$ the corresponding branches.  (As before, these new points of $\PP$ have distance $2n+3$ from $v_0$, since $U'$ has distance $2n+2$.)

Since the two elements $h_1,h_P^{-1}h_{\pi_1} \in H(\wh R_{\wp_1})$ have the same image in $(H/S)(\wh R_{\wp_1})$, there exists $s_1\in S(\wh R_{\wp_1})$ such that $h_1=s_1h_P^{-1}h_{\pi_1}$.  Similarly, there exists $s_2 \in S(\wh R_{\wp_2})$ such that
$h_2 = s_2h_P^{-1}h_{\pi_2}h_U^{-1}$.
Hence $s_1h_P^{-1}h_{\pi_1}g_1g = h_1g_1g\in G_{\wp_1}(\wh R_{\wp_1})$ and
$s_2h_P^{-1}h_{\pi_2}h_U^{-1}g_2g = h_2g_2g \in G_{\wp_2}(\wh R_{\wp_2})$.
Since the split torus $S$ is rational, weak approximation yields elements $s_P\in S(\wh R_P)$ and $s_{\pi_i}\in S_{\wp_i}(\wh R_{\wp_i})$ for $i=1,2$ such that $s_i=s_{\pi_i}s_P$ for $i=1,2$.
Thus $G_{\wp_1}(\wh R_{\wp_1})$ contains $s_{\pi_1}s_Ph_P^{-1}h_{\pi_1}g_1g$, hence also
$s_Ph_P^{-1}h_{\pi_1}g_1g$.  Since $G(\wh R_P)$ normalizes $G_{\wp_1}(\wh R_{\wp_1})$, the group $G_{\wp_1}(\wh R_{\wp_1})$ also contains $h_{\pi_1}g_1gs_Ph_P^{-1}$ and thus $g_1gs_Ph_P^{-1}$.
Let $g_P = gs_Ph_P^{-1} \in G(F_P)$.  Thus
$g_Vg_{\wp_1}g_P = g_1g_P \in G_{\wp_1}(\wh R_{\wp_1})$, as desired.
Similarly, setting $g_U=h_U^{-1}$, we find that $g_Ug_{\wp_2}g_P
= h_U^{-1}g_2gs_Ph_P^{-1} \in G_{\wp_2}(\wh R_{\wp_2})$, also as desired.
This completes the inductive step for points of $\PP$ of distance $2n+1$ and elements of $\UU$ of distance $2n+2$, and thus completes the proof.
\end{proof}

Combined with the results in Section~\ref{cosets}, and in the notation of Section~\ref{lgp sgp}, the above proposition gives the following.

\begin{thm} \label{torsors extend}
Let $\XX$ be a normal crossings model of a semi-global field $F$ over a discrete valuation ring $R$, and assume that the reduction graph is a monotonic tree.
Let $G$ be a reductive group over $R$, and let $Z$ be a $G$-torsor over $F$ whose class is in $\Sha(F,G)$. Then $Z$ extends to a torsor over $\XX$ which is trivial when restricted to the reduced closed fiber ${\XX}_k^{\red}$.
\end{thm}

\begin{proof}
Consider an element $\xi$ in $\Sha(F,G)$.  Consider a non-empty collection of closed points $\PP$ on a normal crossings model $\XX$ of $F$, with $\PP$ containing all the singular points of the reduced closed fiber $\XX_k^{\red}$.  By Theorem~\ref{equalshas}, $\xi$ lies in $\Sha_{\PP}(F,G)$. By the double coset description of $\Sha_{\PP}(F,G)$ given in Corollary~3.6 of \cite{HHK15}, the cocycle $\xi$ is represented by a collection of elements $g_{\wp} \in G(F_{\wp})$, where $\wp$ runs over the branches at elements of $\PP$. By Proposition~\ref{anisotropic factorization}, this collection gives rise to a collection of elements $g_{\wp}'\in G_{\wp}(\wh R_{\wp})$ representing the same double coset class, and hence the same torsor over $F$, if $\PP$ is sufficiently large (i.e., possibly after refining the initial choice of $\PP$).
By Corollary~\ref{torsors closed fiber}(\ref{torsors triv cl fib}), this new collection
defines a torsor over $\XX$ whose reduction to the closed fiber is trivial.
\end{proof}

\begin{rem}
\begin{enumerate}
\renewcommand{\theenumi}{\alph{enumi}}
\renewcommand{\labelenumi}{(\alph{enumi})}
\item
It was already shown in \cite[Theorem~4.2]{CTPS12} that every $G$-torsor over $F$ arises from a $G$-torsor over $\XX$, even without assuming that the class of the given torsor lies in $\Sha_{\PP}(F,G)$ or assuming that the reduction graph is a monotonic tree.  But the key new point of Theorem~\ref{torsors extend} is that under the hypotheses stated there, the torsor over $\XX$ may be chosen so as to be trivial along the reduced closed fiber.  This will be essential in the next section. 
\item
In the above proof, Proposition~\ref{sha containments}(\ref{ShaX union}) could have been cited instead of Theorem~\ref{equalshas}, since the set $\PP$ is allowed to grow in the proof.  But the full strength of Theorem~\ref{equalshas} will be needed in the proof of Proposition~\ref{dbl coset surj}, and hence in the results in Sections~\ref{Sha computation section} and~\ref{sect: counterex} that build on that.
\end{enumerate}
\end{rem}

\section{Local-global principles for monotonic trees} \label{lgp mon tree sec}

In this section, building on Theorem~\ref{torsors extend}, we prove local-global principles for torsors in the context of monotonic trees.  First, in Theorem~\ref{main}, we obtain such a result for reductive groups over a complete discrete valuation ring $R$, extending \cite[Theorem~7.3]{CHHKPS}, which had considered the case of tori.
Afterwards we prove related results for groups that are not necessarily connected or reductive, especially in the case where $R$ is of equicharacteristic zero; i.e., where $R = k[[t]]$ for some field $k$ and parameter~$t$.

The proof of Theorem~\ref{main} below relies on \cite{GPS}, and to apply that we need additional mild assumptions.  Given a reductive group $G$ over a scheme $S$, there is a canonical semisimple normal subgroup $G^{\rm ss}$, called the derived subgroup of $G$, such that the quotient $G/G^{\rm ss}$ is a torus (see \cite[Exp.~XXII, Th\'eor\`eme~6.2.1, Remarque~6.2.2]{SGA3.3}).
The semisimple group $G^{\rm ss}$ has a simply connected cover $G^{\rm sc} \to G^{\rm ss}$; this is the universal central isogeny to $G^{\rm ss}$.  (See \cite[beginning of Section~1.2]{harder}.)  The kernel of this isogeny is a finite commutative $S$-group scheme $\mu(G^{\rm ss})$ of multiplicative type, called the {\em algebraic fundamental group} of $G^{\rm ss}$ (e.g., see \cite[Remark~2.3]{gonzalez}); this agrees with the \'etale fundamental group (as in \cite{SGA1}) over a field of characteristic zero.  More generally, write
\[\mu(G) := \ker(G^{\rm sc} \times_S {\rm rad}(G) \to G),\]
where ${\rm rad}(G)$ is the radical of the reductive group $G$ (this is trivial if $G$ is semisimple).  The group $\mu(G)$ is an extension of $\ker(G^{\rm ss} \times_S {\rm rad}(G) \to G)$ by $\ker(G^{\rm sc} \to G^{\rm ss})$; and so $\mu(G)$, like the latter two groups, is a finite $S$-group scheme of multiplicative type (see \cite[Exp.~XXII, 6.2.3]{SGA3.3}).

Now let $R$ be a complete discrete valuation ring having residue field $k$, let $S = \Spec(R)$, and let $G$ be a reductive group over $R$ (or equivalently, over $S$).
The group $\mu(G)$ is \'etale provided $\mu(G_k) = \mu(G)_k$ is; and that condition is automatic if $\cha(k)=0$ since then every $k$-group scheme is smooth.  More generally, since $\mu(G)$ is a finite central subgroup of $G^{\rm sc}$, it is \'etale if $\cha(k)$ does not divide the order of the center of $G^{\rm sc}$.  As the list on \cite[page~332]{PlaRap} shows, this includes the cases where 
$G^{\rm ss}$ is of type $B_n$, $C_n$, or $D_n$, with $\cha(k)\ne 2$;
or of type ${}^1A_n$ or ${}^2A_n$ if $\cha(k)$ does not divide $n+1$. 

The following proposition is a consequence of the main theorem (Theorem~1.1) of the paper~\cite{GPS} by Gille, Parimala and Suresh.

\begin{prop} \label{GPS prop}
Let $F$ be a semi-global field over a complete discrete valuation ring $R$ with residue field $k$, let $\XX$ be a regular model of $F$ over $R$,
and let $G$ be a reductive group over $\XX$ such that $\mu(G)$ is \'etale.
If an element $\zeta \in H^1(\XX, G)$ has trivial image in $H^1(\XX_k, G)$, then its image in $H^1(F, G)$ is also trivial.  More generally, if two elements
$\zeta, \zeta' \in H^1(\XX, G)$ have equal images in $H^1(\XX_k, G)$, then their images in $H^1(F, G)$ are also equal.
\end{prop}

\begin{proof}  Let $T = {\rm rad}(G)$ and write $\mu=\mu(G)$.  As in the proof of  \cite[Theorem~7.1]{GPS},
there is a commutative diagram
\[\xymatrix{
H^1(\XX,\mu) \ar[r] \ar[d] & H^1(\XX,G^{\rm sc}) \times H^1(\XX,T) \ar[r] \ar[d]^\alpha & H^1(\XX,G)  \ar[r] \ar[d]^\beta & H^2(\XX,\mu) \ar[d]  \\
H^1(\XX_k,\mu) \ar[r]  & H^1(\XX_k,G^{\rm sc}) \times H^1(\XX_k,T) \ar[r] & H^1(\XX_k,G)  \ar[r] & H^2(\XX_k,\mu)
}\]
with exact rows.  
Here the left and right vertical arrows are bijective by proper base change (see \cite[Exp.~XII, Corollaire~5.5]{SGA4}).  A
diagram chase then implies that the natural map $\ker(\alpha) \to \ker(\beta)$
is surjective.  Here $\ker(\alpha) = \ker(\alpha_1) \times \ker(\alpha_2)$, where
$\alpha_1:  H^1(\XX,G^{\rm sc}) \to H^1(\XX_k,G^{\rm sc})$ and
$\alpha_2: H^1(\XX,T) \to H^1(\XX_k,T)$ are the natural maps.
Every element of $\ker(\alpha_1)$ has trivial image in
$H^1(F,G^{\rm sc})$ by \cite[Theorem~1.1]{GPS}.  Moreover, every element of $\ker(\alpha_2)$ is contained in the kernel of $H^1(\XX,T) \to H^1(\XX_k^{\rm red},T)$ and thus
has trivial image in $H^1(F,T)$ by Proposition~\ref{GPS tori}.  So by the surjectivity of $\ker(\alpha) \to \ker(\beta)$, it follows that every element in $\ker(\beta)$ has trivial image in $H^1(F,G)$.  This proves the first assertion.

For the second assertion, recall that given a scheme $\ZZ$, a group scheme $G$ over $\ZZ$, and a cocycle $\tau \in Z^1(\ZZ,G)$, there is an associated twist $G^\tau$ of $G$ by $\tau$, together with a functorial bijection between $H^1(\ZZ,G)$ and $H^1(\ZZ,G^\tau)$, such that the neutral element of $H^1(\ZZ,G^\tau)$ corresponds to the class of $\tau$ in $H^1(\ZZ,G)$.  (See \cite[Section~I.5.3]{Serre:CG}, especially \cite[Proposition~I.5.3.35]{Serre:CG}, for the case 
that $\ZZ$ is the spectrum of a field.  This was generalized in \cite[Chapitre~III, Sections~2.3,~2.6]{Giraud}; see in particular \cite[Chapitre~III, Remarque 2.6.3]{Giraud}.)
Thus if $\zeta, \zeta' \in H^1(\XX,G)$ have the same image in $H^1(\XX_k, G)$, and if we pick an element $\tau \in Z^1(\XX,G)$ in the class of $\zeta$ with image $\tau_k \in  Z^1(\XX_k,G)$, 
then the element of $H^1(\XX,G^\tau)$ corresponding to $\zeta'$ is in the kernel of the map 
$H^1(\XX, G^\tau) \to H^1(\XX_k, G^{\tau_k})$.  So by the first assertion, that element is the neutral class of $H^1(\XX, G^\tau)$, and thus $\zeta'=\zeta$.
\end{proof}

The next result is an analogue of Theorem~7.3 of \cite{CHHKPS}, where the group was required to be a torus.   Recall the definition of a ``monotonic tree'' from the beginning of Subsection~\ref{subsec mono}, and also recall that whether this property holds for the reduction graph of a normal crossings model of a semi-global field $F$ is independent of the choice of the model $\XX$ and finite set $\PP$ (Remark~\ref{tree independent}).

\begin{thm}\label{main}
Let $\XX$ be a normal crossings model of a semi-global field $F$ over a complete discrete valuation ring $R$.  Assume that the closed fiber of $\XX$ is reduced and that the associated reduction graph $\Gamma$ is a monotonic tree. If $G$ is any reductive group over $R$ such that $\mu(G)$ is \'etale, then
$\Sha(F,G)=1$.
\end{thm}

\begin{proof}
Since $\Gamma$ is a monotonic tree, 
Theorem~\ref{torsors extend} applies.  Thus each element $\xi$ in $\Sha(F,G)$ corresponds to a torsor over $F$ that extends to $\XX$ and is trivial over the reduced closed fiber; and this is the same as the closed fiber since the closed fiber is reduced.  Since $\mu(G)$ is \'etale,
we can apply Proposition~\ref{GPS prop}, and conclude that $\xi$ is isomorphic to the trivial $G$-torsor over~$F$.  That is, an arbitrary element $\xi \in \Sha(F,G)$ is trivial.
\end{proof}

In the special case of an equicharacteristic zero ring $R=k[[t]]$ and a group $G$ over $k$,
the reductivity hypothesis can be dropped (and $\mu(G)$ is automatically \'etale):

\begin{cor} \label{cor:char0conn}
Let $G$ be a connected linear algebraic group over a field $k$ of characteristic zero, and  let $F$ be a semi-global field over $R:=k[[t]]$. Suppose that $F$ has a normal crossings model over $R:=k[[t]]$ whose
closed fiber is reduced and whose
reduction graph  is a monotonic tree.  Then $\Sha(F,G)$ is trivial.
\end{cor}

\begin{proof}
Let ${\mathcal R}_{\mathrm u}(G)$ be the unipotent radical of $G$.  Thus the group $G/{\mathcal R}_{\mathrm u}(G)$ is reductive over $k$.  The exactness of
$1 \to {\mathcal R}_{\mathrm u}(G) \to G \to G/{\mathcal R}_{\mathrm u}(G) \to 1$ yields that
$H^1(F,{\mathcal R}_{\mathrm u}(G)) \to H^1(F,G) \to H^1(F,G/{\mathcal R}_{\mathrm u}(G))$ is exact.  But $H^1(F,{\mathcal R}_{\mathrm u}(G))$ is trivial, because
${\mathcal R}_{\mathrm u}(G)$ is unipotent and $F$ is perfect (see \cite[III \S2.1, Prop.~6]{Serre:CG}).
So the map $H^1(F,G) \to H^1(F,G/{\mathcal R}_{\mathrm u}(G))$ has trivial kernel.  But this map sends
$\Sha(F,G)$ to $\Sha(F,G/{\mathcal R}_{\mathrm u}(G))$, and the latter is trivial by Theorem~\ref{main}, using that $\mu(G)$ is \'etale in residue characteristic zero.  It thus follows that $\Sha(F,G)$ is also trivial.
\end{proof}

So far, we have been focusing on connected groups, but new issues arise for local-global principles for groups that are disconnected.  For example, in \cite{HHK15}, it was shown that if a (smooth) linear algebraic group $G$ over a semi-global field $F$ is connected and rational as an $F$-variety then $\Sha_X(F,G)$ is trivial, where $X$ is the reduced closed fiber; whereas if $G$ is disconnected and each connected component of the $F$-variety $G$ is rational, then the vanishing of $\Sha_X(F,G)$ depends on the reduction graph.  
Below, in Theorems~\ref{char0 triv Sha surj} and~\ref{char0disconn}, we prove analogues of Corollary~\ref{cor:char0conn} in the disconnected case.

In preparation, we consider the finite case, in Lemma~\ref{la:tree} and Proposition~\ref{tree_fin_grp}.  Namely, let $G$ be a finite \'etale group scheme over a semi-global field $F$ over a complete discretely valued field $K$ with valuation ring $R$.  Thus $G_{F^\sep}$ is a finite constant group over the separable closure $F^\sep$ of $F$, equipped with an action of $\Gal(F^\sep/F)$.  There is then
a finite Galois extension $L/F$ that splits $G$; i.e., such that $G_L$ is a constant finite group over $L$.
(In fact, the minimal such finite extension is the fixed field $L_0 := (F^\sep)^N$ under the kernel $N$ of the Galois action on $G_{F^\sep}$.)

Pick such a field $L$ and let $\Delta = \Gal(L/F)$.  Take a normal crossings model $\XX$ for $F$ over $R$, and let $\ms Y$ be the normalization of $\XX$ in $L$.  Here $\ms Y$ need not be regular; but
$\pi:\ms Y \to \XX$ is a $\Delta$-Galois branched cover of normal $R$-curves, with reduced closed fibers $\ms Y_k^{\rm red}$ and $\XX_k^{\rm red}$.  If a closed point $Q \in \ms Y_k^{\rm red}$ lies over $P \in \XX_k^{\rm red}$, then the set of branches of $\ms Y_k^{\rm red}$ at $Q$ surjects onto the set of branches of $\XX_k^{\rm red}$ at $P$, by the going-down theorem applied to the extension $\wh R_P \subseteq \wh R_Q$.
Let $\PP$ be a non-empty finite set of closed points on $\XX_k^{\rm red}$, and assume that
\[ \PP' := \pi^{-1}(\PP) \quad \text{contains the set of points where $\ms Y_k^{\rm red}$ is not regular}. \eqno{(*)} \]
 Thus $\PP'$ is a non-empty $\Delta$-stable finite set of closed points of $\ms Y_k^{\rm red}$.  Note that $\PP$ contains all the non-unibranched points of $\XX_k^{\rm red}$, by the above surjectivity on branches, since $\PP'$ contains all the non-unibranched points of $\ms Y_k^{\rm red}$.
Consider the set $\UU$ of connected components of the complement of $\PP$ in $\XX_k^{\rm red}$, and similarly the set $\UU'$ with respect to $\PP'$ and $\ms Y_k^{\rm red}$.

Let $\Gamma_F$ and $\Gamma_L$ be the reduction graphs associated to these models and sets, as in Section~\ref{patching subsec} (where there is no regularity requirement).  Up to homotopy equivalence, these are determined by $F$ and $L$; i.e., they do not depend on the choice of $\PP$ and $\PP'$.  (See \cite[Remark 6.1(b)]{HHK15} and the sentence
just before Corollary~6.5 in that paper.  Note that these also do not require $\ms Y$ to be a regular model.)  Observe that $\Delta$ acts on $\Gamma_L$, and 
$\Gamma_F = \Gamma_L/\Delta$.
This is because the action of $\Delta$ on $\ms Y$ over $\XX$ induces
actions of $\Delta$ on $\PP'$ and on $\UU'$, with the elements of $\PP$ and $\UU$ being the orbits of this action (because $\XX_k^{\rm red} = \ms Y_k^{\rm red}/\Delta$).

Let $\mc V(\Gamma)$ denote the set of vertices of a graph $\Gamma$.  For $\xi \in \mc V(\Gamma_F)$, let $\mc V_\xi(\Gamma_L) \subseteq \mc V(\Gamma_L)$ be the set of vertices that map to $\xi$.
The tensor product $F_\xi \otimes_F L$ is a $\Delta$-Galois \'etale $F_\xi$-algebra consisting of a direct
product of the fields $L_\zeta$ (ranging over $\zeta \in \mc V_\xi(\Gamma_L)$).  So
if $\zeta \in \mc V_\xi(\Gamma_L)$, then $L_\zeta$ is a finite Galois extension of $F_\xi$, and its Galois group $\Delta_\zeta$ is a subgroup of $\Delta$.
More precisely, $\Delta_\zeta \subseteq \Delta$ is the stabilizer of the factor $L_\zeta$ in the direct product
(or equivalently, the stabilizer of $\zeta$) under the action of $\Delta$.

\begin{lem} \label{la:tree}
As above, let $G$ be a finite \'etale group scheme over a semi-global field $F$ and let $L$ be a finite Galois extension of $F$ that splits $G$.  Let $\XX$ be a normal crossings model for~$F$, let $\ms Y$ be the normalization of $\XX$ in $L$, and let $\Gamma_F, \Gamma_L$ be the reduction graphs for $\XX, \ms Y$ associated to finite subsets $\PP, \PP'$ of their reduced closed fibers such that $\PP'$ is the inverse image of $\PP$.  If
$\Gamma_L$ is a tree, then
\[\Sha_{\PP}(F,G) = \ker\bigl(H^1(L/F,G(L)) \to \prod_{\substack{\xi \in \mc V(\Gamma_F) \\ \zeta \in \mc V_\xi(\Gamma_L)}} H^1(L_\zeta/F_\xi,G(L_\zeta))\bigr).\]
\end{lem}

\begin{proof}
Consider the commutative diagram
\[
\xymatrix{
\Sha_{\PP}(F,G) \ar[r] \ar[d]
& H^1(F,G) \ar[d]
\\
\Sha_{\PP'}(L,G_L) \ar[r]
& H^1(L,G_L),
}
\]
where the horizontal arrows are inclusions.  Since $\Gamma_L$ is a tree
and since $G_L$ is a constant finite group, it follows from \cite[Corollary~6.5]{HHK15} that
$\Sha_{\PP'}(L,G_L)$ is trivial.  Thus $\Sha_{\PP}(F,G)$ is contained in
$H^1(L/F,G) = \ker(H^1(F,G) \to H^1(L,G_L))$. Hence
\begin{eqnarray*}
\Sha_{\PP}(F,G) &=& \ker\bigl(H^1(F,G(F^\sep)) \to \prod_{\xi \in \mc V(\Gamma_F)} H^1(F_\xi,G(F_\xi^\sep)\bigr)\\
&=& \ker\bigl(H^1(L/F,G(L)) \to \prod_{\xi \in \mc V(\Gamma_F)} H^1(F_\xi,G(F_\xi^\sep)\bigr)\\
&=& \ker\bigl(H^1(L/F,G(L)) \to \prod_{\substack{\xi \in \mc V(\Gamma_F) \\ \zeta \in \mc V_\xi(\Gamma_L)}} H^1(L_\zeta/F_\xi,G(L_\zeta)\bigr),
\end{eqnarray*}
as asserted.
\end{proof}

We then obtain the following sufficient criterion for the vanishing of $\Sha(F,G)$ if $G$ is a finite \'etale group scheme over a semi-global field $F$:

\begin{prop} \label{tree_fin_grp}
Let $F$ be a semi-global field over a complete discretely valued field with valuation ring $R$, and let
$\XX$ be a normal crossings model of $F$ over $R$.  Let
$G$ be a finite \'etale group scheme over $F$, and let $L$ be a finite Galois extension of $F$ such that $G_L$ is a constant finite group.  Let $\ms Y$ be the normalization of $\XX$ in $L$, and consider its associated reduction graph $\Gamma_L$.  If $\Gamma_L$ is a tree, then $\Sha(F,G)$ is trivial.
\end{prop}

\begin{proof}
Let $X = \XX_k^{\red}$.
By Proposition~\ref{sha containments}, 
$\Sha(F,G) = \Sha_X(F,G) = \bigcup \Sha_{\PP}(F,G)$, where the right hand side is an increasing union over the non-empty finite sets $\PP$ of closed points of $X$ that contain all the singular points of $X$. 
So it suffices to show that $\Sha_{\PP}(F,G)$ is trivial for all such sets $\PP$ that are sufficiently large; in particular, we may now restrict attention to sets $\PP$ that satisfy the condition $(*)$ in the discussion before Lemma~\ref{la:tree}.

As above, let $\Delta = \Gal(L/F)$ and $\Delta_\zeta = \Gal(L_\zeta/F_\xi)$ for $\zeta \in \mc V(\Gamma_L)$ lying over $\xi \in \mc V(\Gamma_F)$.  For $\PP$ as above, Lemma~\ref{la:tree} 
applies since $\Gamma_L$ is a tree, and yields
\[\Sha_{\PP}(F,G) = \ker\bigl(H^1(\Delta,G(L)) \to \prod_{\substack{\xi \in \mc V(\Gamma_F) \\ \zeta \in \mc V_\xi(\Gamma_L)}} H^1(\Delta_\zeta,G(L_\zeta))\bigr),\]
where each map $H^1(\Delta,G(L)) \to H^1(\Delta_\zeta,G(L_\zeta))$ is induced by restriction.

As noted before Lemma~\ref{la:tree}, the action of $\Delta$ on $\ms Y$ stabilizes $\PP'$ and thus also the associated set $\UU'$.  Hence it acts on $\Gamma_L$ without inversion, in the sense of \cite[Section~I.3.1]{Serre:Trees}; i.e.,
no element of $\Delta$ can interchange two adjacent vertices of the bipartite tree $\Gamma_L$.  But by \cite[Theorem~I.6.1.15]{Serre:Trees}, any action of a finite group on a tree without inversion has a global fixed point; i.e., there is a vertex that is fixed by the entire group.  So there exists
$\zeta \in \mc V(\Gamma_L)$ such that $\Delta_\zeta=\Delta$.  
But $G_L$ is a constant finite group, and so the inclusion $G(L) \to G(L_\zeta)$ is an equality; hence 
$H^1(\Delta,G(L)) = H^1(\Delta_\zeta,G(L_\zeta))$.
Thus  $\Sha_{\PP}(F,G)$ is trivial, as required.
\end{proof}

\begin{cor} \label{finite monotonic}
Let $R,F,\XX$ be as in Proposition~\ref{tree_fin_grp}, and let $G$ be a finite \'etale group scheme over $R$.  Suppose that the reduction graph associated to $\XX$ is a monotonic tree.   Then $\Sha(F,G)$ is trivial.
\end{cor}

\begin{proof}
By smoothness, there is
an \'etale algebra $S/R$ that splits $G$.  Let $\ell$ be the residue field of $S$, and let $L = F\otimes_R S$.  Thus $G_\ell$ and $G_L$ are finite constant groups.  Also, $\XX_S := \XX \times_R S$ is finite \'etale over $\XX$; it is the normalization of $\XX$ in $L$; and it is a normal crossings model of its function field $L$ over $S$.
The reduced closed fiber $(\XX_S)_\ell^{\rm red}$ of $\XX_S$ is the base change $X_\ell := X \times_k \ell$ of the reduced closed fiber $X$ of $\XX$ (with $k$ the residue field of $R$); and $X_\ell \to X$ is \'etale since $S/R$ is.  The reduction graph of $\XX_S$, which is the graph associated to $X_\ell$, is a tree because the reduction graph of $\XX$ is a monotonic tree (see \cite[Proposition~7.6(a)]{CHHKPS}).
By Proposition~\ref{tree_fin_grp}, it follows that $\Sha(F,G)$ is trivial (where we view $G$ as a finite \'etale group scheme over $F$ here).  
\end{proof}

\begin{rem}
The conclusion of Corollary~\ref{finite monotonic} can fail if the reduction graph is not a monotonic tree.  
For example, suppose the closed fiber of $\XX$ consists of two copies of $\P^1_k$ that meet transversally at a single point $P$ having a bigger residue field $k'$.  (Such a model is given in Example~\ref{nonmono tree}.)  Suppose that $G$ is a finite \'etale group scheme over $R$ such that 
$G(k)$ is trivial but $G(k')$ is non-trivial.  Take $\PP = \{P\}$, so that we then have $\UU = \{U_1,U_2\}$ and $\BB = \{\wp_1,\wp_2\}$, where $U_1,U_2$ are affine open subsets of projective $k$-lines, and $\wp_1,\wp_2$ are the branches at $P$.  Since $G$ is a finite \'etale group scheme, 
$G(A)=G(L)$ for an $R$-algebra $A$ that is a domain with fraction field $L$; and $G(A) = G(A/I)$
if $A$ is $I$-adically complete.  In particular, $G(F_P) = G(\wh R_P) = G(\kappa(P)) = G(k')$, and $G(F_{U_i}) = G(\wh R_{U_i}) = G(k[U_i]) = G(k(x)) = G(k)$.  Similarly, $G(F_{\wp_i}) = G(\wh R_{\wp_i}) = G(\wh R_{\wp_i}/\wp_i) = G(\wh R_P/\wp_i) = G(\kappa(P)) = G(k')$.  
By the double coset description of $\Sha_\PP(F,G) \subseteq \Sha(F,G)$ (Equation~(\ref{Sha_P dbl coset}) of Section~\ref{lgp sgp}), we have that $\Sha(F,G) \supseteq \Sha_\PP(F,G) = (G(k')\times G(k'))/G(k') \ne 1$.
\end{rem}

By combining Corollary~\ref{cor:char0conn} with Proposition~\ref{tree_fin_grp}, we will obtain Theorem~\ref{char0 triv Sha surj} below, concerning local-global principles for groups that need not be connected.  First we prove a lemma that is analogous to \cite[Corollary~2.6]{HHK15}, though that result applied to $\Sha_\PP(F,G)$ rather than to the situation of discrete valuations.

\begin{lem} \label{disconn trivial Sha lem}
Let $F$ be a field equipped with a set $\Omega$ of discrete valuations.  Let $G$ be a linear algebraic group over $F$, with identity component $G^0$, and write $\bar G = G/G^0$.  Suppose that the associated obstruction sets $\Sha_\Omega(F,G^0)$ and $\Sha_\Omega(F,\bar G)$ are both trivial.  If $G(F_v) \to \bar G(F_v)$ is surjective for all discrete valuations $v \in \Omega$, then $\Sha_\Omega(F,G)$ is trivial.
\end{lem}

\begin{proof}
We have the following commutative diagram with exact rows and columns:
{\small
\[\xymatrix{
& & 1 \ar[d] & 1 \ar[d] & 1 \ar[d]\\
& & \Sha_\Omega(F,G^0) \ar[d] & \Sha_\Omega(F,G) \ar[d] & \Sha_\Omega(F,\bar G) \ar[d] \\
& & H^1(F,G^0)  \ar[r] \ar[d] & H^1(F,G)  \ar[r] \ar[d] & H^1(F,\bar G) \ar[d] \\
\prod_v G(F_v) \ar[r] & \prod_v \bar G(F_v) \ar[r] & \prod_v H^1(F_v,G^0)  \ar[r] & \prod_v H^1(F_v,G)  \ar[r] & \prod_v H^1(F_v,\bar G)
}\]}
\ \ The assertion now follows by a diagram chase.
Namely, if $\xi \in \Sha_\Omega(F,G) \subseteq H^1(F,G)$,
then the commutativity of the right hand square implies that the image of $\xi$ in $H^1(F,\bar G)$ lies in $\Sha_\Omega(F,\bar G)$ and hence is trivial.  Thus $\xi$ is the image of some $\xi^0 \in H^1(F,G^0)$.  Since $\xi \in \Sha_\Omega(F,G)$, the image $(\xi_v^0)_v$ of $\xi^0$ in $\prod_v H^1(F_v,G^0)$ is in the kernel of the map to $\prod_v H^1(F_v,G)$; hence it is the image of an element of $\prod_v \bar G(F_v)$.  By the surjectivity hypothesis, it follows that $(\xi_v^0)_v$ is trivial.  Hence $\xi^0 \in \Sha_\Omega(F,G^0)$, which is trivial.  Thus $\xi$ is trivial.
\end{proof}

\begin{thm} \label{char0 triv Sha surj}
Let $F$ be a semi-global field over a complete discrete valuation ring $R$, with a normal crossings model $\XX$\!\!.  Assume that the closed fiber of $\XX$ is reduced and the reduction graph is a monotonic tree.  Suppose that either
\begin{enumerate}
\renewcommand{\theenumi}{\roman{enumi}}
\renewcommand{\labelenumi}{(\roman{enumi})}
\item \label{reduc comp hyp}
$G$ is a smooth affine group scheme over $R$ such that $G^0$ is reductive and $\mu(G^0)$ is \'etale; or
\item \label{eq char 0 hyp}
$R=k[[t]]$ for some field $k$ of characteristic zero, and $G$ is a linear algebraic group over $k$.
\end{enumerate}
Let $\bar G = G/G^0$.  We then have the following conclusions:
\begin{enumerate}
\renewcommand{\theenumi}{\alph{enumi}}
\renewcommand{\labelenumi}{(\alph{enumi})}
\item \label{outer Shas trivial}
Both $\Sha(F,G^0)$ and $\Sha(F,\bar G)$ are trivial.
\item \label{char0 disconn Sha}
If $G(F_v) \to \bar G(F_v)$ is surjective for all divisorial discrete valuations $v$, then $\Sha(F,G)$ is trivial.
\end{enumerate}
\end{thm}

\begin{proof}
By Corollary~\ref{finite monotonic}, 
$\Sha(F,\bar G)$ is trivial.  But $\Sha(F,G^0)$ is also trivial, by Theorem~\ref{main} in case~(\ref{reduc comp hyp}) and by
Corollary~\ref{cor:char0conn} in case~(\ref{eq char 0 hyp}).  So part~(\ref{outer Shas trivial}) holds; and part~(\ref{char0 disconn Sha}) then follows by applying Lemma~\ref{disconn trivial Sha lem} with $\Omega$ equal to the set of divisorial discrete valuations on $F$.
\end{proof}

Note that the surjectivity hypothesis of Theorem~\ref{char0 triv Sha surj}(\ref{char0 disconn Sha}) is satisfied in particular if $k$ is algebraically closed, or if the morphism $G \to G/G^0$ has a section.
Even without that surjectivity hypothesis, we can still obtain a local-global principle, though for $\Sha_\PP(F,G)$, or equivalently $\Sha_X(F,G)$, where $X$ is the reduced closed fiber of a normal crossing model (see Proposition~\ref{sha containments}(\ref{ShaX union})).  We first prove the next lemma.

\begin{lem} \label{pts_on_comps}
Let $\XX$ be a normal crossings model of a semi-global field $F$ over a complete discrete valuation ring $R$.  Let $\PP, \UU, \BB$ be
as in Section~\ref{patching subsec}, with associated reduction graph $\Gamma$. 
Suppose that either
\begin{enumerate}
\renewcommand{\theenumi}{\roman{enumi}}
\renewcommand{\labelenumi}{(\roman{enumi})}
\item \label{reduc comp hypoth}
$G$ is a smooth affine group scheme over $R$ such that $G^0$ is reductive; or
\item \label{eq char 0 hypoth}
$R=k[[t]]$ for some field $k$ of characteristic zero, and $G$ is a linear algebraic group over $k$.
\end{enumerate}
Let
$\wp \in \BB$ be a branch of $\XX_k^{\rm red}$ at $P \in \PP$, lying on $U \in \UU$.  
Let $G_1$ be a connected component of $G$, and suppose that $G_1(F_\wp)$ is non-empty.  Then $G_1(F_P)$ is also non-empty; and if $\kappa(U)=\kappa(P)$ then $G_1(F_U)$ is non-empty.
\end{lem}

\begin{proof}
We first consider case~(\ref{reduc comp hypoth}).  Since $G^0$ is reductive over $R$, it is also reductive over the discrete valuation ring $\wh R_\wp$, whose fraction field is $F_\wp$.  It follows that the kernel of $H^1(\wh R_\wp,G^0) \to H^1(F_\wp,G^0)$ is trivial (see \cite{Nis84} and \cite[Th\'eor\`eme~I.1.2.2]{Gil94}).  Hence the $G^0$-torsor $G_1$ is trivial over $\wh R_\wp$; i.e., it has an $\wh R_\wp$-point.  Reducing that point modulo the maximal ideal $\wp$ of $\wh R_\wp$, we obtain an $E$-point of $G_1$, where $E$ is the fraction field of the complete local ring $A$ of $\bar U$ at $P$.  Since $A$ is a discrete valuation ring, applying the above cited result
a second time shows that $G_1$ has an $A$-point.  Reducing that point modulo the maximal ideal of $A$ yields a $\kappa(P)$-point of $G_1$.  By formal smoothness, the latter lifts to an $\wh R_P$-point of $G_1$.  Thus there is also an $F_P$-point of $G_1$, as asserted.

If in addition $\kappa(U)=\kappa(P)$, the above $\kappa(P)$-point of $G_1$ is thus also a $\kappa(U)$-point of $G_1$, and that induces a $k[U]$-point of $G_1$.  By formal smoothness, this lifts to an $\wh R_U$-point of $G_1$, and so there is an $F_U$-point on $G_1$.

We next consider case~(\ref{eq char 0 hypoth}).  Let
$\mc R_{\mathrm u}(G)$ be the unipotent radical of $G$ (or equivalently, of $G^0$), and let $\tilde G = G/\mc R_{\mathrm u}(G)$.   Since $G_1$ has an $F_\wp$-point $g_\wp$, the image $\tilde G_1$ of $G_1$ in $\tilde G$ also contains an $F_\wp$-point, viz.\ the image $\tilde g_\wp$ of $g_\wp$.  But the identity component $\tilde G^0$ of $\tilde G$ is $G^0/\mc R_{\mathrm u}(G)$, which is reductive.
So by case~(\ref{reduc comp hyp}) of the lemma,
$\tilde G_1$ contains an $F_P$-point $\tilde g_P$, and also an $F_U$-point $\tilde g_U$ if $\kappa(U)=\kappa(P)$.  Since $H^1(F_\xi,\mc R_{\mathrm u}(G))$ is trivial for $\xi = P,U$ (using $\cha(k)=0$, as in the proof of Corollary~\ref{cor:char0conn}), it follows from the cohomology exact sequence that the map $G(F_\xi) \to \tilde G(F_\xi)$ is surjective.  Pick a point $g_P \in G(F_P)$ that lies over $\tilde g_P$; and also pick $g_U \in G(F_U)$ that lies over $\tilde g_U$ if $\kappa(U)=\kappa(P)$.  Since
$\tilde g_P$ (resp., $\tilde g_U$) lies on the connected component $\tilde G_1$ of $\tilde G = G/\mc R_{\mathrm u}(G)$, and since the group $\mc R_{\mathrm u}(G)$ is connected, it follows that $g_P$ (resp., $g_U$) lies on $G_1$, as asserted.
\end{proof}

We now obtain the following local-global result for groups that need not be connected, but which are given over a field of characteristic zero.

\begin{thm} \label{char0disconn}
Let $G$ be a linear algebraic group over a field $k$ of characteristic zero, let $F$ be a semi-global field over $K:=k((t))$, and let $\XX$ be a normal crossings model of $F$ over $R:=k[[t]]$.  Assume that the closed fiber $X$ of $\XX$ is reduced, and that the associated reduction graph is a monotonic tree.  Then $\Sha_X(F,G)$ is trivial.
\end{thm}

\begin{proof}
It suffices by Proposition~\ref{sha containments}(\ref{ShaX union}) to show that if $\PP \subset X$ is a finite set as in Section~\ref{patching subsec} then $\Sha_\PP(F, G)$ is trivial.  Given $\PP$, we have associated sets $\UU, \BB$, along with a reduction graph $\Gamma$, which is a monotonic tree.
As explained before Remark~\ref{tree independent},
we choose the root of this monotonic tree to be an element $U_0 \in \UU$.  
For each element of $\PP \cup \UU$, we may consider its distance $n \ge 0$ from $U_0$ in the associated reduction graph $\Gamma$; here $n=0$ for $U_0$ itself.

As in Section~\ref{lgp sgp}, write
$F_\UU = \prod_{U \in \UU} F_U$,
$F_\PP = \prod_{P \in \PP} F_P$, and
$F_\BB = \prod_{\wp \in \BB} F_\wp$,
and identify $\Sha_\PP(F,G)$ with the double coset space $G(F_\UU)\backslash G(F_\BB)/G(F_\PP)$,
by \cite[Corollary 3.6]{HHK15}.
It suffices to show that $G(F_\UU)\backslash G(F_\BB)/G(F_\PP)$ consists just of the trivial double coset.  Since $\Sha_\PP(F,G^0) =  G^0(F_\UU)\backslash G^0(F_\BB)/G^0(F_\PP)$ is trivial by Corollary~\ref{cor:char0conn}, it suffices to show that every element of $ G(F_\UU)\backslash G(F_\BB)/G(F_\PP)$ contains a representative $(g^0_\wp)_{\wp \in \BB}$ such that each $g^0_\wp \in G^0(F_\wp)$.  Equivalently, we will show that
for every $(g_\wp)_{\wp \in \BB} \in G(F_\BB)$,
there exist elements $g_P \in G(F_P)$ and $g_U \in G(F_U)$ for all $P \in \PP$ and $U \in \UU$ such that
$g_U^{-1}g_\wp g_P \in G^0(F_\wp)$ for each branch $\wp \in \BB$ at $P \in \PP$ lying on $U \in \UU$.  We will construct these elements $g_\xi$, for all $\xi \in \UU \cup \PP$, by induction on the distance in $\Gamma$ from the root $U_0$ to~$\xi$.

If the distance is zero, then $\xi = U_0$, and we set $g_\xi = g_{U_0} = 1 \in G^0(F_{U_0})$.  Now take $i \ge 1$ and assume that $g_\xi$ has been defined for all $\xi \in \UU \cup \PP$ of distance less than $i$ from the root.  Let $\xi$ be a vertex of $\Gamma$ that has distance $i$ from the root $U_0$.

If $i$ is odd, then $\xi$ is an element $P \in \PP$.
In this case the vertex of $\Gamma$ that is adjacent to $P$ and that lies between $P$ and the root is an element $U \in \UU$ of distance $i-1 \ge 0$ from the root.  By the inductive hypothesis, $g_U \in G(F_U)$ has been defined.  Let $g_P$ be an element of $G(F_P)$ that lies on the same connected component of $G$ as $g_\wp^{-1} g_U \in G(F_\wp)$, where $\wp \in \BB$ is the branch at $P$ on $U$; such an element exists by Lemma~\ref{pts_on_comps}(\ref{eq char 0 hypoth}).  Thus $g_U^{-1}g_\wp g_P \in G^0(F_\wp)$.

On the other hand, if $i$ is even, then $\xi$ is an element $U \in \UU$, and the adjacent vertex between $\xi$ and the root is some $P \in \PP$ of distance $i-1$ from $U_0$.  The element $g_P \in G(F_P)$ has then been inductively defined, and again by Lemma~\ref{pts_on_comps}(\ref{eq char 0 hypoth}) there is an element $g_U \in G(F_U)$ that lies on the same connected component of $G$ as $g_\wp g_P \in G(F_\wp)$, where $\wp \in \BB$ is the branch at $P$ on $U$.  (Here $\kappa(U) \subseteq \kappa(P)$ since $P \in \bar U$, and the reverse containment holds because $\Gamma$ is monotonic.  So Lemma~\ref{pts_on_comps} indeed applies in this case.)  Again, we have $g_U^{-1}g_\wp g_P \in G^0(F_\wp)$, as asserted.
\end{proof}

\begin{rem}
Concerning the relationship between Theorem~\ref{char0 triv Sha surj}(\ref{char0 disconn Sha}) and Theorem~\ref{char0disconn}, it is not known in general whether the containment
$\Sha_X(F,G) \to \Sha(F,G)$ is an equality.  Theorem~8.10(ii) of \cite{HHK15} provides a set of conditions under which equality holds, for a linear algebraic group $G$ over $F$: that $G^0$ is a reductive group over the model $\XX$; that $\bar G := G/G^0$ is a constant finite group scheme; and that moreover the homomorphism $G((F_P)_v) \to \bar G((F_P)_v)$ is surjective for every point $P \in X$ and every discrete valuation $v$ on~$F_P$.
\end{rem}

\section{A lower bound on the Tate-Shafarevich set}\label{sectionlb}

In this section, we give a combinatorial description of an explicit quotient set of the Tate-Shafarevich set associated to a reductive group $G$ over the ground ring $R$ of our semi-global field $F$.  This ``lower bound'' on $\Sha(F,G)$, which is given in Proposition~\ref{dbl coset surj}(\ref{kappa-surj}), is in the context of a model $\XX$ such that each irreducible component of the reduced closed fiber
is isomorphic to a projective line (over some extension of the residue field $k$).  In that situation, it will afterwards be used to obtain counterexamples to a local-global principle in Section~\ref{sect: counterex}.

Our result on $\Sha(F,G)$ below relies on the notion of $\R$-equivalence (see the discussion just before Proposition~\ref{branch factorization}).
By Theorems~\ref{dvr specialization} and~\ref{main appendix} of the Appendix, 
if $A$ is a regular local ring of dimension at most two, with fraction field $L$ and residue field $\ell$, and for any reductive group $G$ over $A$, there is an associated homomorphism on $\R$-equivalence classes $\spcl_A : G(L)/\R \to G(\ell)/\R$, known as the {\it specialization map}, which is compatible with the natural reduction map $G(A) \to G(\ell)$.
Here the compatibility condition is that the compositions $G(A) \to G(L) \to G(L)/\R \overset{\spcl_A}\to G(\ell)/\R$ and $G(A) \to G(\ell) \to G(\ell)/\R$ agree, where $G(A) \to G(\ell)$ is the natural pullback map. If $A$ is complete, the specialization map is surjective 
since $G(A) \to G(\ell)$ is.

We preserve the notation from Section~\ref{patching subsec}, with $F$ a semi-global field over a complete discrete valuation ring $R$ having fraction field $K$, and with $\XX$ a normal crossings model of $F$ together with sets $\PP, \UU, \BB$.
As in Section~\ref{factor sec}, we write $\kappa(P)$ for the residue field at a point $P \in \PP$ and $\kappa(U)$ for the constant field of $U \in \UU$.  We write $k(U)$ for the function field of $U \in \UU$ and $k[U]$ for the coordinate ring of $U$; these each contain $\kappa(U)$.  
If $\wp$ is a branch on $U \in \UU$ at $P \in \PP$, we define $\kappa(\wp)$ as $\kappa(\wp) := \kappa(P)$. This is not to be confused
with the residue field $k_{\wp}$ of the discrete valuation ring $\wh{R}_{\wp} \subset F_{\wp}$.

\begin{lem} \label{thetas}
With notation as above, let $G$ be a reductive group over $R$.
\renewcommand{\theenumi}{\alph{enumi}}
\renewcommand{\labelenumi}{(\alph{enumi})}
\begin{enumerate}
\item \label{thetas exist}
The specialization maps on regular local rings induce specialization maps
\[\xymatrix{\theta_\wp  :  G(F_\wp) \to  G(\kappa(\wp))/\R}, \ \ \theta_P : G(F_P) \to G(\kappa(P))/\R, \ \ \theta_{U} : G(F_U)\to G(k(U))/\R, \]
for each $\wp\in \BB$, $P\in \PP$, $U\in \UU$, which 
are homomorphisms that
factor through $G(F_\wp)/\R$, $G(F_P)/\R$, and $G(F_U)/\R$, respectively.
\item \label{theta P wp}
If $\wp \in \BB$ is a branch at $P \in \PP$, then
the maps $\theta_P$ and $\theta_\wp$ are surjective, and $\theta_P$ is the restriction of $\theta_\wp$ to $G(F_P)$.
\item \label{theta U}
The restriction of $\theta_\wp$ to $G(F_U)$ factors through $\theta_U$, if $\wp \in \BB$ is a branch on $U\in \UU$.
\item \label{U lines}
If $U\in \UU$ is isomorphic to an open subset of a projective line over a finite field extension of $k$, then the natural map $G(\kappa(U))/\R \to G(k(U))/\R$ is an isomorphism, and $\theta_U$ is surjective.
\end{enumerate}
\end{lem}

\begin{proof}
The map $\theta_P$ is defined to be the composition of $G(F_P) \to G(F_P)/\R$ with the 
specialization map of Theorem~\ref{main appendix} with respect to the local ring $\wh R_P$; and this is surjective since $\wh R_P$ is complete.
If $\wp \in \BB$ is a branch at $P \in \PP$, then
the map $\theta_\wp$ is defined to be the composition
\[G(F_\wp) \to G(F_\wp)/\R \to G(k(U)_P)/\R \to G(\kappa(P))/\R = G(\kappa(\wp))/\R,\]
where the second and third maps in this composition are specializations,
and $k(U)_P$ is the completion of $k(U)$ at $P$ (which is the same as the residue field $k_\wp$ of $F_\wp$).  The first map is trivially surjective, and the second and third maps are surjective because they are specializations with respect to complete discrete valuation rings.
By Theorem~\ref{main appendix}(\ref{sp hom})
together with the definitions of $\theta_P$ and $\theta_\wp$, 
the map $\theta_P$ coincides with the composition $G(F_P) \hookrightarrow G(F_\wp) \,{\buildrel \theta_\wp \over \to}\, G(\kappa(\wp))/\R = G(\kappa(P))/\R$.
Note that this composition factors through $G(F_P)/\R$ and $G(F_\wp)/\R$ by the properties of the specialization map in Theorem~\ref{main appendix}.

For $U \in \UU$, let $\eta$ be the generic point of the curve $U$, so that the localization $\wh R_{U,\eta}$ of $\wh R_U$ at $\eta$ is a discrete valuation ring with fraction field $F_U$, residue field $k(U)$, and completion $\wh R_\eta$.  We
define the map $\theta_U$ as the composition of the natural map $G(F_U) \to G(F_U)/\R$ with the specialization map $G(F_U)/\R \to G(k(U))/\R$ with respect to the discrete valuation ring $\wh R_{U,\eta}$.
By Remark~\ref{dvr spcl rk}(\ref{spcl via completion}), $\theta_U$ is the same as the composition
$G(F_U) \to G(F_\eta) \overset{\spcl}\to G(k(U))/\R$, where $F_\eta$ is the fraction field of $\wh R_\eta$ and the second map is specialization with respect to $\wh R_\eta$. Note that if $\wp \in \BB$ is a branch on $U$, then the restriction of $\theta_\wp$ to $G(F_U)$ is the composition of $\theta_{U}$ with the map $G(k(U))/\R \to G(k(U)_P)/\R \to G(\kappa(P))/\R$, where the latter map is given by specialization as above.
This completes the proof of parts~(\ref{thetas exist})-(\ref{theta U}).

For part~(\ref{U lines}), suppose that $U$ is isomorphic to an open subset of a projective line $\P^1_{k'}$ over some finite extension $k'$ of $k$.
Consider the natural map $\alpha:G(\kappa(U))/\R \to G(k(U))/\R$ that is induced by the inclusion of $k'=\kappa(U)$ into $k(U) \simeq k'(x)$, where $x$ is a coordinate function on the affine line.  Then $\beta\alpha$ is the identity on $G(k(U))/\R$, where $\beta:G(k(U))/\R \to G(\kappa(U))/\R$ is given by specialization at a $k'$-point of $\bar U$.  Hence $\alpha$ is injective.  For surjectivity, we want that every element of $G(k'(x))$ is R-equivalent
to an element in the image of $G(k')$.  An element of $G(k'(x))$ is given by a
map $g:S \to G$, for $S$ an open  subset of the $x$-line over $k'$ that we may assume
contains $x=0$ (after making a change of coordinates on the line).
View $S \times S$, with coordinates $s,x$, as an open subset of a family of $x$-lines parametrized by the $s$-line; and consider the rational map $f:S \times S \dashrightarrow G$ given by $f(s,x)=g(sx)$.
The restriction of $f$ to $s=0$ is the constant morphism $S \to G$ with value $g(0) \in G(k') \subset
G(k'(x))$, and the restriction of $f$ to $s=1$ is $g:S \to G$.  So $f$ defines a rational map
$\A^1_{k'(x)} \dashrightarrow G_{k'(x)}$ that is defined at $s=0,1$ and
connects $g$ to $g(0)$, proving that $\alpha$ is surjective and hence an isomorphism.

Still under the hypothesis that $U$ is isomorphic to an open subset of a projective line~$\P^1_{k'}$, we then want to show that $\theta_U: G(F_U) \to G(k(U))/\R \simeq G(k')/\R$ is surjective, where $k'=\kappa(U)$.  Every class in $G(k(U))/\R$ is represented by an element of $G(k')$ and hence by an element of $G(k[U])$, where we view $G(k') \subseteq G(k[U]) \subseteq G(k(U))$.  Also, the reduction map $G(\wh R_U) \to G(k[U])$ is surjective by formal smoothness.
But the composition $G(\wh R_U) \to G(k[U]) \to G(k(U)) \to G(k(U))/\R$ is the restriction of $\theta_U$ to $G(\wh R_U)$, since the restriction of $G(F_\eta) \overset{\spcl}\to G(k(U))/\R$ 
to the subgroup $G(\wh R_\eta)$ is induced by the reduction map.  Hence
$\theta_U$ is surjective under the additional hypothesis on $U$.
\end{proof}

By Theorem~\ref{equalshas} above, together with the double coset formula
\cite[Corollary~3.6]{HHK15}, we have bijections
\[\Sha(F, G) \simeq \Sha_\PP(F, G) \simeq \prod_{U \in \UU} G(F_{U}) \,\backslash \prod_{\wp \in \BB} G(F_\wp) \,/ \prod_{P \in \PP} G(F_{P}).\]  
The first part of the next result shows that this double coset description is compatible with the specialization map, while the second part gives our explicit quotient of $\Sha(F,G)$.

\begin{prop} \label{dbl coset surj}
With notation as above, let $G$ be a reductive group over $R$.
\renewcommand{\theenumi}{\alph{enumi}}
\renewcommand{\labelenumi}{(\alph{enumi})}
\begin{enumerate}
\item \label{k-surj}
The maps $\theta_\wp$ together define a surjection
\[\theta_0:\Sha(F, G) \to
\prod_{U \in \UU} (G(k(U))/\R)  \backslash \prod_{\wp \in \BB}  (G(\kappa(\wp))/\R)  \slash \prod_{P \in \PP} (G(\kappa(P))/\R)\]
via the above identification
\[\Sha(F, G) \simeq \prod_{U \in \UU} G(F_{U}) \,\backslash \prod_{\wp \in \BB} G(F_\wp) \,/ \prod_{P \in \PP} G(F_{P}).\]
\item \label{kappa-surj}
The natural surjection
\[\xymatrix{
\prod_{U \in \UU} (G(\kappa(U))/\R)  \backslash \prod_{\wp \in \BB}  (G(\kappa(\wp))/\R)  \slash \prod_{P \in \PP} (G(\kappa(P))/\R \ar@{->>}^\pi[d] \\
\prod_{U \in \UU} (G(k(U))/\R)  \backslash \prod_{\wp \in \BB}  (G(\kappa(\wp))/\R)  \slash \prod_{P \in \PP} (G(\kappa(P))/\R,\\
 }\]
is an isomorphism if each $U \in \UU$ is an open subset of a projective line over a finite field extension of $k$.  Hence in that case $\theta_0$ becomes identified with the surjection
\[\theta = \pi^{-1} \circ \theta_0:\Sha(F, G) \to
\prod_{U \in \UU} (G(\kappa(U)/\R)  \backslash \prod_{\wp \in \BB}  (G(\kappa(\wp))/\R)  \slash \prod_{P \in \PP} (G(\kappa(P))/\R).\]
\end{enumerate}
\end{prop}

\begin{proof}
As $\wp$ ranges over $\BB$, the specialization maps $\theta_\wp$ in Lemma~\ref{thetas}(\ref{thetas exist})
together define a homomorphism $\prod_\wp G(F_\wp) \to \prod_\wp G(\kappa(\wp)/\R$.
Similarly, there are homomorphisms $\prod_U G(F_U) \to \prod_U G(k(U)/\R$
and $\prod_P G(F_P) \to \prod_P G(\kappa(P)/\R$.  By parts (\ref{theta P wp}) and (\ref{theta U}) of Lemma~\ref{thetas}, these maps are compatible.  So together they induce a homomorphism $\theta_0$ on the double coset space  $\prod_{U \in \UU} G(F_{U}) \,\backslash \prod_{\wp \in \BB} G(F_\wp) \,/ \prod_{P \in \PP} G(F_{P})$, which, as above, we identify with
$\Sha(F, G)$, via Theorem~\ref{equalshas} and
\cite[Corollary~3.6]{HHK15}.  The map $\theta_0$ is surjective by the surjectivity of $\theta_\wp$ in
Lemma~\ref{thetas}(\ref{theta P wp}).  This proves part~(\ref{k-surj}).
(Note here that in writing these quotients, we do not assume that the maps from $\prod_U G(k(U))/\R$ and $\prod_P G(\kappa(P))/\R$ to $\prod_\wp  G(\kappa(\wp))/\R$ are injective; just that we are quotienting by their images in $\prod_\wp  G(\kappa(\wp))/\R$.)

For part~(\ref{kappa-surj}), the vertical map is induced by the identity map on
$\prod_{\wp \in \BB}  (G(\kappa(\wp))/\R)$ via
the inclusion $\kappa(U) \hookrightarrow k(U)$, and is therefore surjective.  If each $U \in \UU$ is an open subset of a projective line, then the stated properties follow from Lemma~\ref{thetas}(\ref{U lines}).
\end{proof}

\begin{rem}
\renewcommand{\theenumi}{\alph{enumi}}
\renewcommand{\labelenumi}{(\alph{enumi})}
\begin{enumerate}
\item
The above maps $\theta_0$ and $\theta$ on 
\[\Sha(F,G) \simeq \prod_{U \in \UU} G(F_{U}) \,\backslash \prod_{\wp \in \BB} G(F_\wp) \,/ \prod_{P \in \PP} G(F_{P})\] 
factor through $\prod_U (G(F_U)/\R) \backslash \prod_\wp (G(F_\wp)/\R) \slash \prod_P (G(F_P)/\R)$, since each of the maps $\theta_\wp$, $\theta_P$, $\theta_U$ factors through the corresponding group of $\R$-equivalence classes
by Lemma~\ref{thetas}(\ref{thetas exist}).
\item
In the case that $G$ is a torus, the maps $G(F_P)/\R \to G(\kappa(P))/\R$ and $G(F_\wp)/\R \to G(\kappa(\wp))/\R$ are isomorphisms, and so is the surjection $\theta_0$ in Proposition~\ref{dbl coset surj}(\ref{k-surj}).  This follows from \cite[Proposition~4.6(a,b) and Lemma~6.2]{CHHKPS} (which are given in terms of flasque resolutions), together with \cite[Th\'eor\`eme~2, p.199]{CTS77} (which interprets $T(k)/\R$ in terms of flasque resolutions if $T$ is a torus; see also \cite[Theorem~3.1]{CTS87}).
As a result, if $G$ is a torus, then the vertical map in Proposition~\ref{dbl coset surj}(\ref{kappa-surj}) defines an epimorphism
\[\prod_{U \in \UU} (G(\kappa(U))/\R)  \backslash \prod_{\wp \in \BB}  (G(\kappa(\wp))/\R)  \slash \prod_{P \in \PP} (G(\kappa(P))/\R \to \Sha(F,G),\]
which is an isomorphism (inverse to $\theta$) if each $U \in \UU$ is an open subset of a projective line over a finite field extension of $k$.
See \cite[Proposition~6.3(a,c)]{CHHKPS}.
\end{enumerate}
\end{rem}

In the situation of Proposition~\ref{dbl coset surj}(\ref{kappa-surj}), if $g \in \prod_{\wp \in \BB} G(F_\wp)$, we will write $\theta(g)$ for $\theta([g])$, where $[g]$ is the class of $g$ in $\prod_U G(F_{U}) \,\backslash \prod_\wp G(F_\wp) \,/ \prod_P G(F_{P}) \simeq \Sha(F,G)$.  This is the same as the class of $(\theta_\wp(g))_\wp$ in
$\prod_U (G(\kappa(U))/\R)  \backslash \prod_\wp  (G(\kappa(\wp))/\R)  \slash \prod_P (G(\kappa(P))/\R)$.

Recall (from Section~\ref{factor sec}, Equation~(\ref{specializable}))
that if $\wp$ is a branch on $U$ at $P$, there is the subgroup $G_{\rm s}(\wh R_\wp) = G(\wh R_P)G_\wp(\wh R_\wp) \subseteq G(\wh R_\wp)$ of specializable elements, together with the
specialization map $\Theta_\wp:G_{\rm s}(\wh R_\wp) \to G(\kappa(\wp))$ obtained by composing the reduction maps $G_{\rm s}(\wh R_\wp) \to G(\wh R_P/\wp) \to G(\kappa(P)) = G(\kappa(\wp))$. By Theorem~\ref{main appendix}(\ref{special restriction}), the map
$\Theta_\wp:G_{\rm s}(\wh R_\wp) \to G(\kappa(\wp))$ lifts the restriction of $\theta_\wp:G(F_\wp) \to
G(\kappa(\wp))/\R$ to $G_{\rm s}(\wh R_\wp)$.  Note also that by Proposition~\ref{branch factorization}, for any $(\tilde g_\wp) \in \prod_{{\wp \in \BB}} G(F_\wp)$ there is a representative $(g_\wp)$ of the same double coset in $\Sha(F,G)$ such that each $g_\wp$ is R-equivalent to $\tilde g_\wp$ and lies in $G_{\rm s}(\wh R_\wp)$.

The next result will be useful in proving Theorem~\ref{computation-of-sha}.

\begin{cor} \label{equal-in-gkr}
Let $G$ be a reductive group over $R$, and
suppose that every irreducible component of the reduced closed fiber of the model $\XX$ is isomorphic to a projective line over a finite field extension of $k$.  Let
$g,g' \in \prod_{\wp \in \BB} G(F_\wp)$, with $g = (g_\wp)_\wp$, $g' = (g'_\wp)_\wp$.  Suppose that $\theta(g) = \theta(g')$, with $\theta$ as in Proposition~\ref{dbl coset surj}(\ref{kappa-surj}).  Then there is a family $(\alpha_\xi)_{\xi \in \UU \cup \PP}$ with $\alpha_\xi \in G(F_\xi)$ such that
\[\alpha_U g'_\wp \alpha_P \in G_{\rm s}(\wh R_\wp) \ {\rm and} \ \theta_\wp(\alpha_U g'_\wp \alpha_P) = \theta_\wp(g_\wp) \in G(\kappa(\wp))/\R\]
for each branch $\wp \in \BB$ at $P \in \PP$ lying on $U \in \UU$.
\end{cor}

\begin{proof}
By hypothesis, there exist $g_U \in G(\kappa(U))/\R$ for all $U \in \UU$, and $g_P \in G(\kappa(P))/\R$ for all $P \in \PP$, such that $\theta_\wp(g_\wp)=g_U \theta_\wp(g'_\wp) g_P \in G(\kappa_\wp)/\R$ for every branch $\wp \in \BB$ at $P \in \PP$ lying on $U \in \UU$.
Since the morphisms $\theta_P$ and $\theta_U$ are surjective by Lemma~\ref{thetas}(\ref{theta P wp}),(\ref{U lines}),
there exist $\alpha_U \in G(F_U)$ and $\alpha_P' \in G(F_P)$ such that $\theta_{U}(\alpha_U) = g_U$ and
$\theta_P(\alpha_P') = g_P$ (where we identify $G(\kappa(U))/\R$ with $G(k(U))/\R$ via Lemma~\ref{thetas}(\ref{U lines})).
Thus
$\theta_\wp(\alpha_U g'_\wp \alpha_P') = \theta_\wp(g_\wp) \in G(\kappa(\wp))/\R$ for each branch $\wp \in \BB$.  By Proposition~\ref{branch factorization}, there is an element $\tilde g = (\tilde g_P)_P \in \prod_{P \in \PP} G(F_P)$ such that $(\alpha_U g'_\wp \alpha_P') \tilde g_P \in G_{\rm s}(\wh R_\wp)$ and $\tilde g_P$ is R-trivial in $G(F_\wp)$, for every branch $\wp \in \BB$ at $P \in \PP$ on $U \in \UU$.  Setting $\alpha_P = \alpha_P' \tilde g_P$ completes the proof, since
$\theta_\wp(\alpha_U g'_\wp \alpha_P) = \theta_\wp(\alpha_U g'_\wp \alpha_P') = \theta_\wp(g_\wp) \in G(\kappa(\wp))/\R$ by the R-triviality of $\tilde g_P$.
\end{proof}

We can view the above corollary as saying that the adjustment of elements appearing in Proposition~\ref{branch factorization} can be done compatibly with $\theta$. 

\section{An exact computation of the Tate-Shafarevich set} \label{Sha computation section}

In this section, we build on the results of Section~\ref{sectionlb} to obtain an explicit computation of $\Sha(F,G)$ under a somewhat stronger hypothesis; see Theorem~\ref{explicit Sha}.  As before, we let $G$ be a reductive group over a complete discrete valuation ring $R$; we let $\XX$ be a normal crossings model of a semi-global field $F$ over $R$; and we take $\PP, \UU, \BB$ as in Section~\ref{patching subsec}.  The additional condition that we will need below is the following:

\begin{hyp} \label{rational comps ints}
Assume that each irreducible component of the reduced closed fiber of $\XX$ is a projective line over the residue field $k$ of $R$, and that each intersection point of components is defined over $k$.  Thus $\kappa(P)=\kappa(U)=k$ for all $P \in \PP$ and all $U \in \UU$.
\end{hyp}

Using Theorem~\ref{explicit Sha}, we will obtain counterexamples to the local-global principle for semisimple simply connected groups in Section~\ref{sect: counterex}, beyond those that will arise via Proposition~\ref{dbl coset surj}(\ref{kappa-surj}).  These will
include an example in which $\Sha(F,G)$ is infinite.  

\subsection{A double coset description of the Tate-Shafarevich set}

Under Hypothesis~\ref{rational comps ints}, we show in Theorem~\ref{computation-of-sha} that the surjective specialization map
\begin{eqnarray*}
\theta:\Sha(F, G) &\to& \prod_{U \in \UU} (G(\kappa(U))/\R)  \backslash \prod_{\wp \in \BB}  (G(\kappa(\wp))/\R)  \slash \prod_{P \in \PP} (G(\kappa(P))/\R\\
&=& \prod_{U \in \UU} (G(k)/\R)  \backslash \prod_{\wp \in \BB}  (G(k)/\R)  \slash \prod_{P \in \PP} (G(k)/\R)
\end{eqnarray*}
defined in Proposition~\ref{dbl coset surj}(\ref{kappa-surj})
is bijective, thereby providing a computable double coset description of
$\Sha(F, G)$.
Afterwards, in Corollary~\ref{explicit Sha}, we make this more explicit, describing $\Sha(F,G)$ just in terms of the number of loops in the reduction graph.

We first prove a lemma about specialization.  Here we preserve the notation from Section~\ref{sectionlb}.

\begin{lemma} \label{equal-in-gk}
Under Hypothesis~\ref{rational comps ints}, let $G$ be a reductive group over $R$ and let $g,g' \in \prod_{\wp \in \BB} G_{\rm s}(\wh R_\wp)$, with $g = (g_\wp)_\wp, g' = (g'_\wp)_\wp$.  Suppose that $\theta_\wp(g_\wp) = \theta_\wp(g'_\wp) \in G(\kappa(\wp))/\R$ for all branches $\wp \in \BB$.
Then there exists a finite set of closed points $\PP^*$ containing $\PP$ such that
the set $\UU^*$ of connected components of the complement of $\PP^*$ in the reduced closed fiber has the following property: there exists a family $(\alpha_V)_{V \in \UU^*}$ of elements $\alpha_V \in G(\wh R_V)$ for $V \in \UU^*$ such that if $P \in \PP$ is in the closure of $V \in \UU^*$ and $\wp$ is a branch on $V$ at $P$, then $\alpha_V g'_\wp \in G_{\rm s}(\wh R_\wp)$ and
$\Theta_\wp(\alpha_V g'_\wp ) = \Theta_\wp(g_\wp) \in G(\kappa(\wp))$.
\end{lemma}

Note that in the above lemma, each element of $\UU^*$ is a dense open subset of an element of $\UU$, and so the sets $\UU$ and $\UU^*$ are in natural bijection.  Note also that the elements $g_\wp, g'_\wp$ are defined only for $\wp \in \BB$; i.e., for $\wp$ a branch at a point $P \in \PP$.  For this reason, the conclusion of the above lemma considers just points $P \in \PP$.

\begin{proof}
Let  $U \in \UU$ and
let $\bar{U} \cap \PP = \{P_1,\dots,P_n\}$, with $n \ge 1$.  For each $i = 1,\dots,n$, let $f_i$
 be a rational function on $\bar U$ that has a pole (of some order) at $P_i$ and is regular elsewhere, with zeros (of order at least one) at each $P_j$ for $j \ne i$.  Let $\wp_i$ be the branch at $P_i$ that lies on $U$.

Since $\theta_{\wp_i}(g_{\wp_i}) =  \theta_{\wp_i}(g'_{\wp_i})  \in G(\kappa(\wp_i))/\R
= G(k)/\R$, the element $\Theta_{\wp_i}(g_{\wp_i}) \Theta_{\wp_i}(g'_{\wp_i})^{-1} \in G(k)$ is $\R$-trivial.  So by \cite[Lemma~II.1.1]{GilleIHES}
there exists some $\gamma_{U,i} \in G(k({\mathbb P}^1))$
 such that
${\gamma}_{U,i}(0) = 1  \in G(k)$
  and ${\gamma}_{U,i}(\infty)  = \Theta_{\wp_i}(g_{\wp_i}) \Theta_{\wp_i}(g'_{\wp_i})^{-1} \in G(k)$ (viewing ${\gamma}_{U,i}$ as a rational map from ${\mathbb P}^1_k$ to $G$).
  Pulling back $\gamma_{U,i}$ by $f_i$, we obtain an element
$\beta_{U,i}  \in G(k(\bar{U}))$
such that
$\beta_{U,i}(P_i) = \Theta_{\wp_i}(g_{\wp_i}) \Theta_{\wp_i}(g'_{\wp_i})^{-1}  \in G(\kappa(\wp_i)) = G(k)$
  and $\beta_{U,i}(P_j)  = 1 \in G(k)$ for $j \ne i$.
Since $\beta_{U,i}$ is a rational map from $\bar U$ to $G$
that is defined at $\{P_1,\dots,P_n\}$, there is an affine open subset $W_i \subset \bar U$ containing $\{P_1,\dots,P_n\}$ such that $\beta_{U,i} \in  G(k[W_i])$.
Let $\tilde{\alpha}_{W_i} \in G(\wh R_{W_i})$ be a lift of $\beta_{U,i}$.  Note that
$\wh R_{W_i}$ is contained in $\wh R_{P_1},\dots,\wh R_{P_n}$, and also in $\wh R_{V_i}$ where $V_i = W_i \cap U$.  Thus $\tilde{\alpha}_{W_i}$ is an element of $G(\wh R_{V_i})$ and of $G(\wh R_{P_j}) \subseteq G_{\rm s}(\wh R_{\wp_j}) \subseteq G(\wh R_{\wp_j})$ for $j=1,\dots,n$.
By the definition of $\Theta_{\wp_i}$, we have
\[ \Theta_{\wp_i}(\tilde{\alpha}_{W_i} )  = \beta_{U,i}(P_i)
  = \Theta_{\wp_i}(g_{\wp_i}) \Theta_{\wp_i}(g'_{\wp_i})^{-1}  \in G(k);\]
and for $j \ne i$ we have
\[ \Theta_{\wp_j}(\tilde{\alpha}_{W_i} )  = \beta_{U,i}(P_j)
  = 1  \in G(k).\]

Let $W_U = W_1 \cap \cdots \cap W_n \subseteq \bar U$; let
$V_U = V_1 \cap \cdots \cap V_n = W_U \cap U \subseteq U$;
and let $\alpha_{W_U} = \prod_{i=1}^n \tilde{\alpha}_{W_i}\in G(\wh R_{W_U})$.
Then for each $i$, $\alpha_{W_U}$
is an element of $G(\wh R_V)$ and of $G(\wh R_{P_j}) \subseteq G_{\rm s}(\wh R_{\wp_j}) \subseteq G(\wh R_{\wp_j})$ for all $j$; and
we have
\[\Theta_{\wp_i}(\alpha_{W_U}) = \Theta_{\wp_i}(\tilde{\alpha}_{W_i } ) = \Theta_{\wp_i}(g_{\wp_i})
\Theta_{\wp_i}(g'_{\wp_i})^{-1}  \in G(k).\]

After performing the above construction (separately) for each $U \in \UU$, we have that 
  \[ \Theta_\wp (\alpha_{W_U}g'_{\wp}) = \Theta_\wp(\alpha_{W_U}) \Theta_\wp(g'_\wp)
= \Theta_\wp(g_\wp) \in G(\kappa(\wp)) = G(k)\]
for every branch $\wp \in \BB$ at a point $P \in \PP$ lying on an element $U \in \UU$.

Finally, let $\PP^* = {\XX}_k^{\red} \smallsetminus(  \bigcup_{U \in \UU_0}V_U)$.  Then the set  $\UU^*$ of irreducible components of ${\XX}_k^{\red} \smallsetminus \PP^*$ is just the collection of the sets $V_U$, for $U \in \UU$.
For each $V=V_U \in \UU^*$, write $\alpha_V = \alpha_{W_U}$.
Let $P \in \PP$ be in the closure of $V \in \UU^*$ and $\wp$ is a branch on $V$ at $P$; thus $P \in \PP \cap \bar{U}$.
By the above display,
$\Theta_\wp(\alpha_V g'_\wp ) = \Theta_\wp(g_\wp) \in G(\kappa(\wp))$, as asserted.
\end{proof}

Recall (from the beginning of Section~\ref{lgp mon tree sec}) that given a reductive group $G$, there is an associated finite group scheme $\mu(G)$.

\begin{thm} \label{computation-of-sha}
Under Hypothesis~\ref{rational comps ints}, assume that the closed fiber $\XX_k$ is reduced.
Let $G$ be a reductive group over $R$ such that $\mu(G)$ is \'etale (which is automatic if $\cha(k)=0$).
Then the natural map
\[\theta:\Sha(F, G) \to \prod_{U \in \UU}  (G(k)/\R)  \backslash \prod_{\wp \in \BB}  (G(k)/\R)  \slash \prod_{P \in \PP} (G(k)/\R)\]
is a bijection.
\end{thm}

\begin{proof}
By Proposition~\ref{dbl coset surj}(\ref{kappa-surj}), $\theta$ is surjective.  So it suffices to prove injectivity.  Given elements $\zeta,\zeta'$ of $\Sha(F,G)$ with the same image under $\theta$, we will show that $\zeta=\zeta'$.

We may identify $\Sha(F,G)$ with $\prod_{U \in \UU} G(F_{U}) \,\backslash \prod_{{\wp \in \BB}} G(F_\wp) \,/ \prod_{P \in \PP} G(F_{P})$ for a choice of $\PP$ and $\UU$.
Pick respective representatives
$(g_\wp)_\wp, (g'_\wp)_\wp \in \prod_{{\wp \in \BB}} G(F_\wp)$ for our elements $\zeta,\zeta' \in \Sha(F,G)$.  By Proposition~\ref{branch factorization}, after adjusting the representative of $\zeta$ on the right by an element of $\prod_{P \in \PP} G(F_P)$, we may assume that each $g_\wp \in G_{\rm s}(\wh R_\wp)$.
Next, by Corollary~\ref{equal-in-gkr}, after adjusting the representative of $\zeta'$ on the left by an element of $\prod_{U \in \UU} G(F_U)$ and on the right by an element of $\prod_{P \in \PP} G(F_P)$, we may assume that $g'_\wp \in G_{\rm s}(\wh R_\wp)$ and
$\theta_\wp(g_\wp) = \theta_\wp(g'_\wp) \in G(k)/\R$ for all $\wp \in \BB$.

Now invoking Hypothesis~\ref{rational comps ints}, we may apply Lemma~\ref{equal-in-gk} to obtain a finite set $\PP^*$ containing $\PP$, together with the associated set $\UU^*$ in bijection with $\UU$, and elements $\alpha_V \in G(\wh R_V)$ for $V \in \UU^*$, satisfying the condition stated there.  Let $\BB^*$ be the associated set of branches; this consists of the elements of $\BB$ together with an additional element at each point of $\PP^*$ that is not in $\PP$.  (Note that the closed fiber of our model is unibranched at these latter points.)
Under the isomorphism
\[\Sha(F,G) \simeq \prod_{U \in \UU^*} G(F_{U}) \,\backslash \prod_{{\wp \in \BB^*}} G(F_\wp) \,/ \prod_{P \in \PP^*} G(F_{P}),\]
the elements $\zeta,\zeta' \in \Sha(F,G)$ respectively have representatives
$(g_\wp)_{\wp \in \BB^*}, (g'_\wp)_{\wp \in \BB^*}$, where $g_\wp, g'_\wp$ are as before for $\wp \in \BB$, and are equal to $1$ if $\wp \in \BB^*$ is not in $\BB$.
(See the end of Section~\ref{lgp sgp}.)
By the conclusion of Lemma~\ref{equal-in-gk}, after replacing $\PP, \UU, \BB$ by $\PP^*, \UU^*, \BB^*$, and after adjusting the representative of $\zeta'$ on the left by an element of (the new) $\prod_{U \in \UU} G(F_U)$, we may assume that $\Theta(g_\wp)=\Theta(g'_\wp) \in G(\kappa(\wp)) = G(k)$ for all branches $\wp \in \BB$ at points of $\PP$ at which the closed fiber is not unibranched.

By Lemma~\ref{equal-in-rpmt}, after adjusting the representative of $\zeta$ on the right by some element of $\prod_{P \in \PP} G(F_P)$, we may assume that $g_\wp, g'_\wp \in G_{\rm s}(\wh R_\wp) \subseteq G(\wh R_\wp)$ have equal images in $G(k_\wp)$.  So by Corollary~\ref{torsors closed fiber}(\ref{torsors same cl fib}), the tuples $(g_\wp)$, $(g'_\wp)$ define $G$-torsors over $\XX$ whose restrictions to $\XX_k$ are isomorphic.  By the hypotheses on the closed fiber and on $\mu(G)$, Proposition~\ref{GPS prop} applies.  Hence the $G$-torsors over $F$ defined by these tuples are isomorphic; i.e., $\zeta=\zeta'$.
\end{proof}

\subsection{Explicit description of the obstruction set}

Below we make the target space of $\theta$ in Theorem~\ref{computation-of-sha} more concrete, in terms of the structure of the reduction graph.  The resulting description of $\Sha(F,G)$, given in
Theorem~\ref{explicit Sha},
provides a necessary and sufficient condition for a local-global principle to hold under the above hypothesis.  This theorem will rely on Lemma~\ref{decorated} below, which draws on ideas from \cite{CHHKPS}, where we studied $\Sha(F,G)$ in the case that $G$ is a torus, by means of the cohomology of decorated graphs.  Given a (connected) graph $\Gamma$ and a {\em coefficient system} $A_\bullet$ consisting of abelian groups $A_v$ and $A_e$ associated to the vertices $v$ and edges $e$ of the graph, we had defined $H^i(\Gamma,A_\bullet)$ for $i=0,1$.  We then applied that to the case that $\Gamma$ is the reduction graph associated to a model of a semi-global field, where the coefficient system depended on the torus.  In our current more general situation, our algebraic groups need not be commutative, and so we instead need non-abelian coefficients.  Since the reduction graph of a model is a bipartite graph, and since the description of $\Sha(F,G)$ requires just $H^1$, we restrict ourselves here to defining $H^1$ with non-abelian coefficients just in the case of bipartite graphs, where the description is a bit simpler.

So let $\Gamma$ be a bipartite graph with edge set $E$ and vertex set $V$, with $V$ the disjoint union of subsets $V_1, V_2$, such that each edge has one vertex in $V_1$ and the other in $V_2$.  A {\em coefficient system} $G_\bullet$ will be a system of groups $G_e,G_v$ associated to the edges and vertices of $\Gamma$, together with a homomorphism $G_{(v,e)}:G_v \to G_e$ for every pair $(v,e)$ such that $v$ is a vertex of an edge $e$.  The homomorphisms $G_{(v,e)}$ for $v \in V_1$ together define a homomorphism $G_1:\prod_{v \in V_1} G_v \to \prod_{e \in E} G_e$, whose image acts on $\prod_{e \in E} G_e$ on the left.  Similarly, the homomorphisms $G_{(v,e)}$ for $v \in V_2$ together define a homomorphism $G_2:\prod_{v \in V_2} G_v \to \prod_{e \in E} G_e$, whose image acts on $\prod_{e \in E} G_e$ on the right.  With respect to these actions, we may form the double coset space
\[H^1(\Gamma,G_\bullet) := \prod_{v \in V_1} G_v \backslash \prod_{e \in E} G_e / \prod_{v \in V_2} G_v.\]
When the groups in $G_\bullet$ are abelian, this agrees with the definition in \cite[Section~5]{CHHKPS}.
(Also, whether or not the groups are abelian, $H^1(\Gamma,G_\bullet)$ is equal to $H^1(\Gamma,\ms G)$, where $\Gamma$ is regarded as a topological space together with a sheaf of groups $\ms G$ that is induced by the groups $G_e,G_v$.)

Viewing $\Gamma$ as a one-dimensional simplicial complex, we may also consider its cohomology group $H^1(\Gamma,\mbb Z)$.  This is a free $\mbb Z$-module of finite rank.

Given a group $G$ and a positive integer $m$, two elements $(g_1,\dots,g_m), (g_1',\dots,g_m') \in G^m$ are {\em uniformly conjugate} if there exists $h \in G$ such that $g_j' = hg_jh^{-1}$ for all $j$.  This equivalence relation will be denoted by $\sim$.

\begin{lem} \label{decorated}
Let $G_\bullet$ be a coefficient system on a bipartite graph $\Gamma$, as above.
Suppose that the groups $G_e$ and $G_v$ are all equal to a given group $G$, and the maps $G_{(v,e)}$ are each the identity.  Let $m$ be the number of cycles in the graph $\Gamma$; i.e., the rank of $H^1(\Gamma,\mathbb Z)$.  Then
$H^1(\Gamma,G_\bullet) \simeq G^m/{\sim}$
as pointed sets.
\end{lem}

\begin{proof}
First consider the case where $m=0$; i.e., where $\Gamma$ is a tree.  We proceed by induction on the number of edges of $\Gamma$.  If there is just one edge (and two vertices), then the assertion is trivial.  Now assume the assertion holds for trees with $n$ vertices.  If $\Gamma$ is a tree with $n+1$ vertices, let $v_0$ be a terminal vertex, and let $e_0$ be the unique edge having $v_0$ as a vertex.  Let $\Gamma'$ be the (bipartite) tree obtained by deleting $v_0$ and $e_0$, with edge set $E' = E \smallsetminus \{e_0\}$ and vertex set $V' = V \smallsetminus \{v_0\}$.  Let $G_\bullet'$ be the constant coefficient system $G$ on $\Gamma'$.
The projection map $\prod_{e \in E} G \to \prod_{e \in E'} G$ induces a bijection on the respective double coset spaces, since each tuple in $\prod_{e \in E} G$ is in same double coset as one whose $e_0$ entry is trivial.  Since $H^1(\Gamma',G_\bullet')$ is trivial by inductive hypothesis, so is $H^1(\Gamma,G_\bullet)$.

Now consider the case where $m>0$.  There exists a set $E_0 = \{e_1,\dots,e_m\} \subseteq E$ of $m$ edges  such that the graph obtained by deleting these edges (and retaining all the vertices) is a maximal subtree $\Gamma'$ of $\Gamma$.  For each $e \in E$ let the vertices of $e$ be $v_{e,1}, v_{e,2}$, with $v_{e,i} \in V_i$.  For $e=e_j$ write $v_{j,i} = v_{e_j,i}$.

By the previous case, $H^1(\Gamma',G_\bullet')$ is trivial, where $\Gamma'$ is given the induced bipartite structure and $G_\bullet'$ is the constant coefficient system on $\Gamma'$ as above.  Thus every element of $H^1(\Gamma,G_\bullet)$ is a double coset represented by an element of $\prod_{e \in E} G$ whose non-trivial entries all have indices in $E_0$.
Hence the composition
$G^m = \prod_{e \in E_0} G \to \prod_{e \in E} G \to H^1(\Gamma,G_\bullet)$ is surjective, where the first map assigns trivial entries for edges not in $E_0$.
To complete the proof, we will show that two elements $(g_1,\dots,g_m), (g_1',\dots,g_m') \in G^m$ have the same image in $H^1(\Gamma,G_\bullet)$ if and only if they are uniformly conjugate by some element $h \in G$.

Let $g_1,\dots,g_m,h \in G$, and for $1 \le i \le m$ let $g_j' = hg_jh^{-1}$.  Then the double cosets in $H^1(\Gamma,G_\bullet)$ associated to $(g_1,\dots,g_m)$ and to $(g_1',\dots,g_m')$ (with $1$'s in the other entries)
are the same, by taking the constant tuple $(h)$ on $V_1$ and taking the constant tuple $(h^{-1})$ on $V_2$.  Conversely, suppose that $(g_1,\dots,g_m), (g_1',\dots,g_m') \in G^m$ define the same double coset, where again we set $g_e = g_e' = 1$ for all $e \not\in E_0$.  That is, there are tuples $(h_v)_{v \in V_1}$ and $(h_v)_{v \in V_2}$ of elements of $G$ taking $(g_1,\dots,g_m)$ to $(g_1',\dots,g_m')$, and taking $1$ to $1$ in the other entries.  More precisely,  $g_j'=h_{v_{j,1}}g_jh_{v_{j,2}}$ for each $j=1,\dots,m$, and $1 = h_{v_{e,1}}\cdot 1 \cdot h_{v_{e,2}}$ for each $e \not\in E_0$.  Thus $h_{v_2}=h_{v_1}^{-1}$ if $v_1 \in V_1$ and $v_2 \in V_2$ are the vertices of some edge $e \not\in E_0$; i.e., some edge of $\Gamma'$.  Since $\Gamma'$ is a (connected) tree, and since the vertices of $\Gamma$ and $\Gamma'$ are the same, it follows that all the elements $h_{v_1}$ (for $v_1 \in V_1$) are equal to a common element $h \in G$, and that all the elements $h_{v_2}$  (for $v_2 \in V_2$) are equal to $h^{-1}$.  Hence $g_j' = hg_jh^{-1}$ for all $j$.
\end{proof}

By Theorem~\ref{computation-of-sha} and Lemma~\ref{decorated}, we obtain the following explicit description of $\Sha(F,G)$ under Hypothesis~\ref{rational comps ints}.  As above, $\sim$ denotes uniform conjugacy, as defined just before Lemma~\ref{decorated}. 

\begin{thm} \label{explicit Sha}
Under Hypothesis~\ref{rational comps ints}, assume that the closed fiber of the model $\XX$ is reduced and let $G$ be a reductive group over $R$ such that $\mu(G)$ is \'etale. Then
$\Sha(F,G)$ is in bijection with $(G(k)/\R)^m/\!\sim$ as pointed sets, where $m$ is the number of cycles in the reduction graph of the model.  In particular, $\Sha(F,G)$ is trivial if and only if the reduction graph is a tree or $G(k)/\R$ is trivial.
\end{thm}

\begin{proof}
Under these hypotheses, the fields $\kappa(U), \kappa(P), \kappa(\wp)$ are all equal to $k$,
and so the groups $G(\kappa(U))/\R$, $G(\kappa(P))/\R$, $G(\kappa(\wp))/\R$ are all equal to $G(k)/\R$.
Thus by Theorem~\ref{computation-of-sha},we have a bijection $\Sha(F,G) \to H^1(\Gamma,G_\bullet)$, where $\Gamma$ is the reduction graph associated to $\XX$ and $G_\bullet$ is the constant coefficient system on $\Gamma$ given by the group $G(k)/\R$.  The assertion now follows from Lemma~\ref{decorated}.
\end{proof}

It is a well-known open problem whether for any (connected) reductive group $G$ over a field $k$, the group $G(k)/\R$ is abelian.  When this holds, conjugation is trivial, and $\Sha(F,G)$ is in bijection with $(G(k)/\R)^m$.
This generalizes \cite[Theorem~6.4(c)]{CHHKPS}, which treated the case where $G$ is a torus $T$. (That result was phrased in terms of an abelian group $H^1(k,S)$ rather than $T(k)/\R$, but these are isomorphic by \cite[Th\'eor\`eme~2, p.~199]{CTS77}; see also
\cite[Theorem~3.1]{CTS87}.)

\begin{rem}  
\begin{enumerate}[(a)]
\item In Theorem~\ref{explicit Sha}, the integer $m$ is the arithmetic genus of the closed fiber, because of Hypothesis~\ref{rational comps ints}.  One may ask whether,
even without assuming Hypothesis~\ref{rational comps ints}, it remains true that $\Sha(F,G)$ is non-trivial if the arithmetic genus of the closed fiber is positive; e.g., if the closed fiber is a smooth $k$-curve of positive genus.  This is in fact not the case, since under that smoothness hypothesis the reduction graph is trivial and hence is a monotonic tree; see
Theorem~\ref{main} and Corollary~\ref{cor:char0conn}.

\smallskip

\item In combination with Theorem~\ref{equalshas}, the statement of Theorem~\ref{explicit Sha} says that under the given 
hypotheses, and for any choice of $\PP$ as in Theorem~\ref{equalshas}, there is a bijection
\[\Sha(F,G) \simeq \Sha_\PP(F, G) \simeq \Hom(\pi_1(\Gamma),G(k)/\R)/{\sim},\]
where $\Gamma$ is the reduction graph of the model, and $\sim$ is the equivalence relation induced by the action of conjugation by $G(k)/\R$ on maps to that group.  This directly parallels Corollary~6.5 of~\cite{HHK15}, which concerns linear algebraic groups $G$ over a semiglobal field~$F$, where $G$ is not assumed to be connected or reductive, but for which each connected component of the $F$-variety $G$ is assumed to be rational.  Namely, the result there says that
$\Sha_\PP(F, G) \simeq \Hom(\pi_1(\Gamma),G/G^0)/{\sim}$, where $G^0$ is the identity component of $G$.  By the rationality hypothesis on the group, $H(k)/\R$ is trivial for every component $H$ of $G$; hence $G/G^0$ is identified with the finite constant group $G(F)/\R$ and we obtain $\Sha_\PP(F, G) \simeq \Hom(\pi_1(\Gamma),G(F)/\R)/{\sim}$.  In the special case where $G$ is induced by a group over $R$ (and its components are rational), this is equivalent to saying that $\Sha_\PP(F, G) \simeq \Hom(\pi_1(\Gamma),G(k)/\R)/{\sim}$.
\end{enumerate}
\end{rem}

\section{Simply connected groups $G$ over a semi-global field $F$ with $\Sha(F,G) \neq1$} \label{sect: counterex}

In this section we give counterexamples to local-global principles for reductive groups $G$ in the context of semi-global fields $F$, by using the results of the previous two sections.  That is, we obtain examples where $\Sha(F,G)$ is non-trivial, and in some cases we compute this obstruction explicitly.

As context, we recall that if
$G$ is any $F$-rational linear algebraic group, then $\Sha_\PP(F,G)=\Sha_X(F,G)=1$ by \cite[Theorems~4.2 and 5.10]{HHK15}; and hence $\Sha(F,G)=1$ if in addition $G$ is a reductive group over $\XX$, by Proposition~\ref{sha containments}(\ref{ShaX equals Sha}).
  On the other hand, in \cite[Section~8]{CHHKPS}, we gave counterexamples to the local-global principles where the group is a torus over $R$.
In \cite{CTPS12}, it was conjectured that the local-global principle always holds for semisimple simply connected groups for a semi-global field $F$ over a $p$-adic field $K$ (in which case the residue field of $K$ has cohomological dimension 1, and $\cd(F)=3$).  This has since been proven in a number of cases involving classical groups; see \cite{preeti}, \cite{Hu}, \cite{PPS}, \cite{PS20}.

Here we show that the local-global principle does {\em not} always hold for  semisimple simply connected groups in the case of a semi-global field  over a general complete discretely valued field.  Namely, in Examples~\ref{triangle ex} and~\ref{nonmono tree}, 
we provide counterexamples to the local-global principle for  
semisimple simply connected groups $G$ over a field $k$ of cohomological dimension~4 or more, with the local-global principle being considered over 
a semiglobal field~$F$ over $K=k((t))$.  Thus $F$ has cohomological dimension at least 6 in these examples where we show that
$\Sha(F,G)$ is non-trivial.  

The fields $k$ that we consider in these examples include ones of the forms $\kappa(x,y)$ and $\kappa((x))((y))$, where $\kappa$ is either a number field or a field of the form $\kappa_0(u,v)$ and $\kappa_0((u))((v))$ for some field $\kappa_0$.
Our groups $G$ are of the form $\SL_1(D)$, where $D$ is a biquaternion algebra; here $\SL_1(D)$ is the group of elements $g$ whose reduced norm $\Nrd(g)$ (e.g., see \cite[Section~2.6]{GS}) is equal to $1$.
  
Afterwards, in Proposition~\ref{cd 4 needed}, we show that a counterexample of the type we produce cannot exist if $\cd(k)\le 3$.  Also, whereas it has been conjectured for a $p$-adic field $K$ that $\Sha(F,G)=1$ for $G$ a semi-simple simply connected group over a semi-global field $F$ over $K$, we show in Example~\ref{p-adic counterex} that the local-global principle for such $F$ and $G$ can fail if one instead considers just those discrete valuations on $F$ that are trivial on $K$, rather than all the divisorial discrete valuations on $F$.  

A major ingredient in constructing examples where $\Sha(F,G)\neq 1$ will
be Proposition \ref{dbl coset surj}  and its consequences,
Theorems~\ref{computation-of-sha} and~\ref{explicit Sha}.
Proposition~\ref{dbl coset surj}(\ref{kappa-surj}) provides a quotient set of $\Sha(F,G)$ under certain hypotheses on a model $\XX$ of $F$, and Theorems~\ref{computation-of-sha} and~\ref{explicit Sha} provide an exact computation of $\Sha(F,G)$ assuming an additional condition on $\XX$.
In those situations, to obtain counterexamples to local-global principles we will want to find examples of semisimple simply connected groups $G$ over $k$
such that $G(k)/\R \neq 1$, or such that the map $G(k)/\R \to G(k')/\R$ is
not surjective for some extension $k'$ of $k$.

\subsection{Determining $G(k)/\R$ }\label{SK1 subsec}

In order to compute $\Sha(F,G)$ by means of  
Proposition \ref{dbl coset surj}  and its consequences, we will need to determine $G(k)/\R$.  In the case that $G$ is of the form $\SL_1(D)$ for some (central) division algebra $D$ over $k$, this can be done using a result of Voskresenski\v{\i}.  Recall that if $D$ is a division algebra over a field $k$, and $G = \SL_1(D)$, then $\SK_1(D)$ is defined to be the quotient of 
$G(k) = D^{\times 1} := \{g \in D\,|\,\Nrd(g)=1\}$
by the commutator subgroup $[D^\times,D^\times]$; this is an abelian group. Voskresenski\v{\i}'s result says that the natural homomorphism $G(k) = D^{\times 1} \to \SK_1(D)$ induces an isomorphism $G(k)/\R \iso \SK_1(D)$.  (See \cite[\S 18.2]{vosk}, where this is shown using a result of Bass-Platonov; viz., that the natural map $\SK_1(D) \to \SK_1(D_{k(t)})$ is an isomorphism.)

We begin with a result that helps us compute $\SK_1(D)$ in some cases, by relating it to the group of $\R$-equivalence classes in an appropriate torus $T$.  This builds on work of Platonov, who had shown the second assertion in part~(\ref{Platonov surj iso}) below by a different argument (sketched in \cite[Section~18.3]{vosk}).  
Before stating the result, we recall some background.  

Given 
a finite Galois field extension $E/\kappa$, the {\em norm one torus} $T := \R^1_{E/\kappa}{\mathbb G}_m$ is the kernel of the norm map 
$\R_{E/\kappa}{\mathbb G}_m \to \G_m$, where $\R_{E/\kappa}\G_m$ is the Weil restriction of $\G_m$.  The group $T(\kappa)/\R$ coincides
with the Tate cohomology group $\wh H^{-1}(\Sigma,E^\times)$, where $\Sigma = \Gal(E/\kappa)$.   (See \cite[Proposition~15]{CTS77}.)  This cohomology group (and hence $T(\kappa)/\R$) is
the quotient ${}^{N}L^\times/I_{G}L^\times$, where ${}^{N}L^\times$ is the subgroup of $L^\times$ consisting of the elements of norm~1, and where $I_{G}L^\times$ is the subgroup of ${}^{N}L^\times$
generated by the elements $\sigma(x)/x$ for
$x\in L^\times$ and $\sigma \in \Sigma$.  In particular, $T(\kappa)/\R$ is abelian.
Also, if $\Sigma$ is cyclic with generator $\sigma$, and if $z \in \kappa^\times$, there is an associated cyclic algebra $(E/\kappa, \sigma, z)$, which is a central simple $\kappa$-algebra split by $E$; see \cite[Section~2.5]{GS} for details.

\begin{thm} \label{reduction to torus}
Let $\kappa$ be a field, let $K/\kappa$ and $L/\kappa$ be linearly disjoint cyclic field extensions of $\kappa$ with compositum $E=KL$, and let $T = \R^1_{E/\kappa}{\mathbb G}_m$.  
Let $\sigma$ and $\tau$ be generators of the Galois groups
$\Gal(K/\kappa)$ and $\Gal(L/\kappa)$.  Write $k=\kappa(x,y)$ and $k' = \kappa((x))((y))$, and let $D = (K(x, y)/k, \sigma, x) \otimes_k (L(x, y)/k, \tau, y)$ and $D' = D_{k'}$.  Then:
\renewcommand{\theenumi}{\alph{enumi}}
\renewcommand{\labelenumi}{(\alph{enumi})}
\begin{enumerate}
\item \label{section} 
 There are homomorphisms   $ T(\kappa)/\R \to \SK_1(D)$ and $\SK_1(D') \to T(\kappa)/\R$ such that the composition 
  $T(\kappa)/\R \to \SK_1(D) \to \SK_1(D') \to T(\kappa)/\R$ is the identity map.
\item \label{Platonov surj iso}
The maps $\SK_1(D) \to T(\kappa)/\R$ and $\SK_1(D') \to T(\kappa)/\R$ are surjective.  Moreover if $\cha(\kappa)$ does not divide $[K:\kappa][L:\kappa]$ then the latter map is an isomorphism.
\end{enumerate}  
\end{thm}

\begin{proof} 
For part~(\ref{section}), let
$n = [K :\kappa]$ and $m = [L : \kappa]$.  
Since $K$ and $L$ are linearly disjoint over~$\kappa$, $E/\kappa$ is a Galois extension of degree $nm$
and Gal$(E/\kappa)$ is an abelian group generated by~$\sigma$ and $\tau$ (where $\sigma$ and $\tau$ are treated as automorphism of $E$ in the obvious way). 
There exist commuting elements~$\pi, \delta \in D^\times$ such that $\pi^n = x$, $\delta^m = y$, and such that the restrictions of the 
 inner automorphisms  given by $\pi$ and $\delta$  to $E$ are $\sigma$ and $\tau$, respectively. We again write $\sigma, \tau$ for those inner automorphisms of $D^\times$; these automorphisms each fix both $\pi$ and $\delta$.
    
  Since $E(x, y) \subset D$ is a maximal subfield, for $\varepsilon \in  E(x, y)$ we have $\N_{E(x,y)/k}(\varepsilon) = \Nrd_D(\varepsilon)$ (see \cite[Proposition 2.6.3(2)]{GS}).
  Since $E \subset E(x, y)$,   we have  
$T(\kappa) =  \R^1_{E/\kappa}{\mathbb G}_m(\kappa) \subseteq D^{\times 1} =\SL_1(D)(k)$.  
The group $\R T(\kappa)$ is generated by the elements of the form $\alpha^{-1} \sigma^i \tau^j(\alpha)$  with $\alpha \in E^\times$ and $i, j \in {\mathbb Z}$, by \cite[Proposition 15]{CTS77} (as recalled above).
Since $\sigma$ and $\tau$ are restrictions of inner automorphisms on $D$ (by $\pi, \delta$ above), it follows that these generators are commutators.  Thus $\R T(\kappa) \subseteq [D^\times, D^\times]$, and hence we have 
a homomorphism $T(\kappa)/\R \to \SK_1(D)$.  

We now define a homomorphism $\SL_1(D')(k') \to T(\kappa)$. 
Let  $\Gamma$ be the unique maximal $\kappa((x))[[y]]$-order of $D'$ (\cite[Theorem 12.8]{reiner}). 
Then $\delta \in \Gamma$, and $\Gamma \delta $ is a 2-sided ideal 
of $\Gamma$ with $\Gamma/\Gamma \delta$ is  isomorphic to the division algebra $ D_0 := (K((x))/\kappa((x)), \sigma,  x) \otimes_{\kappa((x)) } L((x))$. Here $\sigma$ and $\tau$ induce automorphisms 
$\overline{\sigma}$ and $\overline{\tau}$ on $\Gamma/\Gamma \delta$.
Let $\Gamma_0$ be the unique maximal $L[[x]]$-order of $D_0$, and let $\overline{\pi}$ be the image of $\pi$ in $\Gamma_0$. Then $\Gamma_0\overline{\pi}$ is a 2-sided ideal of 
$\Gamma_0$, and $\Gamma_0/\Gamma_0\overline{\pi}$ is isomorphic to $E$.   

Let $\theta \in \SL_1(D')(k')$. Then $\theta \in \Gamma$, and it is a unit (\cite[p.~139]{reiner}). Let $\overline{\theta}$ be the image of $\theta$ in $\Gamma/\Gamma\delta \simeq D_0$.
Since $\Nrd(\theta) = 1$, we have  $\N_{L((x))/\kappa((x))}($Nrd$(\overline{\theta})) = 1$, and so $\Nrd(\overline{\theta}) \in L[[x]]$ is a unit.
Thus $\overline{\theta}$ lies in $\Gamma_0$ and is a unit. Let $\overline{\theta}_E$ be the image of $\overline{\theta}$ in $\Gamma_0/\Gamma_0\overline{\pi} \simeq E$.
Since $\Nrd(\theta) =1$, it follows that $\overline{\theta}_E\in \R^1_{E/\kappa}{\mathbb G}_m(\kappa) = T(\kappa)$. Hence we have a map $\SL_1(D')(k') \to T(\kappa)$ 
which is a homomorphism. 
It follows from the definition that the composite map $T(\kappa) \to \SL_1(D)(k) \to \SL_1(D')(k') \to T(\kappa)$ is the identity.

Let $f, g \in D'^\times$. Then $f = u\delta^s$ and $g = v\delta^t$ for some $u, v $ units in $\Gamma$ and $s, t \in \Z$.  
We have 
$$fgf^{-1}g^{-1} = u\delta^s v \delta^t \delta^{-s} u^{-1}\delta^{-t} v^{-1} =  u\delta^s v \delta^{-s} \delta^t  u^{-1}\delta^{-t} v^{-1} 
= u\tau^s(v) \tau^{t}(u^{-1}) v^{-1} \in\Gamma,$$ using that $\tau$ is conjugation by $\delta$. 
Writing $\ov u, \ov v$ for the images of $u,v$ in $\Gamma_0$, 
we have $\overline{u} = u_0 \overline{\pi}^{s_0}$ and $\overline{v} =  v_0 \overline{\pi}^{t_0}$ for some units 
$u_0, v_0 \in \Gamma_0$ and $s_0, t_0 \in {\mathbb Z}$. Then 
\begin{align*}\overline{fgf^{-1}g^{-1}}  &= \overline{ u\tau^s(v) \tau^{t}(u^{-1}) v^{-1}} 
= u_0 \overline{\pi}^{s_0} \overline{\tau}^s(v_0 \overline{\pi}^{t_0}) \overline{\tau}^t(\overline{\pi}^{-s_0}u_0^{-1}  ) \overline{\pi}^{-t_0}v_0^{-1}\\
&= u_0\, \overline{\sigma}^{s_0} \overline{\tau}^{s} ( v_0)\,  \overline{\tau}^{t} \overline{\sigma}^{t_0} ( u_0^{-1})\, v_0^{-1} \ \in\Gamma_0,\end{align*}
using that $\ov \sigma$ is conjugation by $\ov \pi$ and that $\ov \tau$ fixes $\ov \pi$.
Thus the images of $\overline{fgf^{-1}g^{-1}}$ and $\bigl(\overline{u}_0\, \ov\tau^t\ov\sigma^{t_0}(\overline{u}_0^{-1})\bigr) \bigl( \overline{v}_0^{-1}\,\ov\sigma^{s_0}\ov\tau^s(\overline{v}_0)\bigr)$ under $\Gamma_0 \to \Gamma_0/\Gamma_0 \overline{\pi} \simeq E$ are equal, since
$E$ is commutative.  
So the image of $fgf^{-1}g^{-1}$ under the map $\SL_1(D')(k') \to T(\kappa)$ 
is in $\R T(\kappa)$, again by \cite[Proposition 15]{CTS77}.
 Hence we  obtain a homomorphism $\SK_1(D') \to T(\kappa)/\R$ such that the composition 
  $T(\kappa)/\R \to \SK_1(D) \to \SK_1(D')\to T(\kappa)/\R$ is the identity map.  This completes the proof of part~(\ref{section}).  
  
\smallskip

The surjectivity assertion in part~(\ref{Platonov surj iso}) is immediate from part~(\ref{section}).  For the isomorphism assertion under the assumption that $\cha(\kappa)$ does not divide $nm$, let $\theta \in \SL_1(D')(k')$ with $\overline{\theta}_E  = 1 \in T(\kappa)/\R$ (where $\overline{\theta}_E$ is as in the proof of part~(\ref{section})). 
  Since $\R T(k) \subseteq [D'^\times, D'^\times] \subseteq \SL_1(D')$, after multiplying $\theta$ by an element in $[D'^\times, D'^\times]$  we may assume that 
  $\overline{\theta}_E$ is equal to~$1\in T(\kappa)$.  Since $k'$ is a complete discretely valued field whose residue field is also a complete discretely valued field, and since $\cha(\kappa)$ does not divide $nm$, by Hensel's lemma there exists  
$\theta_0 \in D'^\times$ with $\theta = \theta_0^{nm}$ and $(\overline{\theta}_0)_E = 1$.  Since $\Nrd(\theta_0)^{nm} = \Nrd(\theta) = 1$ and $(\overline{\theta}_0)_E = 1$, it then follows that 
  $\Nrd(\theta_0) = 1$. Since $\SK_1(D')$ is $nm$-torsion  (see \cite[Lemme 2.2(i)]{Tits78}), we get  $\theta  = 1 \in \SK_1(D')$. Thus the homomorphism $\SK_1(D') \to T(\kappa)/\R$ is an isomorphism.  
\end{proof}

We will use this theorem to obtain examples in which $G(k)/\R$ is non-trivial by showing that
$T(\kappa)/\R$ is non-trivial there, where $T=\R^1_{L/\kappa}\G_{m}$ with $L/\kappa$ Galois.  To do this, 
it will again be useful to describe $T(\kappa)/\R$ as the quotient of the group of norm $1$ elements of $L^\times$ by the subgroup generated by the elements $\sigma(x)/x$ for
$x\in L^\times$ and $\sigma \in \Gal(L/\kappa)$.  We will take groups $G$ of the form $\SL_1(D)$, where $D$ is taken to be a tensor product of two cyclic algebras (as in the above theorem), with a focus on the case of biquaternion algebras.

\begin{ex} \label{Rclasses ex}
\renewcommand{\theenumi}{\alph{enumi}}
\renewcommand{\labelenumi}{(\alph{enumi})}
\begin{enumerate}
\item \label{V4 -1}
Let $L/\kappa$ be a degree four field extension of the form $L = \kappa(\sqrt a, \sqrt b)$, where $-1$ is a square in $\kappa$.  Suppose in addition that the diagonal quadratic form $\langle 1,a,-b \rangle$ is anisotropic over $\kappa$; or equivalently, the quaternion algebra $(a,b)$ over $\kappa$ is not split.
(For example, we could take $\kappa= \QQ(i)$, $a=3$, $b=5$.  Or we could take $\kappa$ to be a function field of the form $\kappa_0(u,v)$ or $\kappa_0((u))((v))$ for some field $\kappa_0$, and take $a=u$ and $b=v$.)  
It was shown in \cite[Example~8.1]{CHHKPS} that ${}^{N}L^\times/I_{G}L^\times = \wh H^{-1}(\Sigma,L^\times) = T(\kappa)/\R$ is non-trivial, where $T = \R^1_{L/\kappa}(\Gm)$.  Hence so is $G(k)/\R$, where $k=\kappa(x,y)$ or $\kappa((x))((y))$, and where $G=\SL_1(D)$, with $D$ the biquaternion $k$-algebra $(a,x) \otimes (b,y)$; this is because
$G(k)/\R = \SK_1(D) \to T(\kappa)$ is surjective by Theorem~\ref{reduction to torus}(\ref{Platonov surj iso}).

\item \label{V4 nfld}
Let $L/\kappa$ be a degree four field extension of number fields of the form $L = \kappa(\sqrt a, \sqrt b)$, where we do not require $-1$ to be a square, and let $T = \R^1_{L/\kappa}(\Gm)$.  Let $d_\kappa$ be the number of places of 
$\kappa$ over which there is a unique place of $L$.  Then, as shown in the proof of \cite[Corollaire~2]{CTS77}, $T(\kappa)/\R$ is non-trivial if and only if $d_\kappa>1$, and in that case it is isomorphic to $(\Z/2\Z)^{d_\kappa - 1}$.  For example, if $\kappa=\QQ$, $a=-1$, and~$b=2$, then $d_\kappa = 1$ and $T(\kappa)/\R$ is trivial; while if $\kappa=\QQ(\sqrt{17})$, then the same choice of $a,b$ yields $T(\kappa)/\R \simeq \Z/2\Z$.  As in part~(\ref{V4 -1}), we can then take $k=\kappa(x,y)$ or $\kappa((x))((y))$; $D$ the biquaternion $k$-algebra $(a,x) \otimes (b,y)$; and $G=\SL_1(D)$.  In the case $k=\kappa(x,y)$
we then find that $G(k)/\R$ is non-trivial if $\kappa=\QQ(\sqrt{17})$; while in the case $k=\kappa((x))((y))$ we find that $G(k)/\R$ is trivial for $\kappa=\QQ$ and that it is isomorphic to $\Z/2\Z$ if $\kappa=\QQ(\sqrt{17})$.

\item \label{bicyclic gl fld}
In the notation of Theorem~\ref{reduction to torus}, assume that $[K:\kappa]=[L:\kappa]=p$ for some prime number $p$.  Let $G = \SL_1(D')$ over $k'=\kappa((x))((y))$.  Then $G(k')/\R \simeq \SK_1(D') \simeq (\Z/p\Z)^{d_\kappa - 1}$, by \cite[Theorem~5.13]{platonov}, if $d_\kappa \ge 1$ (where $d_\kappa$ is defined as in part~(\ref{V4 nfld}) above).  In the case that $p=2$ and $\kappa$ is a number field, this agrees with the conclusion in part~(\ref{V4 nfld}) above.

\item \label{torus infinite} 
For any example of a field $\kappa$ and a torus $T$ over $\kappa$ for which $T(\kappa)/\R$ is non-trivial, there exists an overfield $\tilde \kappa/\kappa$ such that $T(\tilde \kappa)/R$ is infinite.  An explicit example of this was given in \cite[Example~8.2]{CHHKPS}, by taking the union of an increasing tower of number fields $\kappa_n$ such that $d_{\kappa_n} \to \infty$ and using part~(\ref{V4 nfld}) above.  The general case was shown in \cite[Example~8.4]{CHHKPS}, by iterating \cite[Proposition~8.3]{CHHKPS}, which asserts that $T(\kappa(T))$ strictly contains $T(\kappa)$. 
In particular, this applies to the examples given in parts~(\ref{V4 -1}) and~(\ref{V4 nfld}) above where $T(\kappa)$ is non-trivial.  (Note that the group $H^1(\kappa,S)$ in \cite{CHHKPS} is the same as $T(\kappa)/\R$, by \cite[Th\'eor\`eme~2, p.~199]{CTS77}; see also
\cite[Theorem~3.1]{CTS87}.)  With $\tilde k$ equal to $\tilde \kappa(x, y)$ or $\tilde \kappa((x))((y))$, one then has that $G(\tilde k)/\R$ is infinite for $G=\SL_1(D)$ as above, since $G(\tilde k)/\R \iso \SK_1(D_{\tilde k}) \to T(\tilde \kappa)/\R$ is surjective.
\end{enumerate}
\end{ex}

Another approach to computing $\SK_1(D)$ is given in the following result, which is based on a theorem of Rost.

\begin{prop} \label{Rost prop}
Let $k$ be a field with $\cha(k)\ne 2$, and let $D$ be a biquaternion algebra $(a,b) \otimes (c,d)$ over $k$.  
\renewcommand{\theenumi}{\alph{enumi}}
\renewcommand{\labelenumi}{(\alph{enumi})}
\begin{enumerate}
\item \label{SK1 0 part}
If $\cd_2(k) \le 3$, then $\SK_1(D)=0$.
\item \label{SK1 non0 part}
If $-1$ is a square in $k$ and $(a) \cup (b) \cup (c) \cup (d)  \neq 0 \in H^4(k,\Z/2)$, then $\SK_1(D) \ne 0$.
\end{enumerate}
\end{prop}

\begin{proof}
By a result of Rost (see \cite[Theorem~4]{Merk95}), there is an injection $\SK_1(D) \to H^4(k,\Z/2)$; and if $-1$ is a square in $k$ then $\sqrt{-1} \in D^{\times 1}$, and its image 
under the composition $D^{\times 1} \to \SK_1(D) \to H^4(k,\Z/2)$ is the cup product $(a) \cup (b) \cup (c) \cup (d)$.  

Part~(\ref{SK1 0 part}) is then immediate from the injectivity of $\SK_1(D) \to H^4(k,\Z/2)$, and part~(\ref{SK1 non0 part}) holds by the non-triviality of the above composition.
\end{proof}

Part~(\ref{SK1 non0 part}) of this result can be used to obtain examples where $G(k)/\R = \SK_1(D)$ is non-trivial, where $G = \SL_1(D)$.  We would like to go further, enlarging $k$ to obtain examples where $G(k)/\R$ is infinite, as in Example~\ref{Rclasses ex}(\ref{torus infinite}).  That example had relied on \cite[Proposition~8.3]{CHHKPS}, which concerned tori.  In the context of the above proposition, we can instead use the following analogous result for reductive groups. 

\begin{prop} \label{generic infinite}
Let $G$ be a connected linear algebraic $k$-group, with $k$ perfect or $G$ reductive.
 Assume $G(k)/\R$ is non-trivial.
 Then:
\renewcommand{\theenumi}{\alph{enumi}}
\renewcommand{\labelenumi}{(\alph{enumi})}
\begin{enumerate}
\item \label{R-equiv map not onto}
The natural map
$G(k)/\R \to G\bigl(k(G)\bigr)/\R$ is injective but not surjective.
\item \label{inf R-equiv}
There is a field $\tilde k$ of infinite transcendence degree over $k$ such that $G(\tilde k)/\R$ is infinite.
\end{enumerate}
\end{prop}

\begin{proof}
If $k$ is finite, then $G(k)/\R=1$ (\cite[Corollaire~6, p.~202]{CTS77}).
Thus $k$ is infinite. Also, if two $k$-points are
$\R$-equivalent, they are directly $\R$-equivalent
\cite[II.1.1 b)]{GilleIHES}.
Since we have $k$ perfect or $G$ reductive, 
the $k$-variety $G$ is $k$-unirational by \cite[Chap.~V. Theorem~18.2]{Borel};
and the map $U(k)/\R \to G(k)/\R$ is a bijection for any non-empty Zariski open subset $U\subseteq G$, by \cite[Proposition~11, p.~196]{CTS77}.  Since $k$ is infinite and $G$ is unirational, $G(k)$ is dense in $G$, and hence $U(k)$ is non-empty.

To prove injectivity in part (\ref{R-equiv map not onto}), suppose that $g,g' \in G(k)$ become $\R$-equivalent over $k':=k(G)$.  Then they become $\R$-equivalent over some Zariski dense open subset $U\subseteq G$.  Specializing to a $k$-point of $U$ shows that $g,g'$ are $\R$-equivalent over $k$.

To prove that $G(k)/\R \to G(k')/\R$ is not surjective, suppose otherwise
for the sake of contradiction.  Then the generic point  $\eta$ in $G(k')$, which is given by the natural map
$\Spec(k') \to G$, is directly $\R$-equivalent on $G_{k'}$
to an element $g_0 \in G(k) \subseteq G(k')$.
That is, there exists an open set $S \subseteq \A^1_{k'}$
containing two $k'$-points $A$ and $B$
and a $k'$-morphism $S \to G_{k'}$ such that
$A$ maps to $g_0\times_k k' \in G(k')$ and $B$ maps to $\eta \in G(k')$.
Applying an automorphism on $\A^1_{k'}$ we may assume that
the points $A$ and $B$ come from $k$-points of $\A^1_k$.  There then
exist a dense open set $U \subseteq G$,
an open set $W \subseteq \A^1_U$ containing
$A\times U$ and $B \times U$, and a $U$-morphism
$W \to  G\times_k U$, such that
the image of  $A\times U$ is $g_0 \times V$
and the image of $B \times U$ is the diagonal in $U \times U$.

Now for any $k$-point $g\in U(k) \subseteq G(k)$,
the fiber of $W \to U$ over $g$ gives a direct $\R$-equivalence
between  the $k$-point $g$
and the point $g_0 \in G(k)$.
Thus $G(k)/\R = U(k)/\R = 1$.
This is a contradiction, thereby proving part (\ref{R-equiv map not onto}).

Part (\ref{inf R-equiv}) then follows
by infinite iteration of the passage from $k$ to $k(G)$.
\end{proof}

Combining this with Proposition~\ref{Rost prop}, we obtain the following example.

\begin{ex} \label{Rost ex}
\renewcommand{\theenumi}{\alph{enumi}}
\renewcommand{\labelenumi}{(\alph{enumi})}
\begin{enumerate}
\item \label{Rost ex nontriv}
Let $\kappa$ be a field with $\cha(\kappa)\ne 2$, such that $-1$ is a square in $\kappa$.  Let $a,b \in \kappa^\times$ be elements such that the quaternion algebra $(a,b)$ over $\kappa$ is not split (e.g., $\kappa=\QQ(i)$, $a=3$, $b=5$).  Then over the rational function field $k=\kappa(x,y)$, the biquaternion algebra $D = (a,b) \otimes (x,y)$ satisfies $\SK_1(D) \ne 0$, by Proposition~\ref{Rost prop}(\ref{SK1 non0 part}).  Thus $G(k)/\R$ is non-trivial, where $G=\SL_1(D)$.
\item \label{Rost ex inf}
In the situation of part~(\ref{Rost ex nontriv}), there is a field extension $\tilde k/k$ such that $G(\tilde k)/\R$ is infinite, by applying 
Proposition~\ref{generic infinite}(\ref{inf R-equiv}).
\end{enumerate}
\end{ex}
  
\subsection{Examples with $\Sha(F,G)\neq 1$} \label{counterexamples}

We now use the above discussion to obtain examples of semi-global fields $F$ and groups $G$ of the form $\SL_1(D)$ for which $\Sha(F,G)$ is non-trivial.  These are the first known counterexamples to the local-global principle for semisimple simply connected groups over a semi-global field.

As before, we take a semi-global field $F$ over a complete discrete valuation ring $R$ having residue field $k$, and a normal crossings model $\XX$.  
As shown in Theorem~\ref{main}, if the closed fiber of $\XX$ is reduced and the associated reduction graph $\Gamma$ is a monotonic tree, and if $G$ is a reductive group over $R$ such that $\mu(G)$ is \'etale, then $\Sha(F,G)=1$.  We therefore consider examples where either $\Gamma$ is not a tree, or else it is a tree that is not monotonic.  We restrict our attention here to the equicharacteristic case, so that $R=k[[t]]$.  
By taking $G = \SL_1(D)$ for $D$ a biquaternion algebra over $k$ (and hence over $R$ and $\XX$), we have at our disposal Examples~\ref{Rclasses ex} and~\ref{Rost ex}.  Also, in this situation, $G(k)/\R$ is isomorphic to $\SK_1(D)$, and so it is commutative. 

\medskip

We begin with an example in which the closed fiber of $\XX$ consists of three copies of $\P^1_k$ such that each pair intersects at a single point.  Thus the reduction graph is a cycle, rather than a tree; and by a suitable choice of $k$ and $D$ we can make $\Sha(F,G)$ non-trivial, and even infinite.

\begin{ex} \label{triangle ex}
As in Example~8.7 of \cite{CHHKPS}, let $k$ be a field, let $R=k[[t]]$, let $\XX = \operatorname{Proj}(R[u,v,w]/(uvw-t(u+v+w)^3))$, and let $F$ be the function field of $\XX$.  Then $\XX$ is a normal crossings model of $F$ over $R$ whose closed fiber consists of three copies of $\P^1_k$ meeting pairwise at $k$-points to form a triangle.  Thus $m=1$ in Theorem~\ref{explicit Sha}, and for any reductive group $G$ over $R$ we have a bijection from $\Sha(F,G)$ to the set of conjugacy classes $(G(k)/\R)/{\sim}$ in the group $G(k)/\R$.  Take $G=\SL_1(D)$ for 
$D$ a biquaternion algebra over $k$.  Since $G(k)/\R$ is commutative, we have a bijection from $\Sha(F,G)$ to $G(k)/\R$.  

To obtain a counterexample to the local-global principle we can thus take any of the examples with non-trivial $G(k)/\R$ in Section~\ref{SK1 subsec}.  In particular, as in Example~\ref{Rclasses ex}(\ref{V4 nfld}), we can choose $k=\QQ(\sqrt{17})((x))((y))$ and $D = (-1,x)\otimes(2,y)$, so that $\Sha(F,G) \simeq G(k)/\R \simeq \Z/2\Z$.  Alternatively, using other cases in those examples, we get a non-trivial $\Sha(F,G)$ by taking $k=\QQ(\sqrt{17})(x,y)$; or by taking $k=\kappa((x))((y))$ or $\kappa(x,y)$, where $\kappa$ is itself a field of the form $\kappa_0((u))((v))$ or $\kappa_0(u,v)$.  Thus, in particular, we obtain counterexamples to the local-global principle in the case that $k$ is of the form $\kappa_0(u,v,x,y)$ or $\kappa_0((u))((v))((x))((y))$ for some field~$\kappa_0$.

In addition, by applying Example~\ref{Rclasses ex}(\ref{torus infinite}) (or Proposition~\ref{generic infinite}) with any of these choices of $k$, we obtain a field extension $\tilde k/k$ such that $\Sha(\tilde F,G)$ is infinite, where $\tilde F = F \otimes_{k((t))} \tilde k((t))$.
\end{ex}

Our next example involves a model $\XX$ whose reduction graph $\Gamma$ is a tree that is not monotonic.

\begin{ex} \label{nonmono tree}
As in Examples~7.7 and~8.9(a) of \cite{CHHKPS}, let $k$ be a field such that $\cha(k) \ne 2$, let $R=k[[t]]$, let $c$ be a non-square in $k$, and let $F$ be the function field of $\XX = \operatorname{Proj}(R[x,y,z]/((y-x)(xy - cz^2) + tz^3))$.  Then $\XX$ is a normal crossings model of $F$ over $R$ whose closed fiber is reduced and consists of two copies of $\P^1_k$ that meet at a single closed point $P$ whose residue field is $k':=k(\sqrt c)$. 
Thus the reduction graph is a tree, but is not a monotonic tree (and the base change of the closed fiber from $k$ to $k'$ is not a tree).
Let $U_1,U_2$ be the complements of $P$ in the two copies of $\P^1_k$, and for $i=1,2$ let $\wp_i$ be the branch at $P$ lying on $U_i$.
With notation as at the beginning of Section~\ref{sectionlb}, we have $\kappa(U_{i})=k$ and $\kappa(P)=\kappa(\wp_i) = k'$ for $i=1,2$.  
By Proposition~\ref{dbl coset surj}(\ref{kappa-surj}), for any reductive group $G$ over $R$
we have a surjection
\begin{eqnarray*}
\Sha(F,G) &\to& \prod_{U \in \UU} (G(\kappa(U)/\R)  \backslash \prod_{\wp \in \BB}  (G(\kappa(\wp))/\R)  \slash \prod_{P \in \PP} (G(\kappa(P))/\R)\\
&=& (G(k)/\R \times G(k)/\R) \backslash (G(k')/\R \times G(k')/\R)  \slash (G(k')/\R).
\end{eqnarray*}
Here the maps $G(k)/\R \times G(k)/\R \to G(k')/\R \times G(k')/\R$ on the left hand side are the natural maps
induced by the inclusion $k \subset k'$, and the map $G(k')/\R \to G(k')/\R \times G(k')/\R$ on the right hand side is the diagonal map.

To use this to obtain an explicit counterexample to the local-global principle we can let $k=\QQ((x))((y))$ and $c=17$, so that $k' = \QQ(\sqrt{17})((x))((y))$.
Let $D$ be the biquaternion algebra $(-1,x)\otimes(2,y)$ over $k$ and let $G=\SL_1(D)$.  By 
Example~\ref{Rclasses ex}(\ref{V4 nfld}), $G(k)/\R$ is trivial, while $G(k')/\R = \Z/2\Z$.  Thus $\Sha(F,G)$ maps surjectively onto the quotient of $(\Z/2\Z)^2$ by the diagonal $\Z/2\Z$, and hence is non-trivial.  As in Example~\ref{triangle ex}, we can then find a field extension $\tilde k/k$ such that $\Sha(\tilde F,G)$ is infinite, where $\tilde F = F \otimes_{k((t))} \tilde k((t))$.
\end{ex}
 
Observe that in all our examples of non-trivial $\Sha(F,G)$ and non-trivial $G(F)/\R$ for $G=\SL_1(D)$ and $F$ a semi-global field over $k((t))$, we have $\cd(k) \ge 4$.  
Namely, in our examples, $k=\kappa(x,y)$ or $\kappa((x))((y))$, where $\kappa$ is a number field or a field of the form $\kappa_0(u,v)$ or $\kappa_0((u))((v))$.
In fact, in using our approach, this condition on the cohomological dimension is necessary for examples where Theorem~\ref{computation-of-sha} applies:

\begin{prop} \label{cd 4 needed}
Let $k$ be a field, with $\cha(k)\neq 2$. Let $D$ be a biquaternion division algebra over $k$, and 
let $G=\SL_1(D)$.
Let $R=k[[t]]$, and let $F$ be a semi-global field over $R$ that 
admits a normal crossings model $\XX$ whose closed fiber is reduced and consists of
projective $k$-lines meeting at $k$-points. If $\cd_{2}(k) \leq 3$, then $\Sha(F,G)=1$.
 \end{prop}

\begin{proof}
By Proposition~\ref{Rost prop}, $\SK_1(D)$ is trivial.  But  $G(k)/\R \simeq \SK_1(D)$ by \cite[\S 18.2]{vosk}.  Hence $G(k)/\R$ is also trivial.  Thus $\Sha(F,G)$ is trivial, by the description given in 
Theorem~\ref{computation-of-sha}.
\end{proof}

\begin{rem} \label{cd 3 rk}
Proposition~\ref{cd 4 needed} can be extended as follows:
\renewcommand{\theenumi}{\alph{enumi}}
\renewcommand{\labelenumi}{(\alph{enumi})}
\begin{enumerate}
\item \label{cd3 char0}
If $\cha(k)=0$, then the conclusion of Proposition~\ref{cd 4 needed} holds even without assuming that the biquaternion algebra $D$ is a division algebra, because Voskresenski\v{\i} showed in \cite[\S 18.2]{vosk} that the isomorphism $G(k)/\R \simeq \SK_1(D)$ remains true for general central simple algebras in that situation.  Hence in that case, we additionally obtain that $G(k')/\R$ 
is trivial for every finite extension of $k$.  Thus we conclude that 
$\Sha(F,G)$ is trivial even if we assume only that the components of the closed fiber of $\XX$ are projective lines (not necessarily over $k$, or meeting at $k$-points).
\item \label{cd3 cyclic alg}
Let $k'=\kappa((x))((y))$
and let $D'$ be a product of two cyclic algebras as in Theorem~\ref{reduction to torus}.
Assume that  their degrees are  not divisible by the characteristic. If  $\cd(k') \le 3$,
then $\cd(\kappa) \le 1$, hence $T(\kappa)/\R$ is trivial for any torus $T$ over $\kappa$ (by \cite[Corollaire~5(i), p.~201]{CTS77}).  Thus $G(k)/\R \simeq \SK_1(D) \simeq T(\kappa)/\R$ is trivial, with $T$ as in Theorem~\ref{reduction to torus}(\ref{Platonov surj iso}).  Hence $\Sha(F,G)$ is trivial, and so Proposition \ref{cd 4 needed} extends to this setting.
\end{enumerate}
\end{rem}

The above proposition and remark provide examples in which the local-global principle for $F$ and $G$ holds in the situation where $\cd(k) \le 3$.  This suggests the question of whether $\Sha(F,G)$ is trivial for every semisimple simply connected group $G$ over a complete discrete valuation ring $R$ whose residue field $k$ satisfies $\cd(k) \le 3$, where $F$ is a semi-global field over~$R$.

\smallskip

On the other hand, instead of considering the local-global obstruction set $\Sha(F,G)$, which is taken with respect to all the divisorial discrete valuations on $F$ arising from regular models $\XX$ over $R$, one could consider the {\it a priori} larger obstruction set $\Sha_K(F,G)$ taken with respect to the set $\Omega_K$ of discrete valuations on $F$ that are trivial on $K$ (or equivalently, the ones that arise from closed points on the generic fiber of $\XX \to \Spec(R)$).  This obstruction was considered, for example, in \cite{CH}, \cite{HSS}, and \cite{HS16}.  As the next example shows, the associated local-global principle can fail for semisimple simply connected groups even in the case of semi-global fields over a $p$-adic field.  This answers a question posed by D.~Harari.

\begin{ex} \label{p-adic counterex}
Let $p$ be an odd prime, and let $k=\mbb F_p$, $R = \Z_p$, and $K=\QQ_p$.  Consider the projective $R$-curve $\XX = \operatorname{Proj}(R[u,v,w]/(uvw-p(u+v+w)^3))$, with function field $F$.  Thus $\XX$ is a normal crossings model of the semi-global field $F$ over $R$, whose closed fiber consists of three copies of $\P^1_k$, forming a triangle.  (This is a mixed characteristic analogue of the model in Example~\ref{triangle ex}.)  Let $C$ be the fiber of $\XX$ over the generic point of $\Spec(R)$.  Thus $C$ is a smooth projective curve over $K=\QQ_p$, with function field $F$.
By \cite[Lemma~5.10(a),(b)]{HKP}, $\Sha^3_K(F,\Z/2\Z) = \Sha_K(F,\Z/2) = \Z/2\Z$.  
(This can also be obtained using \cite[Corollary~2.9]{Kato}.  We note that \cite{HKP} indirectly relied on \cite{Kato} via its use of \cite{HS16}.)
By \cite[Theorem~3.9]{PS99}, the unique non-trivial element $\xi$ of $\Sha^3_K(F,\Z/2\Z) \subseteq H^3(F,\Z/2)$ is decomposable; i.e., is represented by a symbol $(a)\cup (b) \cup (c)$, for some $a,b,c \in F^\times$ (where we identify $H^1(F,\Z/2\Z)$ with $F^\times/F^{\times 2}$).  Since $\xi \in \Sha^3_K(F,\Z/2\Z)$, its image $\xi_v \in H^3(F_v,\Z/2\Z)$ is trivial, for all $v \in \Omega_K$.

If $G$ is a simple $F$-group of type $G_2$, then for field extensions $L/F$ there is a functorial injective map of pointed sets $\iota_L: H^1(L,G) \to H^3(L,\Z/2)$ whose image is the set of decomposable elements (see \cite[Section 8.1, Th\'eor\`eme~9]{Serre:Bourb}).  
Since the above decomposable element $\xi \in \Sha^3_K(F,\Z/2\Z)$ is non-trivial, it is the image of some non-zero element $\zeta$ of $H^1(F,G)$.  The image $\zeta_v$ of $\zeta$ in $H^1(F_v,G)$ is trivial for all $v \in \Omega_K$, since the injective map $\iota_{F_v}: H^1(F_v,G) \to H^3(F_v,\Z/2\Z)$ sends $\zeta_v$ to $\xi_v = 0$.  Thus $\zeta$ is a non-trivial element of $\Sha_K(F,G)$.

Alternatively, with $a,b,c$ as above, let $D$ be the quaternion algebra $(b,c)$ over $F$, and let $G = \SL_1(D)$.  The Brauer class of $D$ is $2$-torsion and lies in $H^2(F,\Z/2\Z)$, where it may be identified with the cup product $(b) \cup (c)$; here $(b), (c) \in F^\times/F^{\times 2} = H^1(F,\Z/2\Z)$.  For each $L/F$, we may consider the
map $H^1(L,\Z/2\Z) = L^\times/L^{\times 2} \to H^3(L,\Z/2\Z)$ given by $s \mapsto s \cup D_L$.
By a theorem of Merkurjev-Suslin (see \cite[Section~7.2, Th\'eor\`eme~8]{Serre:Bourb}), the kernel of this map consists of the square classes of elements of $\Nrd(D_L^\times) \supseteq L^{\times 2}$;
and so there is an induced injection $\alpha_L:H^1(L,G) = L^\times/\Nrd(D_L^\times) \to H^3(L,\Z/2\Z)$.
The class of $a \in F^\times$ is sent by $\alpha_F$ to the non-zero element $\xi$, and the class of its image $a_v \in F_v^\times$ is sent by the injective map $\alpha_v:H^1(F_v,G) \to H^3(F_v,\Z/2\Z)$ to $\xi_v=0$ for all $v \in \Omega_K$.  Hence the class of $a$ in $H^1(F,G)$ is a non-trivial element of $\Sha_K(F,G)$. 
\end{ex}
     
\appendix

\section{Specialization in dimensions one and two} \label{appendix}

Let $A$ be a regular local ring with fraction field $L$ and residue field $\ell$.  Let $G$ be a reductive group over $A$. In \cite[Th\'eor\`eme~0.2]{GilleTAMS}, under the hypothesis that $A$ is a discrete valuation ring whose residue characteristic is not $2$,
P.~Gille constructs a {\it specialization map} $G(L)/\R \to G(\ell)/\R$ on $\R$-equivalence classes.  This
is a group homomorphism such that the composition $G(A) \to G(L)/\R \to G(\ell)/\R$ agrees with the restriction map $G(A) \to G(\ell)$ composed with the canonical map $G(\ell) \to G(\ell)/\R$.
In this appendix, we show that such a specialization map may be defined for all regular local rings $A$ of dimension one or two, with no restriction on the residue characteristic.
We do this first in the case of dimension one (i.e., $A$ is a discrete valuation ring), in Theorem~\ref{dvr specialization}, and then we use that case to treat the case of dimension two, in
Theorem~\ref{main appendix}.  Before turning to this construction, we state some preliminary results concerning tori and anisotropic groups.

\subsection{Preliminaries on tori and anisotropy}

Below we consider a reductive group $G$ over a regular local ring $A$, and we relate the split tori in $G$ over $A$ to those over the fraction field and residue field of $A$.  We also discuss criteria for isotropy; i.e., for the existence of split tori.

\begin{prop} \label{split tori lift}
Let $G$ be a reductive group over a normal domain $A$.  Assume either that $A$ is complete with respect to an ideal $I$, or that $A$ is a Henselian local ring with maximal ideal $I$. 
\renewcommand{\theenumi}{\alph{enumi}}
\renewcommand{\labelenumi}{(\alph{enumi})}
\begin{enumerate}
\item \label{split lifting}
Every split torus in $G_{A/I}$ lifts to a split torus in $G$.
\item \label{max rank}
The maximal ranks of split tori in $G$ and $G_{A/I}$ agree.
\end{enumerate}
\end{prop}

\begin{proof}
For part~(\ref{split lifting}), let $\ov T$ be a split torus in $G_{A/I}$, say of rank $n$.
Since $G$ is smooth and since groups of the form $\Gm^n$ are of multiplicative type, \cite[Exp.~XI, Corollaire~4.2]{SGA3.2} applies, and asserts that the functor $\underline{\rm Hom}_
{A{\rm-groups}}(\Gm^n,G)$ is representable by a smooth separated $A$-scheme $Z$.  Thus 
$\ov T$ corresponds to an $A/I$-point $\ov\zeta$ of $Z$.

We claim that $\ov \zeta$ lifts to an $A$-point $\zeta$ of $Z$.  If $A$ is $I$-adically complete then this follows from formal smoothness.  If $A$ is a Henselian local ring with maximal ideal $I$,
then for some affine neighborhood $Y \subseteq Z$ of the closed point $\bar\zeta$, the structure morphism $Y \to \Spec(A)$ factors through an \'etale morphism $Y \to\A^N_A$ for some $N$, by smoothness; and so a lift $\zeta$ exists since $A$ is Henselian.

The above lift $\zeta$
corresponds to a lift of $\ov T$ to a split torus $T$ in $G$.  This proves part~(\ref{split lifting}), and then part~(\ref{max rank}) is an immediate consequence.
\end{proof}

Recall that a reductive group $G$ over a ring $A$ is {\it anisotropic} if it does not contain an isomorphic copy of $\G_{{\rm m},A}$.  For an affine group scheme $G$ over $A$ and an $A$-algebra $B$, we write $G_B$ for $G \times_A B$; this is an affine group scheme over $B$.

\begin{prop} \label{local anisotropy}
Let $G$ be a reductive group over a regular local ring
$A$ (of arbitrary dimension) with fraction field $L$ and residue field $\ell$.  
Consider the reductive groups $G_L$ over $L$ and $G_\ell$ over $\ell$.
\renewcommand{\theenumi}{\alph{enumi}}
\renewcommand{\labelenumi}{(\alph{enumi})}
\begin{enumerate}
\item \label{ell to L}
If $G_\ell$ is anisotropic, then $G_L$ is anisotropic.
\item \label{L to A}
If $G_L$ is anisotropic, then $G$ is anisotropic.
\item \label{Hens equiv}
If $A$ is Henselian, then these three conditions are equivalent: $G$ is anisotropic, $G_L$ is anisotropic, $G_\ell$ is anisotropic.
\end{enumerate}
\end{prop}

\begin{proof}
For~(\ref{ell to L}), assume that $G$ is isotropic over $L$.  Let
$x \in A$ be part of a regular system of parameters of $A$, and consider the completion 
$\wh A_x$ of the local ring of $A$ at the corresponding height one prime $\mathfrak p$.  Thus $\wh A_x$ is a discrete valuation ring, whose fraction field is the completion $L_x$ of $L$ with respect to the $\mathfrak p$-adic valuation. 
Since the group $G$ is isotropic over $L$, it is also isotropic over the overfield $L_x$; 
and since $G$ is reductive it is then also isotropic over the discrete valuation ring $\wh A_x$, by \cite[Section~3, Lemma~4]{Guo}.
Thus $G$ is isotropic over the residue field of $\wh A_x$, which is the fraction field of the regular local ring $A/xA$ of dimension one less than that of $A$. By induction on
the dimension of $A$, we conclude that $G$ is isotropic over $\ell$.

Part~(\ref{L to A}) is immediate, since if $G$ is isotropic over $A$, then it is trivially isotropic over $L$.

Finally, if $A$ is Henselian and $G$ is isotropic over $\ell$ then $G$ is also isotropic over $A$ by Proposition~\ref{split tori lift}(\ref{split lifting}).  So by parts~(\ref{ell to L}) and~(\ref{L to A}), part~(\ref{Hens equiv}) follows.
\end{proof}

\begin{prop}\label{BTR} Let $A$ be a discrete valuation ring with fraction field
$L$ and residue field $\ell$. Let $G$ be a reductive group over $A$. 
If $G_\ell$ is anisotropic then $G(A)=G(L)$; and if $G(A)=G(L)$ then $G$ and $G_L$ are anisotropic.
\end{prop}

\begin{proof}
For the first assertion, note that if $G_\ell$ is anisotropic then so is $G_L$, by Proposition~\ref{local anisotropy}.  In the case that $A$ is Henselian, the equality $G(A)=G(L)$ then essentially follows from a result of Bruhat-Tits-Rousseau (see \cite[pp.~197-198]{Pra}).  Namely, under the Henselian hypothesis
they showed that $G_L$ is anisotropic
if and only if $G(L)$ is bounded.  But by 
\cite[Theorem~1.1]{Maculan}, the subgroup $G(A)$ is a maximal bounded subgroup of $G(L)$; hence $G(A)=G(L)$.  (For another proof, see \cite[Section~3, Proposition~6]{Guo}.)

In the general case, where the discrete valuation ring $A$ is not assumed Henselian, let $\wh L$, $\wh A$ be the completions of $L$ and $A$ respectively.  Since $\ell$ is the residue field of $\wh A$, it follows that $G(\wh A) = G(\wh L)$ by the previous case; and so $G(A) = G(\wh A) \cap G(L) = G(\wh L) \cap G(L) = G(L)$.

For the second assertion, assume that $G$ is isotropic over $A$ or $L$. (For a discrete valuation ring, those conditions are equivalent by \cite[Section~3, Lemma~4]{Guo}.)
Then there is a closed immersion $\Gm \hookrightarrow G$ over $A$.  But $\Gm(A)$ is strictly contained in $\Gm(L)$; and $\Gm(A) = G(A) \cap \Gm(L)$ since $\Gm \subseteq G$ is a closed immersion.  So $G(A)$ is indeed strictly contained in~$G(L)$.
\end{proof}

Thus if $A$ is a Henselian discrete valuation ring 
with fraction field $L$ and $G$ is reductive over $A$, then the condition $G(A)=G(L)$ is equivalent to each of the three anisotropy conditions in Proposition~\ref{local anisotropy}(\ref{Hens equiv}).

\begin{prop} \label{H and S}
Let $G$ be a reductive group over a regular local ring $A$ with fraction field $L$.  Let $S \subseteq G$ be a maximal split torus in $G$, and let $H = C_G(S)$ be its centralizer.
\renewcommand{\theenumi}{\alph{enumi}}
\renewcommand{\labelenumi}{(\alph{enumi})}
\begin{enumerate}
\item \label{aniso quotient}
Then the group $H/S$ is anisotropic over $L$.
\item \label{almost anisotropic factorization}
If $A$ is a Henselian discrete valuation ring, then $H(L) = H(A)S(L)$.
\end{enumerate}
\end{prop}

\begin{proof}
Part~(\ref{aniso quotient}) will follow from showing that a central extension of a split torus by a split torus, in the category of groups, is itself a split torus.  But any such extension is in fact a group of multiplicative type, by \cite[
Exp.~XVII, Proposition~7.1.1]{SGA3.2}; and hence it is also an extension of split tori in the category of groups of multiplicative type.  
Such extensions are classified by $\Ext^1(\Z^r,\Z^s)$; 
i.e., extensions of locally
constant sheaves in the \'etale topology on $\Spec(A)$ (where $r,s$ are the ranks of the two split tori).  But this group is trivial, thereby yielding the assertion in part~(\ref{aniso quotient}).

For part~(\ref{almost anisotropic factorization}), observe first that $H/S$ is anisotropic over $L$ by part~(\ref{aniso quotient}), and hence also over the residue field, by Proposition~\ref{local anisotropy}(\ref{Hens equiv}).  Thus $(H/S)(L)=(H/S)(A)$ by Proposition~\ref{BTR}.  Since $S$ is split and $A$ is local, it follows that 
$H^1(A, S) = 0$; and so the map $H(A) \to (H/S)(A)$ is surjective.  Thus for every $h \in H(L)$ there is some $h' \in H(A)$ whose image in $(H/S)(A)=(H/S)(L)$ agrees with that of $h$.
That is, $h = h's$ for some $s \in S(L)$, as desired.
\end{proof}

We note that in the statement of Proposition~\ref{H and S}, the centralizer is taken in a functorial sense, in the context of presheaves on the category of schemes (see \cite[Exp.~XXVI, Section~6]{SGA3.3}).  As a result, the formation of centralizers commutes with base change.  In particular, in the situation of the proposition, if $B=A/\frak p$ for some prime ideal $\frak p$, then $H_B$ is the centralizer of $S_B$ in $G_B$.

\subsection{Specialization in dimension one} \label{spcl dim 1}

Before turning specifically to discrete valuation rings, we 
begin with a general lemma.  This lemma applies in particular to complete regular local rings, since such rings are excellent by \cite[Chap~IV, Scholie~7.8.3(iii)]{EGA4b}.

\begin{lem} \label{context map}
Let $X,Y$ be schemes of finite type over an excellent regular local ring $A$, with $X$ regular and connected of dimension two,
and with $Y$ proper over $A$.  Let $g:X \dashrightarrow Y$ be a rational map of $A$-schemes.  Then there is a morphism $\pi:\til X \to X$ obtained by a sequence of blow-ups at closed points, together with a morphism $\til g:\til X \to Y$, such that $g \circ \pi = \til g$ as rational maps.
\end{lem}

\begin{proof}
Since $A$ is excellent, so are $X$ and $Y$, by \cite[Chap~IV, Proposition~7.8.6(i)]{EGA4b}.  Let $\til X'$ be the closure of the graph   of $g$ in $X \times_{A} Y$, with projection maps $\alpha:\til X' \to X$ and $\gamma:\til X' \to Y$.  Thus $g \circ \alpha = \gamma$ as rational maps; and $\alpha$ is proper, since $Y \to \Spec(A)$ is.  
By excellence, 
\cite[\href{https://stacks.math.columbia.edu/tag/0ADX}{Section 0ADX}]{stacks} and  \cite[\href{https://stacks.math.columbia.edu/tag/0BGP}{Theorem 0BGP}]{stacks} apply, and assert that there exists 
a resolution $\beta:\til X \to \til X'$ of the singularities of $\til X'$ (which is in particular proper).  Thus $\pi=\alpha \circ \beta:\til X \to X$ is a proper birational morphism; and so by 
\cite[\href{https://stacks.math.columbia.edu/tag/0C5R}{Lemma 0C5R}]{stacks} it 
is a composition of blowups at closed points. 
The composition $\til g := \gamma \circ \beta$ is a morphism $\til X \to Y$, and $g \circ \pi = 
g \circ \alpha \circ \beta = \gamma \circ \beta = \til g$ as rational maps.
\end{proof}

In the above situation, we view the residue field $\ell$ of $A$ as an $A$-algebra, and so we write $X(\ell)$ for $\Hom_A(\Spec(\ell),X)$.  Note that every $\ell$-point of $X$ lies over the closed point of $\Spec(A)$; i.e., $X(\ell)$ is contained in the closed fiber $X_\ell$ of $X \to \Spec(A)$. Thus $X(\ell)=X_\ell(\ell)$, and 
we write $X(\ell)/\R$ to mean $X_\ell(\ell)/\R$.

\begin{lem} \label{R trivial fibers}
Let $A$ be an excellent regular local ring with residue field $\ell$,
and let $X$ be a regular connected two-dimensional scheme of finite type over $A$. 
If $\til X \to X$ is a
sequence of blow-ups at closed points, then  the natural map  $\til X
(\ell)/\R \to X(\ell)/\R$ is a bijection.
\end{lem}

\begin{proof}
Let $X_\ell$ be the fiber of $X$ over the closed point of $\Spec(A)$.  If $X_\ell$ is empty then the result is trivial; so we assume that $X_\ell$ is non-empty.

Note that any two $\R$-equivalent points of $\til X(\ell)$ map to $\R$-equivalent points of $X(\ell)$, and so there is a well defined map on $\R$-equivalence classes.  To prove that this map is bijective, by induction we are reduced 
to the case that $\til X \to X$ is a blow-up at a single closed point $P$ of $X$. If $P$ has residue field larger than $\ell$, then $\til X(\ell) = X(\ell)$ and there is nothing to prove.  In the
case that $P$ is an $\ell$-point, the exceptional divisor $E$ is isomorphic to $\P^1_{\ell}$; thus $\til X(\ell) \to X(\ell)$ is surjective, and hence so is $\til X(\ell)/\R \to X(\ell)/\R$.  

For injectivity, suppose that distinct points $Q,Q' \in \til X(\ell)$ have $\R$-equivalent images in $X(\ell)$ under $\pi:\til X \to X$.  If $\pi(Q)=\pi(Q') \in X(\ell)$, then that common image is $P$, and the points $Q,Q'$ both lie in $E(\ell) \cong \P^1(\ell)$ and hence are $\R$-equivalent.  So we may assume that $\pi(Q),\pi(Q') \in X(\ell)$ are distinct $\R$-equivalent points.  Thus there is
a chain of $\ell$-points $P_0=\pi(Q)$, $\dots$, $P_n=\pi(Q')$ of $X$,
with $\ell$-morphisms $f_i: U_i \to X_\ell$ directly linking $P_i$ to $P_{i+1}$, where each $U_i \subset \P^1_{\ell}$ is open.
Here each $P_i$ is of the form $\pi(Q_i)$ for some point $Q_i \in \til X(\ell)$, since $\til X(\ell) \to X(\ell)$ is surjective.
We are thus reduced to the case of two points $Q,Q' \in \til X(\ell)$ whose images in $X(\ell)$ are distinct and directly $\R$-equivalent.  In this situation there is a  single $\ell$-morphism $U \to X_\ell$
from an open subset $U \subseteq \P^1_{\ell}$ to $X_\ell$, having $\pi(Q)$ and $\pi(Q')$ in its image.  

By the valuative criterion for properness, there is a lift of the morphism $U \to X_\ell$ to a morphism $U \to \til X_\ell$.  If $\pi(Q),\pi(Q') \in X$ are each unequal to the blown-up point $P$, then over each of these two points there is a unique point of $\til X$; viz., $Q$ and $Q'$ respectively.  Thus $Q$ and $Q'$ lie in the image of $U \to \til X_\ell$ and so 
are $\R$-equivalent.  On the other hand, if one of these two points of $X$, say $\pi(Q')$, is the blown-up point $P$, then the image of $U \to \til X_\ell$ must contain $Q$ and an $\ell$-point $Q''$ of the exceptional divisor $E \cong \P^1_\ell$ (namely, the point of the image that lies over $\pi(Q')$).  Thus $Q,Q'' \in \til X_\ell(\ell)$ are $\R$-equivalent.  Also,
$Q',Q'' \in E(\ell) \cong \P^1(\ell) \subset \til X_\ell(\ell)$ are directly $\R$-equivalent.  So $Q,Q' \in \til X_\ell(\ell)$  are $\R$-equivalent.
\end{proof}

Given a two-dimensional regular local ring $\mc O$, by a {\em regular} height one prime $\wp$ in $\mc O$ we mean a prime of the form $(\pi)$, where $\pi$ is part of a regular system of parameters $\{\pi,\delta\}$ for $\mc O$ at its maximal ideal $\mathfrak m_{\mc O}$.  Equivalently, $\wp$ has the property that the one-dimensional local ring $\mc O/\wp$ is regular; i.e., a discrete valuation ring.  

If $A$ is a local ring with residue field $\ell$ and $G$ is a group scheme over $A$,
then for any $g \in G(A)$ we write $\ov g \in G(\ell)$ for the image of $g$ under the reduction map $G(A) \to G(\ell)$.

\begin{prop} \label{R special}
Let  $A$ be a 
discrete valuation  ring with fraction field $L$ and residue field $\ell$.
Let $G$ be a reductive group
over $A$ such that $G_\ell$ is anisotropic.
If   $g_0, g_1 \in G(A)$ are $\R$-equivalent as elements of $G(L)$, then their images $\ov g_0, \ov g_1 \in G(\ell)$ are also $\R$-equivalent.
\end{prop}

\begin{proof}
Let $\wh A$ be the completion of $A$, with fraction field $\wh L$.  Thus $A, \wh A$ both have residue field~$\ell$, and the reduction map $G(A) \to G(\ell)$ factors through $\wh A$.  The 
inclusion $G(A) \hookrightarrow G(\wh A)$ induces an inclusion $G(L) \hookrightarrow G(\wh L)$
that takes $\R$-equivalent points to $\R$-equivalent points.  
Hence it suffices to prove the result with $A$ replaced by $\wh A$.  We are therefore reduced to the case that $A$ is complete.  

As $G_L$ is a linear algebraic group over the field $L$, $\R$-equivalent points $g_0,g_1 \in G(A) \subseteq G(L)$ are directly 
$\R$-equivalent over $L$, via some rational map $g_L: \A^1_L \dashrightarrow G_L$ taking $0_L \in \A^1(L)$ to $g_0$ and $1_L \in \A^1(L)$ to $g_1$ (see \cite[Lemme~II.1.1(b)]{GilleIHES}).
We can also view $g_L$ as a rational map $\A^1_A \dashrightarrow G$.  Let $\til G$ be the closure of $G \subseteq \GL_{n,A}$ in $\P^{n^2}_A$, and let $\pi:\til X \to X := \A^1_A$ and $\til g: \til X \to \til G$ be as in Lemma ~\ref{context map}, taking $Y = \til G$ there (that lemma applies because $A$ is complete). 
Thus $\til X \to X$ is a succession of blow-ups at closed points, and the morphism $\til g:\til X \to \til G$ lifts the rational map $g_L:X \dashrightarrow G$.
Hence $\til g$ and $g_L \circ \pi$ agree as rational maps.  In particular, for $i=0,1$, we have $\til g(i_L) = g_i \in G(L)$.

Since $\til X \to X$ is an isomorphism away from the closed fiber, we may identify $\til X_L$ with $X_L = \A^1_L$.
Let $\til 0$, $\til 1 \in \til X (A)$ be the unique lifts of the points
$0_L, 1_L \in \til X_L(L) = \A^1(L)$ (corresponding to their closures in $\til X$), and let $P_0, P_1 \in \til X(\ell)$ be their images under $\til X(A) \to \til X(\ell)$.  
Since any two $\ell$-points of $X_\ell = \A^1_\ell \subset X= \A^1_A$ are $\R$-equivalent in $X_\ell$, it follows from Lemma~\ref{R trivial fibers} that the same holds for two $\ell$-points of $\til X_\ell$; 
and in particular, $P_0,P_1$ are $\R$-equivalent in $\til X(\ell)$.  

Fix $i\in\{0,1\}$.  Then the morphism $g_i:\Spec(A) \to G \subset \til G$ and the composition $\Spec(A)\, {\buildrel \til{i} \over \to}\, \til X \,
{\buildrel \til g \over \to}\,  \til G$ agree at the generic point $\Spec(L)$ of $\Spec(A)$, because $\til g(i_L) = g_i \in G(L)$.  Therefore these two morphisms agree, since $A$ is a discrete valuation ring and $G$ is separated.  Hence the closed point of $\Spec(A)$ has the same image under these two morphisms.  That is, 
\begin{equation} \label{reduction display}
\til g(P_i) = \ov g_i \in G(\ell)
\end{equation}
for $i=0,1$. 

We claim, moreover, that for {\em every} $P \in \til X(\ell)$, the point $\til g(P) \in \til G(\ell)$ lies in $G(\ell)$.  To see this, first choose a regular height one prime $\wp$ in the local ring $\ms O := \ms O_{\til X,P}$  such that $\til g(\wp) \in G$, where we view $\wp \in \Spec(\mc O)$.  
(For example, for any regular system of parameters $x,y$ and any $n \ge 0$, the principal ideal $(x+y^n)$ is a regular height one prime $\wp$ of $\mc O$, and for all but finitely many $n$ the point $\wp \in \Spec(\mc O)$ lies in the dense open subset $\til g^{-1}(G)$.)
Let $B$ be the completion of the discrete valuation ring $\ms O/\wp$, and let $E$ be the fraction field of $B$.  Thus $\til g(\wp) \in G(E)$.  Since $G_\ell$ is anisotropic, 
$G(E)=G(B)$ by Proposition~\ref{BTR}.  That is, the rational map 
$\Spec(B) \dashrightarrow G$
 extends to a morphism $\Spec(B) \to G$
giving a commutative diagram:
\begin{equation} \label{diamond diagram}
\xymatrix @R=.2cm @C=.3cm{
& \Spec(B) \ar[ld] \ar[rd] \ar[dd] \\
\til X \ar[dr]_{\til g} & &  G \ar@{_(->}[dl] \\
& \til G
}
\end{equation}
As the closed point of $\Spec(B)$ maps to $P \in \til X(\ell)$ (in the left hand side of (\ref{diamond diagram})) and maps to a point of $G(\ell) \subseteq \til G(\ell)$ (in the right hand side),
it follows from commutativity of the diagram that $\til g(P) \in G(\ell)$ as claimed.

Since  $P_0,P_1$ are $\R$-equivalent in $\til X(\ell)$, and since $\til g(P_i) = \ov g_i \in G(\ell)$ for $i=0,1$ (by display~(\ref{reduction display})), in order to conclude the proof it suffices to show that $\R$-equivalent points in  $\til X(\ell)$ have $\R$-equivalent images in $G(\ell)$ under $\til g$.  For this, it suffices to consider points $P,P' \in \til X(\ell)$ that are directly $\R$-equivalent; i.e., such that there is a rational map $\phi: \A^1_\ell \dashrightarrow X$ connecting $P$ and $P'$.  By the previous paragraph, 
$P,P' \in \til g^{-1}(G)(\ell)$; and so the generic point of $\mbb P^1_\ell$ is also mapped by $\phi$ into $\til g^{-1}(G)$.  We thus obtain a rational map $g \circ \phi : \A^1_\ell \dashrightarrow G$, showing that $\til g(x), \til g(x') \in G(\ell)$ are (directly) $\R$-equivalent.
\end{proof}
 
\begin{prop}[Decomposition of elements in reductive groups] \label{decomposition}
Let $A$ be a regular local ring of dimension at most $2$ and let $G$ be a reductive group over $A$.  Let $S \subseteq G$ be a maximal split torus, and let $H = C_G(S)$.  Then $H$ is reductive, and there are parabolic subgroups of $G$ with unipotent radicals $\mc U, \mc U'$, respectively, such that the map
\[H \times_A \mc U \times_A \mc U' \to G\]
induced by multiplication is an open immersion that gives an isomorphism of $H \times_A \mc U \times_A \mc U'$
with the dense open subset $\mc C = H\mc U\mc U'$ of $G$.  
\end{prop}

\begin{proof}
Choose a minimal parabolic subgroup $P$ of $G$ containing $S$, 
and let $\mc U$ be the unipotent radical of $P$.
The group $H$ is
a reductive Levi subgroup of $P$; the group $P$ is a semi-direct product of $H$ with the normal subgroup $\mc U$; 
and $H\mc U \cong H \times_A \mc U$ as schemes via the multiplication map (see \cite[Exp.~XXVI,  Proposition~1.6, Proposition~6.8, Remark~6.18]{SGA3.3}).
By \cite[Exp.~XXVI, Theorem~4.3.2]{SGA3.3}, there is a unique parabolic subgroup $P'$ of $G$ such that $P \cap P' = H$, and its unipotent radical $\mc U'$ has the property that the multiplication map $P \times_A \mc U' \to G$ is an open immersion.  (Here we apply the inversion map on $G$ to reverse the order of the factors from the one given in part (vi) of that theorem.)  We conclude that the composition 
$H \times_A \mc U \times_A \mc U' \to P \times_A \mc U' \to G$ is an open immersion with image $\mc C = H\mc U\mc U'$.  That is, $H \times_A \mc U \times_A \mc U' \to P \times_A \mc U' \to \mc C$ is
an isomorphism and $\mc C \subseteq G$ is a dense open subset.
\end{proof}
  
\begin{cor} \label{R equiv in H} 
With notation as in Proposition~\ref{decomposition}, let $L$ be the fraction field of $A$.
\renewcommand{\theenumi}{\alph{enumi}}
\renewcommand{\labelenumi}{(\alph{enumi})}
\begin{enumerate}
\item \label{R equiv bij}
The isomorphism $H \times \mc U \times \mc U' \iso \mc C \subseteq G$
and the map $H 
\times \mc U \times \mc U' \to H \to H/S$ defined by projection induce bijections $G(L)/\R \leftarrow \mc C(L)/\R \to H(L)/\R \to (H/S)(L)/\R$ that together define a group isomorphism
$G(L)/\R \iso (H/S)(L)/\R$.
\item \label{dvr R rep in H}
If, in addition, $A$ is a Henselian discrete valuation ring, then every element of $G(L)$ is $\R$-equivalent to an element of $H(A)$.
\end{enumerate}
\end{cor}

\begin{proof}
If $L$ is a finite field (and so is equal to $A$), then the set of $\R$-equivalence classes of $L$-points in any linear algebraic group over $L$ consists of just one element, by \cite[Corollaire~6, p.~202]{CTS77}.  In particular, this applies to the groups $G,H,S,\mc U,\mc U'$; and also to $\mc C(L)/\R$ since $\mc C = H\mc U\mc U'$.  So the corollary is trivial in that case.

Thus we may assume that $L$ is infinite.
Since $G_L$ is reductive over $L$, and since $\mc C \subseteq G$ is a dense open subset, \cite[Proposition~11, p.~196]{CTS77} then says that the natural map $\mc C(L)/\R \to G(L)/\R$ is a bijection.  For part~(\ref{R equiv bij}) we next  show that we have bijections $\mc C(L)/\R \to H(L)/\R \to (H/S)(L)/\R$.

By~\cite[Exp.~XXVI, \S 2, Corollaire~2.5]{SGA3.3}, the unipotent radicals $\mc U$ and $\mc U'$ are isomorphic to affine spaces $\A^n_A$ for some values of $n$.  Thus the morphism $\mc C \iso H \times \mc U \times \mc U' \to H$ induces a bijection $\mc C(L)/\R \to H(L)/\R$ on $\R$-equivalence classes.  
Since $S$ is a split torus, it is isomorphic to some power of $\G_{{\rm m},A}$; so each fiber of the $S$-torsor $H \to H/S$ over an $L$-point of $H/S$ contains $L$-points, all of which are $\R$-equivalent.  Thus $H(L)/\R \to (H/S)(L)/\R$ is surjective.  
For injectivity, if two $L$-points $P,Q$ of $H/S$ are $\R$-equivalent, then they are
connected by a smooth rational affine $L$-curve $C \subseteq H/S$; and the restriction of the $S$-torsor $H \to H/S$ to $C$ is trivial, since $C$ has trivial Picard group.  
Thus $C$ lifts to a smooth rational affine curve $\til C \subseteq H$, and $\til C$ meets the fibers of $H \to H/S$ over $P$ and $Q$ at $L$-points of $H$.  So all the $L$-points of $H$ lying over $P$ are $\R$-equivalent to those over $Q$, proving the bijectivity of $H(L)/\R \to (H/S)(L)/\R$.

Next, observe that the above bijection $H(L)/\R \to G(L)/\R$ carries the $\R$-equivalence class of $h \in H(L)$ to the $\R$-equivalence class of $h$ in $G(L)$; and so it agrees with the homomorphism induced by $H\to G$ (which is then an isomorphism).  Since the bijection $H(L)/\R \to (H/S)(L)/\R$ is the map induced by $H\to H/S$, it is also an isomorphism, and hence so is $G(L)/\R \iso (H/S)(L)/\R$, completing part~(\ref{R equiv bij}).

Since $H(L)/\R \to G(L)/\R$ is an isomorphism, every element of $G(L)$ is $\R$-equivalent to an element of $H(L)$.  In the case that $A$ is a Henselian discrete valuation ring, Proposition~\ref{H and S}(\ref{almost anisotropic factorization}) says that $H(L) = H(A)S(L)$.  Since $S_L$ is a power of $\G_{{\rm m},L}$, any two $L$-points of $S$ are $\R$-equivalent, and 
it follows that every element of $H(L)$ is $\R$-equivalent to an element of $H(A)$.  This proves part~(\ref{dvr R rep in H}).
\end{proof}

Below we introduce our specialization map for a reductive group $G$ over a discrete valuation ring $A$ with fraction field $L$ and residue field $\ell$.  Given another discrete valuation ring $A'$, we will consider 
injective local homomorphisms between them; i.e.,
injective homomorphisms $\phi:A \to A'$ of discrete valuation rings such that $\phi({\mathfrak m_A}) \subseteq {\mathfrak m_{A'}}$.  Such maps $\phi$ induce inclusions $L \to L'$ and $\ell \to \ell'$, where $L',\ell'$ are the fraction field and residue field of $A'$, respectively.

Here and below, we write $[\mbox{ \ }]_\R$ to denote the $\R$-equivalence class of an element.  
For a reductive group $G$ over a domain $A$ with fraction field $L$, by $G(A)/\R$ we will mean the subset of $G(L)/\R$ consisting of $\R$-equivalence classes that contain an element of $G(A)$.

\begin{thm} \label{dvr specialization}
Let $A$ be a discrete valuation ring with fraction field $L$ and residue field $\ell$, and let $G$ be a reductive group over $A$.  
\renewcommand{\theenumi}{\alph{enumi}}
\renewcommand{\labelenumi}{(\alph{enumi})}
\begin{enumerate}
\item \label{dvr spcl unique}
There is a unique homomorphism 
\[\spcl = \spcl_A = \spcl_{A,G}:G(L)/\R \to G(\ell)/\R\] 
such that $\spcl([g]_\R) = [\ov g]_\R$ for all $g \in G(A) \subseteq  G(L)$, 
and which is functorial in $G$ and in $A$ (in the latter case with respect to injective local homomorphisms of discrete valuation rings).  
\item \label{dvr spcl surj hens}
If $A$ is Henselian, then $\spcl:G(L)/\R \to G(\ell)/\R$ is surjective.
\end{enumerate}
\end{thm}

\begin{proof}
Since $G(\ell)/\R$ is trivial if $\ell$ is finite (by \cite[Corollaire~6, p.~202]{CTS77}), we may assume that $\ell$ is infinite.

For part~(\ref{dvr spcl unique}), we first consider the case of complete discrete valuation rings $A$.  Given such an $A$, choose a maximal split torus $S \subset G$, and let $H = C_G(S)$.  
We have a natural bijection $G(L)/\R \iso (H/S)(L)/\R$, by Corollary~\ref{R equiv in H}(\ref{R equiv bij}).  Here
$H/S$ is anisotropic over $L$ by Proposition~\ref{H and S}(\ref{aniso quotient}), and
over $\ell$ by Proposition~\ref{local anisotropy}(\ref{Hens equiv}) since $A$ is complete.  Thus $(H/S)(L)=(H/S)(A)$ by Proposition~\ref{BTR}, and the reduction map $(H/S)(A) \to (H/S)(\ell)$ induces a homomorphism $
(H/S)(A)/\R \to (H/S)(\ell)/\R$ by Proposition~\ref{R special}.  
Moreover $S_\ell$ is a maximal split torus in $G_\ell$ by Proposition~\ref{split tori lift}, and $H_\ell = C_{G_\ell}(S_\ell)$ (as noted after Proposition~\ref{H and S}); and so we have an isomorphism $(H/S)(\ell)/\R \iso G(\ell)/\R$ by 
Corollary~\ref{R equiv in H}(\ref{R equiv bij}).
We define the map $\spcl_A$ as the composition
\[G(L)/\R \iso (H/S)(L)/\R = (H/S)(A)/\R \to (H/S)(\ell)/\R \iso G(\ell)/\R.\]
This composition is a homomorphism since each of the factors is. 

To show that $\spcl([g]_\R) = [\ov g]_\R$ for all $g \in G(A)$ (still under the assumption that $A$ is complete), we first suppose that $g \in \mc C(A)$, with $\mc C$ as in Proposition~\ref{decomposition}.  Then, with $\mc U, \mc U'$ as in that proposition, $\spcl([g]_\R)$ is the image of $g$ under
the composition $\mc C(A) = \mc U'(A) \times \mc U(A) \times H(A) \to H(A) \to (H/S)(A) \to (H/S)(A)/\R \to (H/S)(\ell)/\R \to G(\ell)/\R$.  It is thus equal to $[\ov g]_\R$, for $g \in \mc C(A)$.  

To prove that $\spcl([g]_\R) = [\ov g]_\R$ for an arbitrary element $g \in G(A)$ (i.e., not necessarily in $\mc C(A)$), note that the dense open subset $\mc C \subseteq G$ meets the closed fiber $G_\ell$ non-trivially (e.g., at the identity), as does its translate $g\mc C$.  So the intersection 
$\mc C \cap g\mc C$ is open and dense, and also meets $G_\ell$ non-trivially.
But $G_\ell$ is unirational, by \cite[Chap.~V, Theorem~18.2]{Borel}.  Since $\ell$ is infinite, it follows that $\mc C \cap g\mc C \cap G_\ell$ contains a $\ell$-point.
That $\ell$-point lifts to a point $c \in G(A)$ by the smoothness of $G$ and the completeness of $A$; and $c$ lies in $\mc C(A) \cap g\mc C(A)$ since its reduction lies in the open set $\mc C \cap g\mc C$.  Thus $c = gc'$ for some $c' \in \mc C(A)$.  So $g = c(c')^{-1} \in G(A)$.
Since $c,c' \in \mc C(A)$ and since $\spcl$ is a homomorphism, 
the above special case yields $\spcl([g]_\R) = \spcl([c]_\R)\spcl([(c')^{-1}]_\R) = [\ov c]_\R[(\ov c')^{-1}]_\R = [\ov g]_\R$, as needed.

The functoriality condition on $\spcl$ with respect to $G$ and $A$ asserts the commutativity of the diagrams
\begin{equation*}
\xymatrix{
G(L)/\R \ar[d] \ar[r]^{\spcl_A} & G(\ell)/\R \ar[d]& &  
G(L)/\R \ar[d] \ar[r]^{\spcl_A} & G(\ell)/\R \ar[d]  
\\
G'(L)/\R  \ar[r]^{\spcl_A} & G'(\ell)/\R & & 
G(L')/\R  \ar[r]^{\spcl_{A'}} & G(\ell')/\R, 
\\
}
\end{equation*}
where $G \to G'$ is a morphism of reductive groups over $A$, and where $A \to A'$ is an injective local homomorphism of discrete valuation rings with $L'$ the fraction field of $A'$ and $\ell \to \ell'$ the induced map on residue fields (which exists by the local homomorphism hypothesis).  
In the current situation in which the discrete valuation rings that we consider are assumed complete, the commutativity of these diagrams follows from the facts that the vertical maps are well defined (on $\R$-equivalence classes); that every $\R$-equivalence class in $G(L)/\R$ contains an element of $H(A) \subseteq G(A)$ (by Corollary~\ref{R equiv in H}(\ref{dvr R rep in H})); and that $\spcl$ is given by reduction modulo the maximal ideal for elements of $G(A)$.

For uniqueness, suppose that $\spcl'_A$ is another such functorial map defined for complete discrete valuation rings $A$.  As above, by 
Corollary~\ref{R equiv in H}(\ref{dvr R rep in H}), each $g \in G(L)$ is $\R$-equivalent to some $h \in H(A) \subseteq G(A) \subseteq G(L)$.
Thus $\spcl'([g]_\R) = \spcl'([h]_\R) = [\ov h]_\R =  \spcl([h]_\R) = \spcl([g]_\R)$.  This shows that $\spcl'=\spcl$, proving uniqueness and completing the proof of~(\ref{dvr spcl unique}) in the case of complete rings $A$.

For a more general discrete valuation ring $A$, say with completion $\wh A$, define $\spcl_A = \spcl_{\wh A} \circ \iota_A$, where $\iota_A:A \to \wh A$ is the natural injection.  This is still a well defined homomorphism.  It follows immediately from this definition and from the complete case that $\spcl([g]_\R) = [\ov g]_\R$ for all $g \in G(A)$.  Functoriality also follows from the complete case, because a local homomorphism of discrete valuation rings induces such a homomorphism between their completions, and because $\spcl_A = \spcl_{\wh A} \circ \iota_A$.  Uniqueness similarly follows from the complete case, since if $\spcl'_A$ is another such functorial map, then $\spcl'_A$ factors through $\spcl'_{\wh A}$ by functoriality and thus $\spcl_A'=\spcl_A$.  This completes the proof of part~(\ref{dvr spcl unique}) in the general case. 

Finally, if $A$ is Henselian, then $G(A) \to G(\ell)$ is surjective by formal smoothness of $G$.  
Thus $G(A)/\R \to G(\ell)/\R$ is surjective, and hence so is $\spcl:G(L)/\R \to G(\ell)/\R$, proving part~(\ref{dvr spcl surj hens}).
\end{proof}

We call the map $\spcl$ in Theorem~\ref{dvr specialization} the {\it specialization map} for $G$ over the discrete valuation ring $A$.

\begin{rem} \label{dvr spcl rk}
\renewcommand{\theenumi}{\alph{enumi}}
\renewcommand{\labelenumi}{(\alph{enumi})}
\begin{enumerate}
\item \label{extend Gille}
As a consequence of the uniqueness assertion in Theorem~\ref{dvr specialization}, the above map $\spcl$ agrees with the map in \cite[Th\'eor\`eme~0.2]{GilleTAMS}, under the hypothesis that $\cha(\ell) \ne 2$.
\item \label{spcl via completion}
Given a discrete valuation ring $A$ with fraction field $L$ and residue field $\ell$, let 
$\wh A$ be the completion of $A$, say with fraction field $\wh L$ (and again residue field $\ell$).  
By the construction in the above proof (or by functoriality), the
map $\spcl=\spcl_A: G(L)/\R \to G(\ell)/\R$ then factors through $\spcl_{\wh A}: G(\wh L)/\R \to G(\ell)/\R$.  As a result, the study of the specialization map on discrete valuation rings can be reduced to the complete case.
\end{enumerate}
\end{rem}

\subsection{Specialization in dimension two} \label{spcl dim 2}

We now turn to the case of a two-dimensional regular local ring $A$, with fraction field $L$ and residue field $\ell$. To define the specialization map in this situation, we can use Theorem~\ref{dvr specialization} twice in succession to obtain a specialization map for reductive groups over $A$.  
Namely, we pick a regular height one prime $\wp$ in $A$ (see the discussion just before Proposition~\ref{R special}).
The localization $A_\wp$
is a discrete valuation ring with fraction field $L$, and its residue field $k_\wp := {\rm frac}(A/\wp)$ is itself the fraction field of a discrete valuation ring that has residue field $\ell$.  
So for any reductive group $G$ over $A$ we can take the composition
\begin{equation} \label{special 2d composition}
\xymatrix @C=.4cm{
\spcl_\wp = \spcl_{A,\wp}= \spcl_{A,\wp,G}:G(L)/\R\, {\buildrel \spcl_{A_\wp} \over \longrightarrow} \, G(k_\wp)/\R \,  {\buildrel \spcl_{A/\wp} \over \longrightarrow} \, G(\ell)/\R\\
}
\end{equation}
of specialization maps associated to discrete valuation rings.

\begin{prop} \label{2d spcl wp}
Let $A$ be a two-dimensional regular local ring with fraction field $L$ and residue field $\ell$, let $\wp$ be a regular height one prime of $A$, and let $G$ be a connected reductive group over $A$.  Then the following hold:
\renewcommand{\theenumi}{\alph{enumi}}
\renewcommand{\labelenumi}{(\alph{enumi})}
\begin{enumerate}
\item \label{2d wp sp hom}
The map $\spcl_{A,\wp}:G(L)/\R \to G(\ell)/\R$ is a group homomorphism.
\item \label{2d wp sp eval}
For all $g \in G(A)$, we have $\spcl_{A,\wp}([g]_\R)=[\ov g]_\R$.  
\item \label{2d wp sp complete surj}
The map $\spcl_{A,\wp}$ is surjective if $A$ is complete.
\item \label{2d wp sp factor comp}
The map $\spcl_{A,\wp}$ factors through $\spcl_{\wh A,\wh\wp}$, where $\wh A$ is the completion of $A$ and $\wh\wp$ is the unique height one prime of $\wh A$ lying over $\wp$.
\end{enumerate}
\end{prop}

\begin{proof}
This is immediate from the definition of $\spcl_{A,\wp}$ in display~(\ref{special 2d composition}) together with Theorem~\ref{dvr specialization}, where in part~(\ref{2d wp sp factor comp}) we use the functoriality of $\spcl$ for discrete valuation rings.
\end{proof}

Above, if we write $L_\wp$ for the fraction field of the completion $\wh A_\wp$ of the discrete valuation ring $A_\wp$, then the first map $\spcl_{A_\wp}:G(L)/\R \to G(k_\wp)/\R$ in~(\ref{special 2d composition})
is the same as the composition 
\begin{equation} \label{special 2d complete}
\xymatrix @C=.4cm{
G(L)/\R \to G(L_\wp)/\R\, {\buildrel \spcl_{\wh A_\wp} \over \longrightarrow}\, G(k_\wp)/\R,\\
}
\end{equation}
by Remark~\ref{dvr spcl rk}(\ref{spcl via completion}).

In the case that $A$ is complete and has fraction field $L$ and residue field $\ell$, we can also describe the map $\spcl_\wp$ 
in terms of a maximal split torus $S \subseteq G$ and its centralizer $H= C_G(S)$.  Namely, by Corollary~\ref{R equiv in H}(\ref{R equiv bij}), $G(L)/\R \iso (H/S)(L)/\R$.  The group $(H/S)_L$ is anisotropic by Proposition~\ref{H and S}(\ref{aniso quotient}); and then so are $(H/S)_\ell$ and hence $(H/S)_{A/\wp}$ and $(H/S)_{\wh A_\wp}$ by successive applications of Proposition~\ref{local anisotropy}.   Together with Proposition~\ref{BTR}, we can then reformulate $\spcl_{A,\wp,G}$ as the following composition:
\begin{align} \label{2 dim spcl compos}
\begin{split}
G(L)/\R \iso (H/S)(L)/\R &\to (H/S)(L_\wp)/\R = (H/S)(\wh A_{\wp})/\R \\
&\to (H/S)(k_{\wp})/\R = (H/S)( A/\wp)/\R\\
&\to (H/S)(\ell)/\R \iso G(\ell)/\R,
\end{split}
\end{align}
or equivalently as 
\begin{equation} \label{2 dim spcl factor HS}
\xymatrix @C=.4cm{
G(L)/\R \iso (H/S)(L)/\R \ {\buildrel \spcl_{A,\wp,H/S} \over {\xrightarrow{\hspace{1.3cm}}}} \ (H/S)(\ell)/\R \iso G(\ell)/\R.\\
}
\end{equation}

In order to prove the analogue of Theorem~\ref{dvr specialization}, we will show that this map is independent of the choice of regular height one prime $\wp$, with the help of the next proposition.

\begin{prop} \label{2d special new}
Let $A$ be a two-dimensional regular local ring with fraction field $L$ and residue field $\ell$.
Let $G \subseteq \GL_{n,A}$ be a reductive group over $A$ such that $G_\ell$ is anisotropic.
Let $g \in G(L)$, viewed as a rational map $X :=\Spec(A) \dashrightarrow G$.
Then there is a morphism $\til X \to X$, obtained by a sequence of blow-ups at closed points, such that 
\renewcommand{\theenumi}{\roman{enumi}}
\renewcommand{\labelenumi}{(\roman{enumi})}
\begin{enumerate}
\item \label{ratl map lift}
the rational map $g:X \dashrightarrow G$ lifts to a unique morphism $\til g:\til X \to \til G$, where $\til G$ is the closure of $G$ in $\P^{n^2}_A$; and
\item \label{2d spcl via lift}
for each regular height one prime $\wp$ of $A$, there is a point $P \in \til X(\ell)$ with $\til g(P) \in G(\ell)$, such that $\spcl_\wp([g]_\R) \in G(\ell)/\R$ is the $\R$-equivalence class of $\til g(P)$.
\end{enumerate}
\end{prop}

\begin{proof}
We claim that it suffices to prove this in the case of complete rings.  Namely, given $A,L,\ell$ as above, let $\wh A$ be the completion of $A$ and let $\wh L$ be the fraction field of $\wh A$.  Then a regular height one prime $\wp$ of $A$ extends to a regular height one prime $\wh\wp$ of $\wh A$, and the diagram
\begin{equation}
\xymatrix{
G(L)/\R \ar@{^{(}->}[d] \ar[r]^{\spcl_{A,\wp}} & G(\ell)/\R \ar@{=}[d] \\
G(\wh L)/\R  \ar[r]^{\spcl_{\wh A,\wh \wp}} & G(\ell)/\R \\
}
\end{equation}
commutes by Proposition~\ref{2d spcl wp}(\ref{2d wp sp factor comp}).  
Moreover every succession of blow-ups of $\Spec(\wh A)$ at closed points is induced by a succession of blow-ups $\til X$ of $\Spec(A)$ at closed points, since pulling back from $\Spec(A)$ to $\Spec(\wh A)$ preserves the closed fiber.  Now the locus of indeterminacy of a rational map $\til X \dashrightarrow \til G$ consists just of closed points $P$, and a rational map is defined at $P$ if and only if it is defined on the complete local ring at $P$.  Since the pullback from $\Spec(A)$ to $\Spec(\wh A)$ preserves complete local rings at closed points, it preserves whether a given rational map is a morphism.
Thus we may replace the given ring $A$ by $\wh A$; and so we will assume that $A$ is complete.

Applying Lemma~\ref{context map} to the rational map $g:X \dashrightarrow G \subseteq \til G$ (thus taking $Y = \til G$ in that lemma), we obtain a succession of blow-ups $\til X \to X$ at closed points satisfying assertion~(\ref{ratl map lift}) of the proposition (where the uniqueness of the lift $\til g:\til X \to \til G$ follows from the fact that regular morphisms from an integral scheme to a separated scheme are equal if they agree as rational maps).  It remains to show that assertion~(\ref{2d spcl via lift}) holds.

Let $\wp$ be a regular height one prime of $A$, viewed as a codimension one point of $X$.  Since $\til X \to X$ is an isomorphism away from the closed point of $X$, there is a unique codimension one point $\til\wp$ of $\til X$ that lies over $\wp$, such that $\wp$ and $\til\wp$ have the same local rings and residue fields.  Thus there is a section of $\til X \to X$ over the point $\wp$, taking $\wp$ to $\til\wp$.  Since $\wp$ is regular, $A/\wp$ is a discrete valuation ring; and by the valuative criterion for properness, the section over $\wp$ extends to a section $\sigma$ over $\Spec(A/\wp)$. Let $P \in \til X(\ell)$ be the image of the closed point of $X$ under $\sigma$; 
write $\til A = \mc O_{\til X,P}$;
and let $Z \subset \til X$ be the (isomorphic) image of $\Spec(A/\wp)$ in $\til X$ under $\sigma$.  (Here we identify the codimension one point $\til \wp \in \til X$ with the corresponding height one prime ideal of $\til A$.)
Thus $L = {\rm frac}(\til A)$.  Also, 
$Z = \Spec(B)$, where $B:=\til A/\til\wp$ is a discrete valuation ring isomorphic to $A/\wp$, whose fraction field $k_{\til \wp}$ is isomorphic to $k_\wp = A_\wp/\wp = {\rm frac}(A/\wp)$.  

Since $G_\ell$ is anisotropic, and since $\ell$ is the residue field of the complete discrete valuation ring $A/\wp$, whose fraction field is $k_\wp$, it follows from Proposition~\ref{local anisotropy} that $G_{k_\wp}$ is anisotropic.  But $k_\wp$ is also the residue field of the discrete valuation ring $A_\wp$, whose fraction field is $L$.
So 
it follows from Proposition~\ref{BTR} that the element $g \in G(L)$ lies in $G(A_\wp)$, 
and thus restricts to an element $\bar g \in G(k_\wp)$.  Since $\til g$ lifts $g$, it follows that $\til g:\til X \to \til G$ carries the point $\til\wp$ to a point of $G$.  Now $\til\wp$ is the generic point of $Z = \Spec(B)$, and $B$ is a complete discrete valuation ring having fraction field $k_{\til\wp}$ and residue field $\ell = \kappa(P)$.  So $G(k_{\til\wp}) = G(B)$ by Proposition~\ref{BTR}; and thus 
the restriction of $\til g$ to $\til\wp$ extends to a morphism $h:Z \to G \subset \til G$.  But the restriction $\til g|_Z$ of $\til g$ to $Z$ is a morphism $Z \to \til G$ that agrees with $h$ at the point $\til\wp$.  Since $\til G$ is separated, it follows that $h=\til g|_Z$; and hence $\til g(P) = h(P) \in G(\ell)$.

By displays~(\ref{special 2d composition}) and~(\ref{special 2d complete}), the specialization map $\spcl_{A,\wp}:G(L)/\R \to G(\ell)/\R$ factors through the natural map $G(L)/\R \to G(L_\wp)/\R$, where $L_\wp$ is the fraction field of $\wh A_\wp$.  Similarly,  $\spcl_{\til A,\til \wp}$ factors through $G( L)/\R \to G(L_{\til \wp})/\R$, where $L_{\til \wp}$ is the fraction field of the completion of $\til A_{\til \wp}$.  But the natural map $G(L_\wp) \to G(L_{\til \wp})$ is an isomorphism, since $A_\wp \to \til A_{\til \wp}$ is.  
With respect to this isomorphism, the one-dimensional specialization maps $\spcl_{A_\wp}$ and $\spcl_{A/\wp}$ in display~(\ref{special 2d composition}) are identified with the corresponding maps $\spcl_{\til A_{\til \wp}}$ and $\spcl_{\til A/\til \wp}$.
Hence $\spcl_{A,\wp}([g]_\R)=\spcl_{\til A,\til \wp}([g]_\R)=\spcl_{\til A,\til\wp}([\til g]_\R)= [\til g(P)]_\R$ by Proposition~\ref{2d spcl wp}, since $\til g \in G({\til A})$.
\end{proof}

Using the above, we now show that the map $\spcl_\wp$ is independent of $\wp$, yielding a {\em two-dimensional specialization map} $\spcl$ that depends just on $A$ and $G$.  We also show other key properties of the map $\spcl$.

\begin{thm} \label{main appendix}
Let $A$ be a regular local ring of dimension two, with fraction field $L$ and residue field $\ell$. Let $G$ be a reductive group over $A$. Then
\renewcommand{\theenumi}{\alph{enumi}}
\renewcommand{\labelenumi}{(\alph{enumi})}
\begin{enumerate}
\item \label{sp hom}
There is a unique group homomorphism $\spcl = \spcl_A = \spcl_{A,G}: G(L)/\R \to G(\ell)/\R$ such that for every regular height one prime $\wp$ of $A$, $\spcl=\spcl_\wp$.
\item \label{complete surjectivity} 
If $A$ is complete, then $\spcl$ is surjective.
\item  \label{special restriction} 
For $g \in G(A)$, $\spcl_A([g]_\R)=[\ov g]_\R$, where $\ov g$ is the image of $g$ under the homomorphism $G(A) \to G(\ell)$ induced by the reduction map $A \to \ell$.
\item \label{2d spcl functorial}
The map $\spcl$ is functorial in $G$ and in $A$ (in the latter case with respect to injective local homomorphisms of two-dimensional regular local rings).
\item \label{functorial uniqueness}
The group homomorphisms $\spcl_A:G(L)/\R \to G(\ell)/\R$, as $A$ varies over regular local rings of dimension $2$, are uniquely characterized by properties~(\ref{special restriction}) and~(\ref{2d spcl functorial}).
\end{enumerate}
\end{thm}

\begin{proof}
We begin by showing that the map $\spcl_\wp$ is independent of the choice of the choice of the regular height one prime $\wp$ in $A$.  By Proposition~\ref{2d spcl wp}(\ref{2d wp sp factor comp}), we may assume that $A$ is complete.  In that situation, first
consider the case where $G_\ell$ is anisotropic.
With notation as in Proposition~\ref{2d special new}, for each $g \in G(L)$ we have a morphism $\til X \to X$, a lifting $\til g:\til X \to \til G$ of $g: X \dashrightarrow G$, and a point  $P \in \til X(\ell)$ such that $\til g(P) \in G(\ell)$ and $\spcl_\wp([g]_\R) =[\til g(P)]_\R$.  Since every point in $\til X(\ell)$ lies over the closed point of $X$, it follows from Lemma~\ref{R trivial fibers} that any two such points are $\R$-equivalent, and hence so are their images under $\til g$.  Thus $\spcl_\wp$ is independent of the choice of $\wp$, if $G_\ell$ is anisotropic.  For a more general reductive group $G$ over the complete ring $A$, display~(\ref{2 dim spcl factor HS}) says that the specialization map $\spcl_\wp$ for $G$ factors through the corresponding map for the group $H/S$, which is anisotropic over $\ell$.  So the assertion follows from the previous case.

Using this, parts~(\ref{sp hom}), (\ref{complete surjectivity}), and~(\ref{special restriction}) follow from 
Proposition~\ref{2d spcl wp}, by taking $\spcl=\spcl_\wp$ for any choice of $\wp$.  Part~(\ref{2d spcl functorial}) follows from display~(\ref{special 2d composition}) together with the functoriality in Theorem~\ref{dvr specialization}(\ref{dvr spcl unique}).  

Finally, for part~(\ref{functorial uniqueness}), suppose that for each regular local ring $A$ of dimension $2$ and each reductive group $G$ over $A$, we have a homomorphism $s_A: G(L)/\R \to G(\ell)/\R$, with $L,\ell$ denoting the fraction field and residue field of $A$, such that $s_A([g]_\R)=[\ov g]_\R$ for all $g \in G(A)$, and such that the maps $s_A$ are functorial in $A$ and $G$.  We wish to show that $s_A = \spcl_A$ for all $A$ and $G$.  

First consider the case that $G_\ell$ is anisotropic, where  
$A,L,\ell,G$ are as above.  Let $g \in G(L)$.  As in Proposition~\ref{2d special new}, we obtain $\til X$, $\til G$, $\til g$, $P$ such that $\spcl_\wp([g]_\R) = [\til g(P)] \in G(\ell)/\R$.  Here $\til g: \til X \to \til G$ is a morphism, and $\til g(P) \in G(\ell)$; so $\til g \in G(A')$, where $A'$ is the complete local ring of $\til X$ at $P$, say with fraction field $L'$.  The inclusion $A \to A'$ is a local homomorphism of complete regular local rings of dimension two, inducing an isomorphism on residue fields.  So, by functoriality, $\spcl_A$ is the composition of the natural inclusion $G(L) \to G(L')$ with the map $\spcl_{A'}$, and $\spcl_A(g) = \spcl_{A'}(g) = \spcl_{A'}(\til g)$.  Since $s_A$ is also functorial, we similarly have $s_A(g)=s_{A'}(\til g)$.  But since $\til g \in 
G(A')$, 
it follows from property~(\ref{special restriction}) that $\spcl_{A'}(\til g) = [\til g(P)]_\R = s_{A'}(\til g)$.  Thus $\spcl_A(g)=s_A(g)$, as desired.

For a more general reductive group $G$, the same conclusion follows from the above special case by the functoriality condition and by the functorial isomorphisms $G(L)/\R \to (H/S)(L)/\R$ as in Corollary~\ref{R equiv in H}, since $(H/S)_\ell$ is anisotropic. 
\end{proof}


\providecommand{\bysame}{\leavevmode\hbox to3em{\hrulefill}\thinspace}

\noindent{\bf Author Information:}\\

\noindent Jean-Louis Colliot-Th\'el\`ene\\
Arithm\'etique et G\'eom\'etrie Alg\'ebrique, Universit\'e Paris-Saclay, CNRS, Laboratoire de\\ Math\'ematiques d'Orsay, 91405, Orsay, France\\
email: jlct@math.u-psud.fr

\medskip

\noindent David Harbater\\
Department of Mathematics, University of Pennsylvania, Philadelphia, PA 19104-6395, USA\\
email: harbater@math.upenn.edu

\medskip

\noindent Julia Hartmann\\
Department of Mathematics, University of Pennsylvania, Philadelphia, PA 19104-6395, USA\\
email: hartmann@math.upenn.edu

\medskip

\noindent Daniel Krashen\\
Department of Mathematics, University of Pennsylvania, Philadelphia, PA 19104-6395, USA\\
email: dkrashen@math.upenn.edu

\medskip

\noindent R.~Parimala\\
Department of Mathematics and Computer Science, Emory University, Atlanta, GA 30322, USA\\
email: parimala@mathcs.emory.edu

\medskip

\noindent V.~Suresh\\
Department of Mathematics and Computer Science, Emory University, Atlanta, GA 30322, USA\\
email: suresh@mathcs.emory.edu

\medskip

The authors were supported on an NSF collaborative FRG grant: DMS-1463733 (DH and JH), DMS-1463901 (DK), DMS-1463882 (RP and VS).  Additional support was provided by NSF RTG grant DMS-1344994 and NSF grant DMS-1902144 (DK); NSF DMS-1805439 and DMS-2102987 (DH and JH); NSF DMS-1401319 (RP); NSF DMS-1301785 (VS); and NSF DMS-1801951 (RP and VS).

\end{document}